\documentclass[a4paper,10pt]{article}

\usepackage[francais,english]{babel}
\usepackage{enumerate}
\usepackage[latin1]{inputenc}   
\usepackage{amsmath,amsthm}
\usepackage{enumerate}
\usepackage{mathdots}
\usepackage{mathtools}

\usepackage[dvips]{graphicx}
\usepackage{amsfonts,amssymb}
\usepackage{geometry}
\usepackage{hyperref}
\usepackage{stmaryrd}
\usepackage{fancyhdr}
\usepackage{tikz}
\usepackage{url}
\usepackage{dsfont}
\newtheorem*{conj*}{Conjecture}

\newtheorem*{thm*}{Theorem}

\newtheorem{prop}{Proposition}[section]

\newtheorem{LM}{Lemma}[section]

\newtheorem{thm}{Theorem}[section]
\newtheorem{df}{Definition}[section]
\newtheorem{cor}{Corollary}[section]

\newtheoremstyle{pourlesremarques}{\topsep}{\topsep}{\normalfont}{}{\bfseries}{.}{ }{}
\theoremstyle{pourlesremarques}
\newtheorem{rem}{Remark}[section]

\newtheorem*{rem*}{Remark}
\newtheoremstyle{pourlesexemples}{\topsep}{\topsep}{\normalfont}{}{\bfseries}{.}{ }{}
\theoremstyle{pourlesexemples}

\renewcommand{\a}{\alpha}

\newcommand{\ind}{\mathrm{ind}}
\newcommand{\SL}{\mathrm{SL}}
\newcommand{\Ind}{\mathrm{Ind}}
\newcommand{\Inv}{\mathrm{Inv}}
\newcommand{\Lie}{\mathrm{Lie}}
\newcommand{\La}{\mathrm{L}}
\newcommand{\diag}{\mathrm{diag}}

\newcommand{\w}{\varpi}
\renewcommand{\d}{\delta}
\newcommand{\e}{\epsilon}
\newcommand{\R}{\mathbb{R}}
\renewcommand{\Re}{\mathrm{Re}}
\newcommand{\Rep}{R(P\backslash G/ H)}
\newcommand{\Hom}{\mathrm{Hom}}
\newcommand{\A}{\mathbb{A}}
\newcommand{\GL}{\mathrm{GL}}
\renewcommand{\l}{\lambda}
\newcommand{\C}{\mathbb{C}}
\newcommand{\s}{\underline{s}}
\renewcommand{\u}{\underline{u}}
\renewcommand{\k}{\kappa}
\renewcommand{\S}{\mathfrak{S}}
\newcommand{\F}{\mathcal{F}}

\newcommand{\m}{\overline{m}}
\newcommand{\n}{\overline{n}}
\newcommand{\Q}{\mathbb{Q}}
\newcommand{\K}{\mathbb{K}}
\newcommand{\N}{\mathbb{N}}
\newcommand{\Nrd}{\mathrm{Nrd}}
\newcommand{\M}{\mathcal{M}}

\newcommand{\Z}{\mathbb{Z}}
\newcommand{\Ze}{\mathrm{Z}}
\renewcommand{\L}{\Lambda}
\newcommand{\JL}{\mathrm{JL}}
\newcommand{\St}{\mathrm{St}}

\newcommand{\1}{\mathbf{1}}

\newcommand{\sm}{\mathcal{C}^\infty}
\newcommand{\D}{\Delta}
\newcommand{\G}{\mathfrak{D}}

\title {\textbf{Gamma factors of intertwining periods and distinction for inner forms of $\GL(n)$}}
\author{Nadir MATRINGE\footnote{Nadir Matringe, Universit\'e de Poitiers, Laboratoire de Math\'ematiques et Applications,
T\'el\'eport 2 - BP 30179, Boulevard Marie et Pierre Curie, 86962, Futuroscope Chasseneuil Cedex. Email: Nadir.Matringe@math.univ-poitiers.fr}}

\begin{document}
\maketitle

\begin{abstract}
Let $F$ be a $p$-adic field, $E$ be a quadratic extension of $F$, $D$ be an $F$-central division algebra of odd index and let 
$\theta$ be the Galois involution attached to $E/F$. Set $H=\GL(m,D)$, $G=\GL(m,D\otimes_F E)$, and let $P=MU$ be a standard parabolic subgroup of $G$. Let $w$ be a Weyl 
involution stabilizing $M$ and $M^{\theta_w}$ be the subgroup of $M$ fixed by the involution 
$\theta_w:m\mapsto \theta(wmw)$. We denote by $X(M)^{w,-}$ the complex torus of $w$-anti-invariant unramified characters of $M$. Following the global methods of \cite{JLR}, we associate to a finite length representation $\sigma$ of $M$ and to a 
linear form $L\in \Hom_{M^{\theta_w}}(\sigma,\C)$ a family of $H$-invariant linear forms called intertwining periods on 
$\Ind_P^G(\chi \sigma)$ for $\chi \in X(M)^{w,-}$, which is meromorphic in the variable $\chi$. 
Then we give sufficient conditions for some of these intertwining periods, namely the open intertwining periods studied in \cite{BD08}, to have 
singularities. By a local/global method, we also compute in terms of Asai gamma factors the proportionality constants involved in their functional 
equations with respect to certain intertwining operators. 
As a consequence, we classify distinguished unitary and ladder representations of $G$, extending respectively the results of 
\cite{M14} and \cite{G15} for $D=F$, which both relied at some crucial step on the theory of Bernstein-Zelevinsky derivatives. We make use of one 
of the main results of \cite{BP18} which in the case of the group $G$ asserts
 that the Jacquet-Langlands correspondence preserves distinction. Such a result is for essentially square-integrable representations, but our
 method in fact allows us to use it only for cuspidal representations of $G$.
\end{abstract}

\newpage

\tableofcontents
\newpage

\section{Introduction}

Let $\mathbf{H}$ be a reductive group defined over some number field $k$, and let $\A$ be the ring of adeles of $k$. Let $l$ be a quadratic extension of $k$, and let $\mathbf{G}=\mathrm{Res}_{l/k}(\mathbf{H})$. Let $\theta$ be the non trivial element of 
$\mathrm{Gal}_k(l)$ and suppose that it stabilizes a minimal parabolic subgroup $\mathbf{P_0}$ of 
$\mathbf{G}$. Let $\mathbf{P}=\mathbf{M}\mathbf{U}$ be a standard parabolic subgroup of $\mathbf{G}$ with Levi subgroup $\mathbf{M}$, and suppose for simplicity that $\theta(\mathbf{M})=\mathbf{M}$. We denote by $\mathbf{Z_M}$ the connected center of $\mathbf{M}$, and identify the complexification of the character group of $\mathbf{Z_M}$ with $\C^t$ for some $t$. If $w$ is a Weyl involution stabilizing $\mathbf{M}(\A)$, we denote by $\C^t(w,-1)$ the space of $w$-anti-invariant vectors in $\C^t$. Let $\sigma$ be a cuspidal 
automorphic representation of $\mathbf{M}(\A)$ which we suppose, for simplification in this introduction, has a trivial central character. Then to $\phi$ in $\mathrm{Ind}_{\mathbf{P}(\A)}^{\mathbf{G}(\A)}(\sigma)$ and certain Weyl involutions $w$ stabilizing $\mathbf{M}$, Jacquet, Lapid and Rogawsky (\cite{JLR} and \cite{LR}) attached the intertwining period $J(w,\phi,\s)$, which is a meromorphic function of the variable $\s\in \C^t(w,-1)$. This intertwining period appears naturally in \cite{JLR} and then \cite{LR}, in the expression of the regularized $\mathbf{H}(\A)$-periods of the Eisenstein series $E(.,\phi,\s)$ on $\mathbf{G}(\A).$ In particular, writing $\sigma[\s]$ for the twist of $\sigma$ by the unramified character of 
$\mathbf{M}(\A)$ attached to $\s$, it defines for almost all $\s$ 
an $\mathbf{H}(\A)$-invariant linear form on the space $\mathrm{Ind}_{\mathbf{P}(\A)}^{\mathbf{G}(\A)}(\sigma[\s])$ and moreover satisfies functional equations with respect to certain standard intertwining operators. Intertwining periods are powerful tools to study distinction globally, for example Offen used them in \cite{O06} and \cite{O06-2} to determine the residual spectrum of $\GL(2n)$ distinguished by the symplectic group. We note that their local version (at least when $w$ corresponds to an open or closed $(P,H)$-orbit) play a prominent role in the local results of 
\cite{FLO} on unitary periods.\\
Here we are interested in local distinction for a very specific Galois pair. We consider $E/F$ a quadratic extension 
of $p$-adic fields, $D$ a division algebra of odd index over its center $F$, and set $H=\GL(m,D)$ and $G=\GL(m,D\otimes_F E)$. The group $H$ is the subgroup of $G$ fixed by the Galois involution $\theta$ of $E/F$.
We classify $H$-distinguished ladder and unitary representations of $G$, thus extending the results of \cite{G15}  
and \cite{M14} obtained for $D=F$. Both papers make use, at some crucial steps of the theory of Bernstein-Zelevinsky derivatives, which is not developed for non split $G$. In fact even if it was, it would probably not give as much information on distinction for the pair $(G,H)$ that it gives in the case of split $G$. To circumvent this difficulty, we use the local version 
of the intertwining periods defined above. Let's be more specific and give the main point of our method, which already 
appears in a very elementary form in \cite{M17}. Let $P=MN$ be a standard parabolic subgroup of $G$, attached to a partition of the form $\m=(m_1,\dots,m_t)$ of $m$, such that 
$m_{t+1-i}=m_i$ for all $i\in \{1,\dots,t\}$, and let $w$ be the element of $\S_t$ defined by $w(i)=t+1-i$, it acts on 
$M$ as follows: $w.\diag(g_1,\dots,g_t)=\diag(g_{w(1)},\dots,g_{w(t)})$. Setting 
$\theta_w=\theta \circ w$, if 
\[\sigma=\d_1\otimes \dots \otimes \d_t\] is an essentially square-integrable representation of $M$ such that $w(\sigma)^\vee=\sigma^\theta$, to $L\in \Hom_{M^{\theta_w}}(\sigma,\C)-\{0\}$, Blanc and Delorme attach in \cite{BD08} a non zero $H$-invariant linear form 
$J_\sigma(w,.,\s,L)$ on $\Ind_P^G(\sigma[\s])$ which depends meromorphically on the variable $\s\in \C^t(w,-1)$. For 
$\s$ in general position, the space $\Hom_H(\Ind_P^G(\sigma[\s]),\C)$ is of dimension $1$, hence for any $\tau\in \S_t$ 
commuting with $w$, if $A(\tau,\s)$ is the standard intertwining operator from $\Ind_P^G(\sigma[\s])$ to 
$\Ind_Q^G(\tau(\sigma)[\tau(\s)])$ (where $Q$ is the appropriate standard parabolic subgroup of $G$), there is a meromorphic function $\alpha(\s)$ such that 
\[J_{\tau(\sigma)}(w,A(\tau,\s).,\tau(\s),L)=\alpha(\s)J_\sigma(w,.,\s,L).\] 
Our first main result (Theorem \ref{alpha}) is to compute this function in terms of Asai gamma factors, when $t=2r$ and $\tau=\tau_r:=(r \ r+1)$. This is done using the functional equation of the global intertwining period of \cite{JLR} and \cite{LR}, and an unramified computation which follows the one done in \cite{JLR} (Section \ref{unramified computations}). Along the way we develop in Section \ref{Section p-adic intertwining periods properties} a theory of $p$-adic intertwining periods more general than those considered in \cite{BD08} (i.e. associated to any Weyl involution stabilizing $M$)
which culminates in Theorem \ref{theorem convergence and meromorphy of admissible periods}. We compute $\alpha$ 
only in the case $t=2r$ and $\tau=\tau_r=(r \ r+1)$, but a similar argument would give a similar expression for any $t\in \N-\{0\}$ and any 
$\tau$ commuting with $w$.\\
Then when $\Ind_P^G(\sigma)$ is a standard module, there is up to scalar 
at most one non zero $H$-invariant linear form $\Lambda$ on  $\Ind_P^G(\sigma)$ (see \cite{G15} for the proof of this fact which still holds in the non split context). This implies that the irreducible quotient $\La$ of $\Ind_P^G(\sigma)$ is distinguished if and only if $\Lambda$ 
descends to $\La$. Now suppose moreover that $\La$ is a ladder representation. When $t=2r+1$ is odd then $\Lambda$ descends to $\La$ if and only if $\d_{r+1}$ is distinguished, and the proof 
for split $G$ given in \cite{G15} is valid. When $t=2r$, as $\Lambda=J_\sigma(w,.,\underline{0},L)$ ($J_\sigma(w,.,\underline{s},L)$ is 
holomorphic at $\s=\underline{0}$ whenever $\Ind_P^G(\sigma)$'s irreducible quotient is a ladder), the part of the proof for split $G$
which is still valid shows that $\La$ is distinguished if and only if $J_\sigma(w,.,\underline{0},L)$ 
vanishes on the image of the regularized standard intertwining operator $A'(\tau_r,-a_r)$ from $\Ind_P^G(\tau_r(\sigma))$ to 
$\Ind_P^G(\sigma)$.
Our second main result (Theorem \ref{sufficient pole open periods ladders}) gives a sufficient condition for the intertwining period $J(w,.,\s,L)$ to have a singularity at $(0,\dots,0,-a_r,a_r,0,\dots,0)$. Combining this with 
our formula for $\alpha$ as well as the knowledge of analytic properties of standard intertwining operators, we show in our third main result (Theorem \ref{even ladder}) that $\La$ is distinguished if and only if the essentially square-integrable representation $\d$ corresponding to the segment which is the union of 
those corresponding to $\d_r$ and $\d_{r+1}$ is $\eta$-distinguished (here $\eta$ is the quadratic character attached to the extension $E/F$).\\
 Notice that we use \cite[Theorem 1]{BP18}, which asserts that an essentially square-integrable representation of $G$ is distinguished if and only if 
its Jacquet-Langlands transfer to the split form is distinguished. In particular, from what is known for the split form, 
the classification of proper ladders that we get is in terms of cuspidal distinguished representations (see Proposition 
\ref{distinction of essentially square-integrable representations}). However we show in Section \ref{reduction discrete cuspidal}, using the ideas of our paper, that we only need Beuzart-Plessis' result in the cuspidal case.\\
Finally, we classify in Section \ref{unitary and ladders} the distinguished ladder (not necessarily proper) and unitary 
representations of $G$.\\

It will be clear to the reader that the ideas of the paper also work when the index of $D$ over $F$ is even, however 
the double cosets $P\backslash G/H$ ($P$ a standard parabolic subgroup of $G$) are different, and the results will not take the same form. \\

\textbf{Acknowledgements}. We thank Erez Lapid for a very useful explanation/clarification concerning the paper 
\cite{JLR}. We thank Ioan Badulescu for correcting a misconception that the author had concerning the global Jacquet-Langlands correspondence. We thank Raphaël Beuzart-Plessis for a useful comment concerning infinite products of $\gamma$ factors. Most importantly we thank the referees for their extremely precise and useful comments and corrections. They in particular lead to the section on admissible intertwining periods (Section \ref{section admissible intertwining periods}) and important simplifications of many proofs. We thank the organizers of the conference "New Developments in Representation Theory" which took place in March 2016 at the National University of Singapore for their invitation, some of the ideas developed here first occurred there to the author. This work benefited from financial support of the grant ANR-13-BS01-0012FERPLAY. 
\section{Notations and preliminaries}

\subsection{Notations}\label{notations}

We denote by $\mathfrak{S}_n$ the symmetric group of degree $n$. For $w\in \S_n$, we denote by 
$\Inv(w)$ the set of inversions of $w$, i.e. the set couples $(i,j)$ of 
$\{1,\dots,n\}\times \{1,\dots,n\}$ such that $i<j$ but $w(i)>w(j)$. By definition $l(w)$ is the cardinality of $\rm{Inv}(w)$, it is known to be the length of $w$ with respect to the set of generators of $\S_n$ given by the transpositions $(i,i+1)$. We denote by 
$w_n$ the element of $\mathfrak{S}_n$ of maximal length, which sends $i$ to $n+1-i$.\\

If $V$ is a complex vector space, and $v$ and $w$ are two nonzero elements of $V$, we write $v\sim w$ if they span the same line. More generally if $M$ is an $R$-module over some ring $R$, and $m$ and $n$ are two nonzero elements 
of $M$, we write $m\underset{R^\times}{\sim} n$ if they are equal up to an invertible element of $R$.\\

If $G$ is a group, we denote by $Z(G)$ or $Z_G$ its center. If $G$ acts on $X$, we denote by $X^G$ the set of points of $X$ fixed by 
$G$. If $A$ is a ring (commutative or not, but always unital), we denote by $\M_{n,m}(A)$ the space of $n\times m$-matrices with coefficients in $A$, and set 
$\M_n(A)=\M_{n,n}(A)$. We denote by $G_n(A)$ the group of invertible elements in $\M_n(A)$.  We will often consider 
$\mathfrak{S}_n$ as a subgroup of $G_n(A)$ via the permutation matrices. If we denote by $A_n(A)$ the diagonal subgroup of $G_n(A)$, and by $T_n(A)$ the center of $A_n(A)$, we will denote by $\a_i$ the simple root of $T_n(A)$ defined by $\a_i(a)=a_i/a_{i+1}$.\\

 If $\n=(n_1,\dots,n_t)$ is a partition of $n$ (i.e. $n=n_1+\dots+n_t$ with the $n_i$ positive), we denote by $P_{\n}(A)=P_{(n_1,\dots,n_t)}(A)$ the subgroup of 
 matrices of the form $\begin{pmatrix} g_1 & \star &  \star \\ & \ddots & \star \\ & & g_t \end{pmatrix}$ with 
$g_i\in G_{n_i}(A)$, and call $P_{\n}(A)$ a standard parabolic subgroup of $G_n(A)$. We denote by 
$M_{\n}(A)$ the standard Levi subgroup of $P_{\n}(A)$, the elements of which are the matrices 
of the form $\diag(g_1,\dots,g_t)$ in $P_{\n}(A)$, and we denote by $N_{\n}(A)$ the subgroup 
of matrices $\begin{pmatrix} I_{n_1} & \star &  \star \\ & \ddots & \star \\ & & I_{n_t} \end{pmatrix}$ in 
$P_{\n}(A)$. If $\n=(1,\dots,1)$, we will write $B_n(A)$ for $P_{\n}(A)$, and $N_n(A)$ for $N_{\n}(A)$.\\

 Notice that $M_{\n}(A)$ being a product of linear groups, we can define in a similar fashion the standard 
parabolic subgroups of $M_{\n}(A)$, and they correspond to sub-partitions of $\n$. To $\n$, we also associate an injection $w\mapsto w^{\n}$ of the set $\S_t$ into $\S_n$. Writing 
$[1,n]=[I_1,\dots,I_t]$ where $I_k$ is an interval on integers of length $n_k$, the permutation 
$w^{\n}$ just permutes the intervals $I_k$ without changing the order inside such an interval. If the context allows it, 
we will most of the time remove the exponent $\n$ of $w^{\n}$. We say that a partition $\n=(n_1,\dots,n_t)$ is self-dual if $n_{w_t(i)}=n_i$ for all $i$ (remember that $w_t(i)=t+1-i$), in which case we say that the standard parabolic subgroup $P_{\n}(A)$ is self-dual as well.\\

We will use the letter $\K$ to denote the fields $\R$ or $\C$, and we set $|.|_\K$ for the absolute value on $\K$, which is 
the usual one on $\R$, and defined by $|z|_\C=z\overline{z}$ on $\C$. We will use the letter $F$ to denote a $p$-adic field (a finite extension of $\Q_p$), and the letter $E$ to denote a 
quadratic extension of $F$. We denote by $\w_E$ (or simply $\w$) a uniformizer of $E$, by $\w_F$ a uniformizer of $F$, by $v_E$ (or just $v$) and $v_F$ the valuations on $E$ and $F$, and by 
$|.|_E$ (or just $|.|$) and $|.|_F$ the normalized absolute values. We denote by $O_E$ (or just $O$) and $O_F$ the respective integer rings, by $q$ or $q_E$ the residual cardinality $|O/\w O|$ of $E$, and by $q_F$ that of $F$. We use the letter $k$ to denote a number field, and the letter $l$ to denote a quadratic extension of $k$. In both cases we denote by $\theta$ the corresponding Galois involution. We denote by $D_F$ a division algebra of center $F$, and of {\bf{\textit{odd index}}} over $F$(the index being the integer which is the square root of the dimension of a division algebra over its center), in which case $D_E:=D_F\otimes_F E$ (which we will also 
denote by $D$) is a division algebra with center $E$, and same index as $D_F$. We denote by $\Nrd_E$ (or $\Nrd$) and $\Nrd_F$ the reduced norms on $\M_m(D_E)$ and $\M_m(D_F)$. We denote by $\nu_E$ (or $\nu$) the map $|.|_E\circ \Nrd_E$, and set $\nu_F=|.|_F\circ \Nrd_F$, notice that $\nu_F=((\nu_E)_{|\M_m(D_F)})^{1/2}$. We also denote by $O_{D_E}$ the maximal order of $D_E$, and by 
$O_{D_F}$ that of $D_F$. Let $N_{E/F}$ be the norm map from $E$ to $F$, and $\eta_{E/F}$ be the quadratic character of $F^\times$ the kernel of which is $N_{E/F}(E^\times)$, we write $\eta$ for the character of $\mathrm{GL}(m,D_F)$ equal to $\eta_{E/F}\circ Nrd_F$. We shall sometimes denote by $\eta$ again an extension of $\eta$ to $\mathrm{GL}(m,D_E)$. We denote by $\A_k$ the ring of adeles of $k$, and by $\A_l$ that of $l$. We recall that we can view $\A_l$ as the restricted product $\prod' l_v$ over the places $v$ of $k$, of the the algebras 
$l_v:=l\otimes_k k_v$ ($k_v$ being the completion of $k$ with respect to $v$), and that either $l_v$ is a field (hence a quadratic extension of $k_v$) if $v$ remains nonsplit in $l$, or $l_v\simeq k_v\times k_v$ if $v$ splits in $l$.

\subsection{Haar measures}

If $G$ is a locally compact topological group, we denote by $\d_G$ its modulus character, defined by the fact that  
$\d_G d_{G,l} g$ is a right invariant Haar measure on $G$, if $d_{G,l} g$ is a left invariant one (some authors define the modulus character of $G$ to be $\d_G^{-1}$). We set $d_G$ to be a right invariant Haar measure on $G$, and most of the time omit the index $G$. 
More generally, if $H$ is a closed subgroup of $G$, then there is up to scalar a unique nonzero right invariant linear form on 
the space \[\mathcal{C}_c(H\backslash G,\d_G^{-1}\d_H)=\]
 \[\{f:G\rightarrow \C, \mathrm{continuous \ with \ support \ compact \ mod}\ H,\
f(hg)=\d_G^{-1}(h)\d_H(h)f(g),h\in H,g\in G\},\] which we will denote by $d_{H\backslash G} g$, or just by $dg$ again.\\

We recall that if $K<H<G$ is a chain of closed subgroups of $G$, then for 
$f\in \mathcal{C}_c(K\backslash G,\d_G^{-1}\d_K)$ and $g\in G$, the map 
$h\mapsto f(hg)\d_G(h)\d_H(h)^{-1}$ belongs to $\mathcal{C}_c(K\backslash H,\d_H^{-1}\d_K),$ and the map 
$f^{H,\d_G\d_H^{-1}}$ defined (up to the choice of $d_H h$) by the equality
 \[f^{H,\d_G\d_H^{-1}}(g)=\int_{K\backslash H}f(hg)\d_G(h)\d_H(h)^{-1}d_{K\backslash H} h\] belongs to 
 $\mathcal{C}_c(H\backslash G,\d_G^{-1}\d_H).$ We then have the formula (up to compatible normalizations):
 
 \[ \int_{H\backslash G}f^{H,\d_G\d_H^{-1}}(g)dg= \int_{K\backslash G}f(g)dg,\] which we can also write 
  \[ \int_{H\backslash G}(\int_{K\backslash H}f(hg)\d_G(h)\d_H(h)^{-1}d_H h)dg= \int_{K\backslash G}f(g)dg.\]

The above formula will tacitly be used a lot. In what follows, the Haar measures on the different subgroups of the general linear groups ($p$-adic, real, adelic) involved will be normalized by giving volume $1$ to maximal compact subgroups. In particular the usual integration formulas with respect to Iwasawa decomposition (for example) will be valid. 

\section{Representations of real, $p$-adic, and adelic groups}

If $\pi$ is a representation of a group or of an algebra, we will write $V_\pi$ for its underlying vector space, or even $\pi$. 
Moreover we will write $c_\pi$ for its central character whenever it has one.

\subsection{Representations of $p$-adic groups}\label{p-adic}

In this subsection we fix some notations in the $p$-adic setting, and recall some generalities about smooth representations. We set $d$ to be the index of $D$, $G=G_m=G_m(D_E)$, and $K=K_m=G_m(O_{D_E})$. We will only consider smooth complex representations of $G$ and its closed subgroups. We will denote by 
$\Ind$ and $\ind$ normalized smooth and compact smooth induction respectively. If $M$ is a standard Levi subgroup of $G$, and $M'$ is a standard Levi subgroup of $M$, we will denote by $r_{M',M}$ the normalized Jacquet functor.\\

 If $\pi$ is a representation of a closed subgroup $L$ of $G$, we call matrix coefficient of $\pi$ a function on $L$ of the form $l\mapsto <\pi(l)v,v^\vee>$ for $v\in \pi$, and $v^\vee$ in the smooth dual $\pi^\vee$ of $\pi$. If $\pi$ is irreducible and has a coefficient the absolute value of which belongs to $L^2(L/Z(L))$, we say that $\pi$ is square-integrable. If $\chi\otimes \pi$ is square-integrable for some character of $L$, we say that $\pi$ is an essentially square-integrable representation of $L$.\\

If $\overline{m}=(m_1,\dots,m_t)$ is a partition of $m$ and $\sigma=\sigma_1\otimes \dots \otimes \sigma_t$ is a representation of $M=M_{\m}=P_{\m}/N_{\m}$, we set \[\sigma_1\times \dots \times \sigma_t= 
\Ind_{P_{\m}}^{G_m}(\sigma)=\ind_{P_{\m}}^{G_m}(\sigma) .\]

If $\d$ is an essentially square-integrable representation of $G$, we denote by $e(\d)$ the unique real number such that 
$\nu^{-e(\d)}\otimes \d$ is square-integrable. If $\d_1,\dots,\d_t$ are essentially square-integrable representations of $G_{m_i}$ such that $e(\d_i)\geq e(\d_{i+1})$, then the representation $\d_1\times \dots\times \d_t$ is called a standard module. 
By \cite{S78}, it has a unique 
irreducible quotient $\La(\d_1,\dots,\d_t)$ (its Langlands' quotient). The representation $\pi=\La(\d_1,\dots,\d_t)$ determines the set $\{\d_1,\dots,\d_t\}$ uniquely, it is called the essentially square-integrable support of $\pi$. Any irreducible 
representation $\pi$ of $G$ is obtained in that manner. If $D=F$, we say that $\pi=\La(\d_1,\dots,\d_t)$ is generic if 
$\pi=\d_1\times \dots\times \d_t$, in which case the product is necessarily commutative. We shall give a more usual 
definition of generic representations in Section \ref{asai}.\\

\subsection{Representations of real groups}\label{real}

In this paragraph, $G=G_n$ will stand for $\GL(n,\K)$, for $\K=\R$ or $\C$. Let $\n=(n_1,\dots,n_t)$ be a partition of 
$n$, and write $P_{\n}=P_{\n}(\K)$, $M_{\n}=M_{\n}(\K)$... If $\K=\R$, $K=K_n$ will denote the compact orthogonal group $\rm{O}(n,\R)$, whereas it will denote the unitary group $\rm{U}(n,\C/\R)$ if $\K=\C$. We will consider 
finitely generated admissible (which is the same as finite length) $(\Lie(G),K)$-modules, as in Section 3.3 of \cite{W88}, or more generally finitely generated admissible $(\Lie(M),K\cap M)$-modules for standard Levi subgroups $M$ of $G$. We will 
call such modules Harish-Chandra modules of $M$. We will also need to consider finitely generated smooth admissible Fréchet modules of moderate growth (see \cite[Section 1]{W88} and 
\cite[Chapter 11]{W92} for the definition) of $M$, we will call such modules Casselman-Wallach representations. By the Casselman-Wallach globalization theorem (\cite{C89}, \cite{W83}, \cite[Chapter 11]{W92}, \cite{BK14}), it is known that if $\pi$ is a Harish-Chandra module of $M$, there is up to isomorphism a unique Casselman-Wallach representation $\overline{\pi}^\infty$, such that $\pi$ is the 
subspace $\overline{\pi}_{K\cap M}$ of $K\cap M$-finite vectors in $\pi$, and that the map $\pi\mapsto \overline{\pi}^\infty$ is 
an equivalence between the categories of $M$-Harish-Chandra modules and Casselman-Wallach representations of $M$, the quasi-inverse of which is given by $\tau\mapsto \tau_{K\cap M}$. It is shown in these references that one can always realize 
$\overline{\pi}^\infty$ as the dense subspace of smooth vectors in some Hilbert representation $\overline{\pi}$ (a continuous representation of $M$ in a Hilbert space), and in fact the space of smooth vectors in 
any Hilbert completion of $\pi$ provides a model for $\overline{\pi}^\infty$.\\

As Casselman-Wallach representations are nuclear Fréchet spaces, the injective and projective completed tensor products of two such representations $V_1$ and $V_2$ are the same, and we denote it by $V_1\widehat{\otimes}V_2$. If $\overline{\sigma}^\infty=\overline{\sigma_1}^\infty\widehat{\otimes}\dots\widehat{\otimes}\overline{\sigma_t}^\infty$ is a Casselman-Wallach representation of $M=M_{\n}=P_{\n}/N_{\n}$, we set 
\[\overline{\sigma_1}^\infty\times\dots\times\overline{\sigma_t}^\infty=
\Ind_{P_{\n}}^{G_n}(\overline{\sigma}^\infty),\] where 
here $\Ind_{P_{\n}}^{G_n}$ stands for normalized smooth parabolic induction as in Section 2.4 of \cite{AGS15}. If $\sigma=\sigma_1\otimes\dots \otimes \sigma_t$ is an $M$-Harish-Chandra module, we then set 
\[\Ind_{P_{\n}}^G(\sigma)=\sigma_1\times\dots \times \sigma_t:=(\overline{\sigma_1}^\infty\times\dots\times\overline{\sigma_t}^\infty)_K.\]

\subsection{Automorphic representations}\label{automorphic}

We recall that the notations $k$, $l$, $\theta$ and others have been introduced in Section \ref{notations}. We fix $\G$ a division algebra with center $k$ of \textit{\textbf{odd index}} $d$, so that the $l$-central simple algebra $\G_l=\G\otimes_k l$ is again a division algebra. For each place $v$ of $k$, we set 
$D_v=\G\otimes_k k_v$, and we say that $\G$ is split at $v$ if $D_v\simeq \M(d,k_v)$, it is in fact split at all places 
except a finite number. Note that $\G$ is split at the infinite places of $k$ because $d$ is odd. In general, $D_v$ is of the form $\M(d',\mathcal{D}_v)$ for $\mathcal{D}_v$ a $k_v$-division algebra, and with our conventions from Section \ref{notations}, one has 
$G_m(D_v)=\M(m,D_v)^\times=\M(md',\mathcal{D}_v)^\times$. With this identification we 
set $O_{D_v}$ to be $\M(d',O_{\mathcal{D}_v})$ for $O_{\mathcal{D}_v}$ the ring of integers of $\mathcal{D}_v$.
We recall that $l_v=k_v\otimes_k l$, and we set $D'_v=D_v\otimes_k l=\G\otimes_k l_v$. If $v$ splits in 
$l$, then $l_v\simeq k_v\times k_v$ and we identify  
$O_{l_v}$ with $O_{k_v}\times O_{k_v}$. We then define $O_{D'_v}$ by the equality $O_{D'_v}=O_{D_v}\otimes_{O_{k_v}} O_{l_v}$. We denote by $\A$ the ring of adeles of $k$. In this context, by definition 
\[G=G_m:=G_m(\G \otimes_k\A_l)=G_m(\G_l \otimes_l\A_l)\] is the restricted direct product of the groups $G_v=G_{m,v}=G_m(D'_v)$ (with respect to the compact open subgroups 
$K_v=K_{m,v}=G_m(O_{D'_v})$). More generally, if $S$ is a subset of $G$, we set $S_v=S\cap G_v$. Extending $\theta$ to $G$ in the natural manner, the subgroup 
$H=G^\theta=G_m(\G \otimes_k\A_k)$ is thus restricted direct product of the groups $H_v=H_{m,v}=G_m(D_v)$ (with respect to the compact open subgroups $K_v^\theta=G_m(O_{D_v})$). We set $G_\infty=\prod_{v|\infty} G_v$, and we identify each
 $G_v$ in the product to $G_{md}(\K)$ for $\K=\R\times \R$ or $\C$, depending of whether $v$ splits or not. We set $K_\infty$ 
to be the corresponding product of the maximal compact subgroups $O(md,\R)\times O(md,\R)$ or $U(md,\C/\R)$ of $G_{md}(\K)$.
 We denote by $K$ the product of $K_\infty$ with the product over $v$ finite of the groups $G_m(O_{D'_v})$. On each 
$G_v$, the reduced norm $\Nrd_{G_v}$ (with values in $l_v$) gives birth to 
$\Nrd_{G}=\prod'_v \Nrd_{G_v}$ with values in $\A_l^\times$, and we set $\nu=\nu_{\A_l}:g\mapsto |\Nrd_{G}(g)|_{\A_l}$.\\

If $M=M_{m_1,\dots,m_t}(\G \otimes_k\A_l)$ is a standard Levi subgroup of $G$, then $Z_M\simeq (\A_l^\times)^t$. We denote by $M^1$ the kernel of the homomorphism 
\[\diag(g_1,\dots,g_t)\in M\mapsto (\nu(g_1),\dots,\nu(g_t))\in (\R_{>0})^t.\] We recall that by definition $l_\infty=k_\infty\otimes_k l=\R\otimes_\Q l$, and denote by 
$A_M$ the subgroup of $Z_{M,\infty}$ corresponding to $(\R_{>0} \otimes_\Q 1)^t$ through the isomorphism above, in particular 
$M=A_M\times M^1 $.\\

To stick with the framework of \cite{JLR} and \cite{LR}, that we shall refer to a lot, we now consider $\sigma$ a $M\cap K$-finite cuspidal automorphic representation of $M$ (\cite[4.6]{BJ}). However we shall need to use other results which are written for smooth or $L^2$ automorphic representations. We denote by $\overline{V_\sigma}^\infty$ its smooth completion in the space of smooth cuspidal automorphic forms, and if $\sigma$ is unitary, we denote by $\overline{V_\sigma}$ its completion in 
the space of $L^2$-cuspidal automorphic forms. We denote by $\overline{\sigma}^\infty$ and by $\overline{\sigma}$ the corresponding representations of $M$. If $P=MN$ is the standard parabolic subgroup of $G$ with standard Levi $M$, we define $\Ind_{P}^{G}(\overline{\sigma}^\infty)$ to be the space of smooth functions from $G$ to $V_{\overline{\sigma}^\infty}$ satisfying:
\[f(mng)(\ . \ )=\d_P^{1/2}(m)f(g)(\ . \ m).\] 
We denote by $\Ind_{P}^{G}(\sigma)$ the subspace of $K$-finite vectors inside $\Ind_{P}^{G}(\overline{\sigma}^\infty)$.
For $f\in \Ind_{P}^{G}(\overline{\sigma}^\infty)$, we denote by $\widetilde{f}$ the map from 
$G$ to $\C$ defined by the equality: 
\[\forall g\in G,\ \widetilde{f}(g)=f(g)(I_n).\]
The map $f\mapsto \tilde{f}$ is injective from $\Ind_{P}^{G}(\overline{\sigma}^\infty)$ to its image, and we identify $\Ind_{P}^{G}(\overline{\sigma}^\infty)$ with this image, as well as  $\Ind_{P}^{G}(\sigma)$. Denoting by $\mathcal{A}_P(G)_{\sigma}$ the space of functions $\phi$ from $M(k) N(\A)\backslash G(\A)$ to $\C$, such that for all $k\in K$, the map $m\mapsto \phi(mk)$ belongs to $\d_P^{1/2}V_{\sigma}$, then the vector space $\Ind_{P}^{G}(\sigma)$ is a subspace of $\mathcal{A}_P(G)_\sigma$.\\

If $\sigma$ decomposes as a module of the global Hecke algebra of $M$: 
\[\sigma\simeq \otimes'_v \sigma_v,\] 
then \[\Ind_{P}^{G}(\sigma)\simeq \otimes'_v \Ind_{P_v}^{G_v}(\sigma_v).\] 
If $\sigma=\sigma_1\otimes \dots \otimes \sigma_t$, we will again use the notation 
 \[\Ind_{P}^{G}(\sigma)=\sigma_1\times \dots \times \sigma_t.\] 
 
 If $\sigma$ is unitary, there is also a natural definition for
 \[\Ind_{P}^{G}(\overline{\sigma})=\overline{\sigma_1}\times \dots \times \overline{\sigma_t},\] its subspace of 
 smooth vectors being $\Ind_{P}^{G}(\overline{\sigma}^\infty)$, and its subspace of $K$-finite vectors being $\Ind_{P}^{G}(\sigma)$.

\section{Local and global Jacquet Langlands correspondence}

We state here results from \cite{Z80}, \cite{DKV84}, \cite{T90}, \cite{B07} and \cite{B08} about essentially square-integrable representations of $G$ 
in the $p$-adic and adelic case, and the Jacquet-Langlands correspondence. Notice that we will use the letter $G$ for the 
possibly non split forms of $\GL(n)$, and $G'$ for the split form, which is the opposite convention to that used in 
\cite{B08} for example.

\subsection{The local correspondence}\label{local JL}

The results here are extracted from \cite{Z80}, \cite{DKV84}, and \cite{T90}, we refer to \cite{DKV84} for the
 definition of the Jacquet-Langlands correspondence. Here $G_m=\GL(m,D_E)$ is as in Section \ref{p-adic}, and we set $n=md$, and $G'_n=\GL(n,E)$. If $\rho'$ is a cuspidal representation of $G'_n$, $a\leq b$ are two real numbers equal modulo $\Z$, and if $\D'$ is the cuspidal segment $[a,b]_{\rho'}=\{\nu^a\rho',\dots,\nu^b\rho'\}$, we denote by $\d'=\La(\D')$ the unique irreducible quotient of the induced representation $\nu^a\rho'\times \dots \times \nu^b\rho'$. 
If $a=\frac{1-k}{2}$, and $b=\frac{k-1}{2}$, we will also write $\St_k(\rho')$ for $\La(\D')$. If $\rho'$ is unitary, then 
$\St_k(\rho')$ is a square-integrable representation, and all square-integrable representations of $G'_n$ are obtained in this manner. Now if $\rho$ is a cuspidal representation of $G_m$, then its Jacquet-Langlands transfer $\JL(\rho)$ to $G'_n$ is of the form 
$\St_l(\rho')$ for a unique $l\in \N_{>0}$ and a unique cuspidal representation of $G'_{\frac{n}{l}}$, and we set $l=l_\rho$ (it is known that $l_\rho$ divides $d$ and is coprime to $m$). This allows to extend the notion of cuspidal segment to $G_m$: if $\rho$ is a cuspidal representation of $G_m$ with $l=l_\rho$, and $c\leq d$ two real numbers equal modulo $\Z$, by definition the cuspidal segment 
$\D=[c,d]_\rho$ is the set $\{\nu^{lc}\rho,\dots,\nu^{ld}\rho\}$. The induced representation $\nu^{lc}\rho\times \dots \times \nu^{ld}\rho$ has a unique irreducible quotient $\d=\La(\D)$, we set $l_{\d}=l=l_\rho$. If $c=\frac{1-k}{2}$, and $d=\frac{k-1}{2}$, we will also write $\St_k(\rho)$ for $\La(\D)$. If $\rho$ is unitary, then 
$\St_k(\rho)$ is a square-integrable representation, and all square-integrable representations of $G_m$ are obtained in this manner. 
If $\rho$ is a cuspidal representation of $G_m$ such that $\JL(\rho)=\St_l(\rho')$, then for all $r\in \N_{>0}$, one has 
$\JL(\St_r(\rho))=\St_{rl}(\rho')$.

\subsection{The global correspondence}\label{global JL}

 We now recall a particular case of the main result of \cite{B08}. Denote by $\mathcal{P}(k)$ the set of places of $k$, and 
by $\mathcal{P}(l)$ that of $l$. We set $G=G_m=G_m(\G_l \otimes_l\A_l) =\prod'_{v\in \mathcal{P}(k)} G_v$ is as in Section \ref{automorphic}. Here it will in fact be more convenient to write $G=\prod'_{w\in \mathcal{P}(l)} G_w$, as the fact that $l$ is a quadratic extension of 
$k$ plays no role. We also set $G'=G'_{n}=G_{n}(\A_l)=\prod'_{w\in \mathcal{P}(l)} G'_w$. Notice that we allow the case $G=G'$ in what follows. For $c$ 
a unitary character of $\A_l^\times/l^\times$, we denote by $L^2(\A_l^\times G_m(\G_l)\backslash G,c)$ the space of functions 
$f$ on $G_{m}(\G_l)\backslash G$, transforming by $c$ under the center of $G$, and such that $|f|^2$ is integrable on 
$\A_l^\times G_{m}(\G_l)\backslash G$. We call $\tau$ a square-integrable representation of $G$ if it is an irreducible subspace (in the topological sense) of $L^2(\A_l^\times G_{m}(\G_l)\backslash G,c)$ for some such unitary character $c$. If $\rho$ is a unitary cuspidal representation 
of $G$ (in the $K$-finite sense), then $\overline{\rho}$ is square-integrable.\\

For $\pi$ a unitary cuspidal representation of $G$, by \cite[Theorem 5.1]{B08}, there exists a square-integrable representation $\JL(\overline{\pi})$ of $G'$ such that for all places $w$ where $G_w$ is split (i.e. $\G_l$ is split), then 
$(\overline{\pi})_w=\overline{\pi_w}=\JL(\overline{\pi})_w$. In fact, if $W$ is the finite set of finite places such that 
$\G_w=\G_l\otimes_l l_w$ is non split, by \cite[Proposition 5.5]{B08} together with \cite[Theorem 5.1]{B08}, a unitary cuspidal representation $\pi'$ of $G'$ is such that $\overline{\pi'}$ is equal to 
$\JL(\overline{\pi})$ for a (unique) unitary (necessarily) cuspidal representation $\pi$ of $G$, if and only if for all $w\in W$, the representation $\pi_w'$ is $d_w$-compatible (see \cite[Section 2.7]{B08}). In this case we set 
$\pi'=\JL(\pi)$, notice that $\pi'$ is the space of $K'$-finite vectors in $\JL(\overline{\pi})$. We don't recall the definition of $d_w$-compatible here but if $\pi_w'$ is square-integrable, then it is $d_w$-compatible; that is all that we need to know. Hence suppose that for all $w\in W$, the representation $\pi'_w$ is square-integrable, then \cite[Theorem 5.1]{B08} tells us that for all places $w\in W$, $\pi_w$ is square-integrable, and for all such places (hence for all $w\in \mathcal{P}(l)$ if we set $\JL$ to be the identity 
for places outside $W$), one has $\JL(\pi)_w=\JL(\pi_w)=\pi_w'$. 


\subsection{A globalization result}\label{globalization of essentially square-integrable representations}

In this section, we explain how to globalize a discrete series of an inner form of $\GL(n)$ as a local component of a cuspidal automorphic representation with cuspidal Jacquet-Langlands transfer.

\begin{LM}\label{lemma globalization of characters}
Let $v$ be a finite place of $l$ and $\chi_v$ be a unitary character of $l_v^\times$. 
Then there is a unitary character $\mu$ of $l^\times \backslash \A_l^\times$ such that $\mu_v=\chi_v$.
\end{LM}
\begin{proof}
Because $O_{l_v}^\times$ is a compact subgroup of $l^\times \backslash \A_l^\times$ the restriction ${\chi_v}_{|O_{l_v}^\times}$ extends to a unitary Hecke character $\mu$ of $\A_l^\times$ by Pontryagin duality. We then adjust $\mu$ if necessary by twisting it by a (necessarily unitary because $\chi_v$ is) power of the adelic norm of $\A_l^\times$.
\end{proof}

We now recall a special case of a result of Shin.

\begin{prop}{[\cite[Theorem 5.13]{Shin12}, special case].}\label{proposition Shin}
Let $S$ be a finite set of places of $l$, and for every $w\in S$ let $\d_w$ be a square-integrable representation of $\SL(n,l_w)$. Then there is a cuspidal automorphic representation $\pi$ of $\SL(n,\A_l)$ such that $\pi_w=\d_w$ for every $w\in S$.
\end{prop}
\begin{proof}
First note that the proof of \cite[Theorem 5.13]{Shin12} is a consequence of \cite[Theorem 4.8]{Shin12} which only assumes the last of the three assumptions of \cite[Section 4]{Shin12}. Moreover this latter assumption is satisfied by $\SL(n)$ thanks to 
\cite[Remark 4.1]{Shin12}. Note that we can always suppose that $S$ contains the set of infinite places, which is what we do in order to apply Shin's result (it is assumed as a hypothesis of \cite[Theorem 4.8]{Shin12}). Finally, thanks to \cite[Example 5.6]{Shin12}, the set $\widehat{U}$ of the statement of 
\cite[Theorem 5.13]{Shin12} can be chosen to be the singleton $\{\otimes_{w\in S}\d_w\}$, and the result follows.
\end{proof}

The result of Shin for $\SL(n)$ translates as follows for $\GL(n)$, as stated in \cite[Section 4]{C86}.

\begin{prop}\label{globalize AC}
Let $S$ be a finite set of places of $l$, and for every $w\in S$ let $\d_w$ be a square-integrable representation of $\GL(n,l_w)$. Then there exist a cuspidal automorphic representation $\pi$ of $\GL(n,\A_l)$ and unitary characters $\chi_w$ of $l_w^\times$ for every $w\in S$, such that $\pi_w=\chi_w\otimes \d_w$ for every $w\in S$. 
\end{prop}
\begin{proof}
For each $w\in S$, the restriction of $\d_w$ to $\SL(n,l_w)$ is a direct sum of finitely many irreducible representations (see the appropriate sections of \cite{HS12}), and 
we choose $\d_w'$ to be one of them. It is a well-known of consequence of Clifford theory that any irreducible representation of $\GL_n(l_w)$ containing $\d_w'$ in its restriction to $\SL_n(l_w)$ is of the form $\chi_w\otimes \d_w$ for $\chi_w$ a character of $l_w^\times$ (see \cite{HS12} again). According to Proposition \ref{proposition Shin}, there is a cuspidal automorphic representation $\pi'$ of 
$\SL(n,\A_l)$ such that $\pi'_w=\d'_w$ for all $w\in S$. It is then a consequence of \cite[Theorem 4.13]{HS12} that there is 
a unitary cuspidal automorphic representation $\pi$ of $\GL(n,\A_l)$ such that for any place $v$ of $l$, the local component 
$\pi_v$ contains $\pi'_v$ in its restriction to $\SL(n,l_v)$. In particular for all $w\in S$, there is a unitary character $\chi_w$ of $l_w^\times$ such that $\pi_w=\chi_w\otimes \d_w$. 
\end{proof}

We will use the following corollary of the result above.

\begin{cor}\label{globalize NM}
With the notations as above, if $\d_{w_0}$ is an essentially square-integrable representation of $G_{w_0}$ for $w_0$ a finite place of $l$, then there is a cuspidal automorphic representation $\pi$ of $G$, such that $\pi_{w_0}\simeq \d_{w_0}$, $\JL(\pi)$ is cuspidal, and $\JL(\pi)_{w_0}\simeq\JL(\d_{w_0})$.
\end{cor}
\begin{proof}
Up to torsion by an unramified character (as unramified characters extend to Hecke characters), we can assume that $\d_{w_0}$ is unitary. We set $\d'_{w_0}=\JL(\d_{w_0})$. Now we select for $w\in W$ (different from $w_0$ if $w_0\in W$) a square-integrable representation $\d'_w$ and 
apply Proposition \ref{globalize AC} to the family of representations $\d'_w$ for $w\in W\cup \{w_0\}$, to get a cuspidal representation $\pi'$ of $G'$ such that $\pi_w'\simeq \chi'_w\otimes \d'_w$ for $\chi'_w$ a unitary character of $l_w^\times$ for all $w\in W\cup \{w_0\}$. Then we extend $\chi'_{w_0}$ to a unitary Hecke character $\chi'$ of $\A_l^\times$ thanks to Lemma \ref{lemma globalization of characters} and set $\pi''=\chi'^{-1}\otimes \pi'$. The cuspidal automorphic representation $\pi''$ is of the form $\JL(\pi)$ for some cuspidal representation $\pi$ of $G$ according to Section \ref{global JL}, 
and moreover $\JL(\pi)_w=\JL(\pi_w)=\pi''_w$ for all places $w$ according to [ibid.]. In particular, 
$\JL(\pi_{w_0})=\JL(\d_{w_0})$, hence $\pi_{w_0}=\d_{w_0}$.
\end{proof}

\section{Basic results on the pair $(G,H)$}

\subsection{Multiplicity one and distinguished essentially square-integrable representations}

In this section, $G$ is as in Section \ref{p-adic}, and $H=G^\theta$ the subgroup of $G$ fixed by $\theta$. If $\pi$ is a representation of $G$, 
we say that it is distinguished (we will also say $H$-distinguished, or $\theta$-distinguished) if $\Hom_H(\pi,\C)$ is non zero. 
More generally if $\chi$ is a character of $H$, we say that $\pi$ is $\chi$-distinguished (we will also say $(H,\chi)$-distinguished, or $(\theta,\chi)$-distinguished) if 
$\Hom_H(\pi,\chi)$ is non zero. A pleasant property of the pair $(G,H)$ is that it affords multiplicity one, as it has been proved by Flicker when $G$ is split, and his proof has been extended to non split $G$ by Conliglio.

\begin{prop}[\cite{F91} Proposition 11, \cite{Con14} Appendix]\label{mult1}
Let $\pi$ be an irreducible representation of $G$, then $\Hom_H(\pi,\C)$ is of dimension at most $1$.
\end{prop}

Another classical result of Flicker (\cite[Proposition 12]{F91}) when $G$ is split is that if $\pi$ is irreducible distinguished, then it is conjugate self-dual, i.e. $\pi^\vee\simeq \pi^\theta$. It is maybe possible to extend Flicker's proof 
to the non split case, but we shall obtain it as a corollary of the classification of distinguished standard modules.\\

For essentially square-integrable representations of split $G$ with central character trivial on the $Z(G)^\theta$, some kind of converse statement is also true and is a result of Kable (\cite[Theorem 7]{K04}). The paper \cite{K04} has had a great influence on many of the author's works, and in particular it uses a local-global argument to obtain an equality of local factors, as we shall do later here. The result below is not original, it is a combination of various 
results of different authors (see the immediate proof for the references), including the very recent \cite[Theorem 1]{BP18} of Beuzart-Plessis already mentioned in the introduction. In fact we shall see in Section \ref{reduction discrete cuspidal} that we only need this result for cuspidal representations of $G$, as the technique developed here allows to reduce the study of distinction 
of essentially square-integrable representations to cuspidal representations. However for the moment, in order to already state the following result for essentially square-integrable representations, we shall use it for essentially square-integrable representations.

\begin{prop}\label{distinction of essentially square-integrable representations}
Let $\d=St_r(\rho)$ be an essentially square-integrable representation of $G$, and $l=l_\rho$. 
\begin{enumerate}
\item The representation $\d$ is distinguished if and only if the cuspidal representation $\rho$ is $\eta^{l(r+1)}$-distinguished. 
\item One has $\d^\vee\simeq \d^\theta$ if and only if either $\d$ is distinguished, or $\d$ is $\eta$-distinguished. The representation $\d$ cannot be both distinguished and $\eta$-distinguished at the same time. 
\end{enumerate}
\end{prop}
\begin{proof}
We recall that if $\JL(\rho)=\St_l(\rho')$, then $\JL(\d)=\St_{lr}(\rho')$. By \cite[Theorem 1]{BP18}, the representation $\d$ is distinguished if and only if $\JL(\d)$ is distinguished, but by Corollary 4.2 of \cite{M09} (see also the last section of \cite{AR05}), this is the case if and only if 
$\rho'$ is $\eta^{lr-1}$-distinguished, so by the same result, if and only if $\St_l(\rho')$ is 
$\eta^{l(r+1)}$-distinguished, i.e. if and only if $\rho$ is $\eta^{l(r+1)}$-distinguished by
 \cite[Theorem 1]{BP18} again. The second statement is also a consequence of \cite[Theorem 1]{BP18}, \cite[Theorem 7]{K04} and 
 \cite[Corollary 1.6]{AKT04}. 
\end{proof}

\subsection{Double cosets $P\backslash G/H$ and the geometric lemma}\label{Double cosets}

We denote by $f$ a field of characteristic different from $2$, and by $e$ a quadratic extension of $f$. We choose $\d_{e/f}$ an element of $e-f$, such that $\d_{e/f}^2\in f$. We denote by $D_f$ 
a central division algebra of odd index over $f$, and by $D_e$ the division algebra $D_f\otimes_f e$. We denote by 
$\theta$ the Galois involution of $e$ over $f$ and its various natural extensions. We denote by $G$ the group $G_m(D_e)$, 
by $H$ the group $G_m(D_f)$, and for $\m=(m_1,\dots,m_t)$ a partition of $m$, we set $P=P_{\m}(D_e)$, with its standard Levi decomposition $P=MN$. We denote by $I(\overline{m})$ the set of symmetric matrices of size $t$ with coefficients in $\N$, such that 
the sum of the $i$-th row is equal to $m_i$. In particular if $a$ belongs to $I(\m)$, the sequence of \textit{non zero coefficients} 
from the upper left corner to the bottom right corner of $a$ form a subpartition 
\[\m_a=(m_{1,1},m_{1,2},\dots,m_{t,t-1},m_{t,t})\] of $\m$. We denote by $P_{\m_a}$ the associated standard parabolic subgroup of $G$, it is contained in $P$. To each $a\in I(\m)$, we associate the element 
$u_a$ of $G$ defined as follows: the block $m_{i,j}\times m_{k,l}$ of $u_a$ is equal to $0_{m_{i,j},m_{k,l}}$ if $\{k,l\}\neq \{i,j\}$, the block $m_{i,i}\times m_{i,i}$ is equal to $I_{m_{i,i}}$, and if $i<j$, the block 
$(m_{i,j}\cup m_{j,i})\times (m_{i,j}\cup m_{j,i})$ (we hope that the reader finds this intuitive notation clear enough, we recall in passing that $m_{i,j}=m_{j,i}$) 
is equal to \[\begin{pmatrix} I_{m_{i,j}} & -\d_{e/f} I_{m_{i,j}}\\ I_{m_{i,j}} & \d_{e/f} I_{m_{i,j}}\end{pmatrix}.\]

The proof of Propositions 3.7 and 3.9 of \cite{M11} is valid in the generality in which we state the following result, to which we add obvious observations.

\begin{prop}\label{representatives}
\begin{itemize}
\item The matrices $u_a$, when $a$ varies in $I(\m)$ form a set of representatives $R(P\backslash G/H)$ of 
the double cosets $P\backslash G/H$. 
\item The element $w_a=u_au_a^{-\theta}$ is a permutation matrix of order $2$, and if one writes $[1,n]=[I_{1,1},I_{1,2},\dots,I_{t,t-1},I_{t,t}]$ with $I_{i,j}$ of length $m_{i,j}$, then $w_a$ fixes $I_{i,i}$, its restriction to any $I_{i,j}$ is order preserving and it exchanges the intervals 
$I_{i,j}$ and $I_{j,i}$. Moreover the map $u_a\mapsto w_a$ is injective.
\item If one sets $\theta_{w_a}:g\in G\mapsto w_a\theta(g)w_a^{-1}$, then $G^{\theta_{w_a}}=u_a H u_a^{-1}$, and for 
\[m=diag(g_{1,1},g_{1,2},\dots,g_{t,t-1},g_{t,t})\in M_{\m_a},\] 
$\theta_{w_a}(m)$ is the element \[diag(g'_{1,1},g'_{1,2},\dots,g'_{t,t-1},g'_{t,t})\in M_{\m_a},\] where 
$g'_{i,j}=\theta(g_{j,i})$.
\end{itemize}
\end{prop}

We shall as well write $w_{u_a}$ instead of $w_a$. We will often write $P_a=M_aN_a$ or $P_{u_a}=M_{u_a}N_{u_a}$ for $P_{\m_a}$. For $X\subset G$ and $u\in\Rep$, we will sometimes write $X^u$ or $X^{w_u}$ for $X^{\theta_{w_u}}$, and $X(u)$ for $u^{-1} X^u u$.\\

Following \cite[Definition 1]{JLR} or \cite[Definition 3.1.2]{LR} with $\sigma=I_d$, we say that $u\in \Rep$ is $P$-admissible if $w_uMw_u^{-1}=M$. In particular any $u_a\in \Rep$ is 
$P_{\m_a}$-\textit{admissible}. If $u$ is $P$-admissible, then $P^u=M^uN^u$, hence $P(u)=M(u)N(u)$. Take $u\in \Rep$, an easy consequence of the inclusion $P\cap w_uPw_u^{-1}\subset P_u$ is the equality (see \cite[Proposition 4.1]{M11})
\[P^u=P_u^u=M_u^u N_u^u.\]


Now $f=F$ and $e=E$. We denote by $\sigma$ a smooth representation of $M$. Then according to the discussion 
before \cite[Lemma 5.5]{M11} (see more generally \cite[Theorem 1.1 and Corollary 6.9]{O17}), one has the following consequence of \cite[Theorem 5.2]{BZ77}.

\begin{prop}\label{Geometric lemma}
\begin{itemize}
\item There is a vector space injection of $\Hom_H(ind_P^G(\sigma),\C)$ into 
\[\prod_{u\in \Rep} Hom_{P_u^u}(\frac{\d_P^{1/2}}{\d_{P_u^u}}\sigma,\C).\]
\item \[Hom_{P_u^u}(\frac{\d_P^{1/2}}{\d_{P_u^u}}\sigma,\C)\simeq Hom_{M_u^u}(r_{M_u,M}(\sigma),\C)\]
\end{itemize}
\end{prop}

\begin{rem}\label{pure tensor} The following useful observation is \cite[Lemma 2.1]{G15}. If $u=u_a\in \Rep$ corresponds to $a\in I(\m)$ (setting $\m_a=(m_{1,1},\dots,m_{t,t})$), $\sigma$ is an admissible representation of $M$, and the Jacquet module $r_{M_u,M}(\sigma)$ is a pure tensor 
$\otimes_{i,j} \sigma_{i,j}$ with $\sigma_{i,j}$ a representation of $G_{m_{i,j}}$, then 
\[Hom_{M_u^u}(r_{M_u,M}(\sigma),\C)\simeq \otimes_{i=1}^t 
Hom_{H_{m_{i,i}}}(\sigma_{i,i},\C)\otimes_{1\leq i<j\leq t} Hom_{G_{m_{i,j}}}(\sigma_{j,i},(\sigma_{i,j}^{\theta})^{\vee}).\]
If $(\l,\dots,\l)$ is a partition of $m$, $\rho$ is a cuspidal representation of $G_\l$, and $a$ and $b$ are integers with 
$b+1-a=m/\l$. We set $\D=[a,b]_\rho$. We say that a partition $\m=(m_1,\dots,m_t)$ of $m$ is $\rho$-adapted if each $m_i$ is a multiple of $\l$. If $\m$ is such a partition, by convention, we write $\D=[\D_1,\dots,\D_t]$ with 
$\D_{t+1-i}=[a+\frac{m_1+\dots+m_{i-1}}{\l}, a+\frac{m_1+\dots+m_i}{\l}-1]_\rho,$ and if $\d=L(\D)$, we set 
$\d_i=L(\D_i)$. We recall that by \cite[Proposition 7.16]{MS14}, \[r_{M_{\m},G}(\d)=0\] if $\m$ is not $\rho$-adapted, and 
\[r_{M_{\m},G}(\d)=\d_1\otimes \dots \otimes \d_t\] if $\m$ is $\rho$-adapted. In particular this remark applies to essentially square-integrable representations.
\end{rem}

\section{Rankin-Selberg and Asai $L$-functions}\label{asai}

\textit{Notice that for simpler notations, in this section and the rest of the paper, our definition of $\gamma$ and $\e$-factors might differ from that of the usual sources by a sign}.\\

Let $L$ be a $p$-adic field, and $\mu$ a non trivial character of $L$. Then we denote by $\mu$ again the non degenerate character of 
$N_n(L)$ defined by $\mu(u)=\mu(\sum_{i=1}^{n-1} u_{i,i+1})$. By \cite[Theorem 9.7]{Z80}, it is equivalent to say that a representation $\pi$ of $G_n(L)$ is generic or that $\Hom_{N_n(L)}(\pi,\mu)\neq \{0\}$, and in this case, the dimension of 
$\Hom_{N_n(L)}(\pi,\mu)$ is one by \cite{GK75}. This allows to embed in a unique way (up to scaling) a generic representation $\pi$ of $G_n(L)$ in $\Ind_{N_n(L)}^{G_n(L)}(\mu)$, in which case we denote by $W(\pi,\mu)$ the image of this embedding, that we call the Whittaker model of $\pi$. For $W\in W(\pi,\mu)$, the map $\widetilde{W}:g\mapsto W(w_n{}^t\!g^{-1})$ belongs to 
$W(\pi^\vee,\mu^{-1})$. The space $\mathcal{C}_c^\infty(L^n)$ is by definition that of smooth functions on $L^n$ 
with compact support. 

\subsection{The $p$-adic Asai $L$-factor}

Let $\psi$ be a non trivial character of $E$ trivial on $F$, it is of the form $z\mapsto \psi'(\d_{E/F}(z-\theta(z)))$ 
for a unique non trivial character $\psi'$ of $F$. We set $\e=\e_n=(0,\dots,0,1)$ in $\M_{1,n}(\Z)$. If $\pi$ is a generic representation of $G_n(E)$, for $W\in W(\pi,\psi)$, $\phi \in \sm_c(F^n)$, and $s\in \C$, we define the Asai integrals 
\[I^+(s,W,\phi)=\int_{N_n(F)\backslash G_n(F)} W(h)\phi (\e h)\nu_F(h)^sdh,\]
and 
\[I^-(s,W,\phi)=\int_{N_n(F)\backslash G_n(F)} W(h)\phi (\e h)\eta(h)\nu_F(h)^sdh.\]

By the appendix of \cite{F93}, there is $r_\pi\in \R$ such that for $\Re(s)\geq r_\pi$, all integrals $I^+(s,W,\phi)$ 
(resp. $I^-(s,W,\phi)$) converge absolutely. They in fact extend to elements of $\C(q_F^{-s})$, and the vector space they span, as $W$ and $\phi$ vary, is a fractional ideal of $\C[q_F^{-s},q_F^s]$ with a unique generator normalized by the fact that it is the inverse of a 
polynomial in $q_F^{-s}$ with constant term $1$, which we denote by \[L^+(s,\pi) \ ({\rm{resp.}}\ L^-(s,\pi))\] and call the even (resp. the odd) Asai $L$-function of $\pi$.\\

The following result describes when the Asai $L$-factor attached to a cuspidal  
representation has a pole. It is a consequence of \cite[Corollary 1.5]{AKT04} (see \cite[Proposition 3.6]{M10} for a different approach).

\begin{prop}\label{poles Asai L factor}
Let $\pi$ be a cuspidal representation of $G_n(E)$. Then the Asai $L$-factor $L^+(s,\pi)$ (resp. $L^-(s,\pi)$) has a pole at $s_0$  if and only if $\pi$ is $\nu_F^{-s_0}$-distinguished (resp. $\eta \nu_F^{-s_0}$-distinguished), and such a pole is always simple.
\end{prop}

Finally, we recall the local functional equation of the local Asai $L$-factor, which can again be found in the appendix of 
\cite{F93}. We denote by $\widehat{\phi}$ the Fourier transform of $\phi$ with respect to $\psi'$-self dual Haar measure 
on $F^n$.

\begin{prop}\label{functional equation local Asai}
Let $\pi$ be a generic representation of $G_n(E)$, and $\mathfrak{e}\in\{+,-\}$, there is a unit $\e^\mathfrak{e}(s,\pi,\psi)$ of $\C[q_F^{\pm s}]$ such that if one sets 
\[\gamma^\mathfrak{e}(s,\pi,\psi)=\e^\mathfrak{e}(s,\pi,\psi)
\frac{L^\mathfrak{e}(1-s,\pi^\vee)}{L^\mathfrak{e}(s,\pi)}, \] then for any $W\in W(\pi,\psi)$ and $\phi \in \sm_c(F^n)$, one has 
\[I^\mathfrak{e}(1-s,\widetilde{W},\widehat{\phi })= \gamma^\mathfrak{e}(s,\pi,\psi)
I^\mathfrak{e}(s,W,\phi).\]
\end{prop}

Finally, we will need the following consequence of the inductivity relation of Asai $L$-factors of essentially square-integrable representations. First notice that we also denote by $\eta$ (see Section \ref{notations}) any extension of the character $\eta_{E/F}\circ \det$ of $G_n(F)$ to $G_n(E)$. 

\begin{prop}\label{asai gamma poles}
Let $\St_k(\rho)$ be a conjugate self-dual and essentially square-integrable representation of $G_n(E)$. Then \[\gamma^+(-s,\St_k(\rho),\psi)^{-1}\gamma^-(s,\St_k(\rho),\psi)^{-1}\underset{\C[q^{\pm s}]^\times}{\sim}
\frac{L^+(s,\eta^k\rho)}{L^+(s+k,\rho)}\frac{L^+(-s,\eta^{k+1}\rho)}{L^+(-s+k,\eta\rho)} .\]
\end{prop}
\begin{proof}
We recall that according to \cite[Corollary 4.2]{M09}, one has the relation 
\[L^\mathfrak{e}(s,\St_k(\rho))=\prod_{i=0}^{k-1}L^\mathfrak{e}(s+i,\eta^{k-1-i}\rho).\]
In particular: 
\[\frac{L^-(s,\St_k(\rho))}{L^+(1+s,\St_k(\rho))}=\frac{\prod_{i=0}^{k-1}L^-(s+i,\eta^{k-1-i}\rho)}
{\prod_{i=0}^{k-1}L^+(s+i+1,\eta^{k-1-i}\rho)}=\frac{\prod_{i=0}^{k-1}L^-(s+i,\eta^{k-1-i}\rho)}
{\prod_{i=1}^{k}L^+(s+i,\eta^{k-i}\rho)}\]
\[=\frac{\prod_{i=0}^{k-1}L^+(s+i,\eta^{k-i}\rho)}
{\prod_{i=1}^{k}L^+(s+i,\eta^{k-i}\rho)}= \frac{L^+(s,\eta^k\rho)}{L^+(s+k,\rho)}.\]
This implies that \[\frac{L^+(-s,\St_k(\rho))}{L^-(1-s,\St_k(\rho))}=\frac{L^+(-s,\eta^{k+1}\rho)}{L^+(-s+k,\eta\rho)}.\] 
Note that because $\St_k(\rho)$ is conjugate self-dual, one has $L^\e(s,\St_k(\rho)^\vee)=L^\e(s,\St_k(\rho)^\sigma)=L^\e(s,\St_k(\rho))$ for $\e\in\{\pm 1\}$, hence we deduce the relation 
\[\gamma^+(-s,\St_k(\rho),\psi)^{-1}\gamma^-(s,\St_k(\rho),\psi)^{-1}\underset{\C[q^{\pm s}]^\times}{\sim}
\frac{L^+(s,\eta^k\rho)}{L^+(s+k,\rho)}\frac{L^+(-s,\eta^{k+1}\rho)}{L^+(-s+k,\eta\rho)} .\]
\end{proof}

\subsection{The $p$-adic Rankin-Selberg $L$-factor}

Let $\psi$ be a non trivial character of $F$, and $\pi$ and $\pi'$ be generic representations of $G_n(F)$. For 
$W\in W(\pi,\psi)$, $W'\in W(\pi',\psi^{-1})$, $\phi\in \sm(F^n)$, and $s\in \C$, we 
define the the Rankin-Selberg integral 
\[I(s,W,W',\phi)=\int_{N_n(F)\backslash G_n(F)} W(h)W'(h)\phi(\e h)\nu_F(h)^sdh.\]

By \cite{JPSS83}, there is $r_{\pi,\pi'}\in \R$ such that for $\Re(s)\geq r_{\pi,\pi'}$, all integrals 
$I(s,W,W',\phi)$ converge absolutely. They in fact extend to elements of $\C(q_F^{-s})$, and the vector space they span, as $W$ $W'$, and $\phi$ vary, is a fractional ideal of $\C[q_F^{-s},q_F^s]$ with a unique generator normalized by the fact that it is the inverse of a polynomial in $q_F^{-s}$ with constant term $1$, which we denote by $L(s,\pi,\pi')$, and call the Rankin-Selberg $L$-factor of $(\pi,\pi')$.\\

The poles of Rankin-Selberg $L$-factors attached to a pair of cuspidal representations are described in \cite[Proposition 8.1]{JPSS83}, the result is as follows.

\begin{prop}\label{poles Rankin-Selberg L factor}
Let $\pi$ and $\pi'$ be cuspidal representations of $G_n(F)$, then the Rankin-Selberg $L$-factor $L(s,\pi,\pi')$ has a pole at $s_0$ if and only if $\pi'\simeq \nu_F^{-s_0}\pi^\vee$. Such a pole is always simple.
\end{prop}

To state the functional equation of the $p$-adic Rankin-Selberg $L$-factor, for $\phi\in \sm(F^n)$, we denote by $\widehat{\phi}$ its Fourier transform with respect to a $\psi$-self dual Haar measure on $F^n$. Then by \cite[Theorem 2.7]{JPSS83}, one has:

\begin{prop}\label{functional equation p-adic rankin-selberg}
Let $\pi$ and $\pi'$ be generic representations of $G_n(F)$, there is a unit $\e(s,\pi,\pi',\psi)$ of $\C[q_F^{\pm s}]$ such that if one sets \[\gamma(s,\pi,\pi',\psi)=\e(s,\pi,\pi',\psi)
\frac{L(1-s,\pi^\vee,{\pi'}^\vee)}{L(s,\pi,\pi')},\] then for any $W\in W(\pi,\psi)$, 
$W'\in W(\pi',\psi^{-1})$ and $\phi \in \sm(F^n)$, one has 
\[I(1-s,\widetilde{W},\widetilde{W'},\widehat{\phi })= \gamma(s,\pi,\pi',\psi)
I(s,W,W',\phi).\]
\end{prop}

\subsection{Archimedean Rankin-Selberg gamma factors}

Let $\K$ be $\R$ or $\C$, we recall some results from \cite{J09}, and refer the reader to the references therein 
for the original bibliography on the subject, which in any case is due to the author of \cite{J09} and his collaborators. We denote by $\psi$ again a non trivial character of $\K$, as well as its extension to $N_n(\K)$ as before. If $\pi$ is an irreducible Casselman-Wallach representation of $G_n(\K)$, such that there is a nonzero continuous linear form in 
$\Hom_{N_n(\K)}(\pi,\psi)$, we will call $\pi$ a generic representation. If $\pi$ is unitary, it is known that such a linear form is unique by \cite{Shal74}, and for generic 
$\pi$ (and more generally when $\pi$ is parabolically induced from discrete series representations), this fact still holds as explained in 
\cite[p. 4]{J09}. If $\pi$ is generic, we denote by $W(\pi,\psi)$ its Whittaker model (the space of functions 
$g\in G_n(\K) \mapsto \l(\pi(g)v)$ for $v\in V_\pi$ and $\l\in \Hom_{N_n(\K)}(\pi,\psi)$). Again $\widetilde{W}:g\mapsto 
W(w_n{}^t\!g^{-1})$ belongs to $W(\pi^\vee,\psi^{-1})$ if $W$ belongs to $W(\pi,\psi)$.\\

We denote by $\mathcal{S}(\K^n)$ the space of Schwartz functions on $\K^n$. For $\phi\in \mathcal{S}(\K^n)$, we denote by 
$\widehat{\phi}$ its Fourier transform with respect to the $\psi$-self dual Haar measure on $\K^n$. For $\pi$ and $\pi'$ two generic representations of $G_n(\K)$, $W\in W(\pi,\psi)$, $W'\in W(\pi',\psi^{-1})$, and 
$\phi\in \mathcal{S}(\K^n)$, the definition of the archimedean Rankin-Selberg integral $I(s,W,W',\phi)$ is the same as in the $p$-adic case. The following assertions are a consequence of \cite[Theorem 2.1]{J09} and its proof. 

\begin{prop}
\begin{itemize}
Let $\pi$ and $\pi'$ two generic representations of $G_n(\K)$.
\item  There is $r_{\pi,\pi'}\in \R$ such that for $\Re(s)\geq r_{\pi,\pi'}$, all integrals 
$I(s,W,W',\phi)$ converge absolutely for $W\in W(\pi,\psi)$, $W'\in W(\pi',\psi^{-1})$, and $\phi$ in 
$\mathcal{S}(\K^n)$, and they extend to meromorphic functions on $\C$.
\item There is a meromorphic function $\gamma(s,\pi,\pi',\psi)$ such that for all $W\in W(\pi,\psi)$, $W'\in W(\pi',\psi^{-1})$, and $\phi$ in 
$\mathcal{S}(\K^n)$, the following equality holds:
\[I(1-s,\widetilde{W},\widetilde{W'},\widehat{\phi})=\gamma(s,\pi,\pi',\psi)I(s,W,W',\phi).\]
\end{itemize}
\end{prop}

\subsection{The functional equation of the global Asai and Rankin-Selberg $L$-functions}\label{functional equation global asai}

First we need to set up a convention. Let $F$ be a $p$-adic field, $\R$ or $\C$, and $\psi$ be a non trivial character of $F$. Let $\pi$ and $\pi'$ be generic representations of $G_n(F)$, so that $\pi\otimes \pi'$ is a smooth representation of $G_n(F)\times G_n(F)$ if $F$ is $p$-adic, and $\pi\widehat{\otimes}\pi'$ is a Casselman-Wallach representation of $G_n(F)\times G_n(F)$ if $F=\R$ or $\C$. Then by definition, we set \[L^+(s,\pi\otimes \pi')=L^-(s,\pi\otimes \pi')=L(s,\pi,\pi'),\] \[\e^+(s,\pi\otimes \pi',\psi\otimes \psi^{-1})=\e^-(s,\pi\otimes \pi',\psi \otimes \psi^{-1})=
\e(s,\pi,\pi',\psi),\] and \[\gamma^+(s,\pi\otimes \pi',\psi\otimes \psi^{-1})=\gamma^-(s,\pi\otimes \pi',\psi\otimes \psi^{-1})=
\gamma(s,\pi,\pi',\psi),\] when $F$ is $p$-adic, and 
 \[\gamma^+(s,\pi\widehat{\otimes} \pi',\psi\otimes \psi^{-1})=\gamma^-(s,\pi\widehat{\otimes} \pi',\psi\otimes \psi^{-1})=
\gamma(s,\pi,\pi',\psi)\] when $F=\R$ or $\C$. In fact, if $F=\R$ or $\C$, and $\pi$ and $\pi'$ are Harish-Chandra modules of $G_n(F)$ such that $\overline{\pi}^\infty$ and $\overline{\pi'}^\infty$ are generic, we set 
\[\gamma^\mathfrak{e}(s,\pi \otimes \pi',\psi\otimes \psi^{-1})=\gamma^\mathfrak{e}(s,\overline{\pi}^\infty\widehat{\otimes} \overline{\pi'}^\infty,\psi\otimes \psi^{-1})\] for $\mathfrak{e}\in \{+,-\}$. \\

We now suppose that the number fields $k$ and $l$ are such that all infinite places of $k$ split in $l$. We denote by $\pi$ a cuspidal representation of $G_n(\A_l)=\prod'_{v\in \mathcal{P}(k)} G_v$, so that $\pi=\otimes'_{v\in \mathcal{P}(k)} \pi_v$. Finally we 
take $\psi$ a non trivial character of $\A_l$, trivial on $l+\A_k$, so that $\psi=\otimes'_{v\in \mathcal{P}(k)}\psi_v$, 
where $\psi_v$ is a character of $l_v$ trivial on $k_v$. The functional equation of the global Asai integrals is proved in 
\cite{F88} but not explicitly stated there, so we refer to \cite[Propositions 5 and 6]{K04} for it. 

\begin{prop}\label{Asai global equation proposition}
Take $\pi$ and $\psi$ as above, and $\mathfrak{e}\in  \{+,-\}$. Let $S\subset \mathcal{P}(k)$ be a finite set which contains all archimedean and ramified places (by ramified we mean that either the representation, or the additive character, or the quadratic extension is ramified). The product $\prod_{v\notin S} L^\mathfrak{e}(s,\pi_v)$ is convergent for $\Re(s)$ large enough, and it extends to a meromorphic function 
$L^{S,\mathfrak{e}}(s,\pi).$ The global functional equation of the Asai $L$-function is:
\begin{equation} \label{Asai global equation} \prod_{v\in S} \gamma^\mathfrak{e}(s,\pi_v,\psi_v) L^{S,\mathfrak{e}}(1-s,\pi^\vee)=L^{S,\mathfrak{e}}(s,\pi).\end{equation}
\end{prop}

We shall also need the following basic result on the partial Rankin-Selberg $L$-function which can be extracted from Section 4 of \cite{JS81}. We will use the following convention: 
for $v\in \mathcal{P}(k)$, we will write $L(s,\pi_v,\pi'_v)$ for $L(s,\pi_w,\pi'_w)$ if $v$ does not split in $l$ and $w$ is the unique place of $l$ dividing $v$, or for the product $L(s,\pi_{w_1},\pi'_{w_2})L(s,\pi_{w_2},\pi'_{w_2})$ if $v$ splits in $l$ into the places $w_1$ and $w_2$. 

\begin{prop}\label{RS global equation proposition}
Take $\pi$, $\psi$ and $S$ as above, and $\pi'$ another cuspidal automorphic representation of $G_n(\A_l)$. The product $\prod_{v\notin S} L(s,\pi_v,\pi'_v)$ is convergent for $\Re(s)$ large enough and it extends to a meromorphic function 
$L^S(s,\pi,\pi').$ 
\end{prop}

\section{Standard intertwining operators}

\subsection{Generalities}\label{generalities on intertwining operators}

Here we consider $G$ as in one of the paragraphs \ref{p-adic}, \ref{real}, or \ref{automorphic}. We take 
$M=M_{\m}$ a standard Levi subgroup of $G$ for $\m=(m_1,\dots,m_t)$, $P=P_{\m}$, and $N=N_{\m}$. Let 
$\sigma$ be a representation of $M$. If $G$ is $p$-adic $\sigma$ is smooth of finite length, if $G$ is real $\sigma$ is a Harish-Chandra module (of finite length by definition), and if $G$ 
is adelic $\sigma$ is $K$-finite cuspidal automorphic representation.\\

For $\s=(s_1,\dots,s_t)$ in $\C^t$, we write $\sigma[\s]=(\nu^{s_1}\otimes \dots \otimes \nu^{s_t})\sigma,$
and $\Ind_P^G(\sigma,\s)=\Ind_P^G(\sigma[\s])$. If $\pi=\Ind_P^G(\sigma)$, we write $\pi_{\s}$ for $\Ind_P^G(\sigma,\s)$. If we write $m\in M$ as 
$m=\diag(g_1,\dots,g_t)$, then in terms of the Iwasawa decomposition of $G$, one defines 
\[\eta_{\s}(umk)=\prod_{i=1}^t \nu(g_i)^{s_i}.\]
We then define, for $f$ in $\Ind_P^G(\sigma)$ and $\s\in \C^t$, the map $f_{\s}=\eta_{\s}f$ which belongs to $\Ind_P^G(\sigma,\s)$. We call 
$f_{\s}$ a flat section (which means that the restriction of $f_{\s}$ to $K$ is independant of $\s$). We denote by 
$\F(\sigma)$ the space of flat sections of $\Ind_P^G(\sigma,\s)$. We shall need the following lemma concerning flat 
sections in the next section.

\begin{LM}\label{surjective restriction}
Take $G$ as in Section \ref{p-adic} and suppose that 
$\sigma$ is of the form $\sigma=\sigma_1\otimes \dots \otimes \sigma_t$. Suppose that $t=2r$ is even, and that $m_{t+1-i}=m_i$ for $i\in [1,t]$, then the map 
\[R:f_{\s}\mapsto [g\mapsto f_{\s}(\diag(I_{m_1},\dots, I_{m_{r-1}}, g , I_{m_{r+2}},\dots,I_{m_t}))]\] defines 
a surjection from the space $\F(\sigma)$ to the space 
\[V_{\sigma_1}\otimes\dots \otimes V_{\sigma_{r-1}} \otimes \F(\sigma_r\otimes \sigma_{r+1})
\otimes V_{\sigma_{r+2}}\otimes \dots \otimes V_{\sigma_t}.\]
\end{LM}
\begin{proof}
The assumptions on the $m_i$'s guarantee that the modulus character of $P$ restricts to 
$P_{(m_r,m_{r+1})}$ as its modulus character, so the image of $R$ is indeed a subspace of  \[V_{\sigma_1}\otimes\dots \otimes V_{\sigma_{r-1}} \otimes \F(\sigma_r\otimes \sigma_{r+1})
\otimes V_{\sigma_{r+2}}\otimes \dots \otimes V_{\sigma_t}.\]
Take now $h_{(s_r,s_{r+1})}$ an element of $\F(\sigma_r\otimes \sigma_{r+1})$, and $v_i\in V_{\sigma_i}$ for $i\notin \{r,r+1\}$, to show the surjectivity of $R$, it is enough to find a flat section $f_{\s}$ in $\F(\sigma)$ 
such that \[R(f_{\s})
=v_1\otimes \dots \otimes v_{r-1} \otimes h_{(s_1,s_2)}\otimes v_{r+2}\otimes \dots \otimes v_t.\] 
Let's denote by $L$ the standard Levi subgroup 
$M_{(m_1,\dots,m_{r-1},2m_r,m_{r+2},\dots,m_t)}$, and by $Q=LU$ the associated standard parabolic subgroup. Fix $c\geq 1$ large 
enough for $K_{m_i}(c)=I_{m_i}+\w^c\M_{m_i}(O_E)$ to fix $v_i$ for $i\notin \{r,r+1\}$, and for 
$K_{2m_r}(c)$ to fix $h_{(s_r,s_{r+1})}$. Then, because $K_m(c)$ has an Iwahori decomposition with respect to $Q$, the map with support in $LUK_m(c)$, defined on this set by the equality 
\[f_s(luk)=\d_Q(l)^{\frac{1}{2}}\sigma_1(g_1)v_1\otimes \dots \otimes \sigma_d(g_{r-1})v_{r-1} \otimes h_{(s_r,s_{r+1})}(g)\otimes 
\sigma_{r+2}(g_{r+2})v_{r+2} \dots \otimes \sigma_t(g_t)v_t\]
for $l=\diag(g_1,\dots,g_{r-1},g,g_{r+2},\dots,g_t)$ is well defined and does the job.
\end{proof}

For $w\in \S_t$, we set $w(\m)=(m_{w^{-1}(1)},\dots,m_{w^{-1}(t)})$, hence 
\[w(M):=M_{w(\m)}=wMw^{-1},\]
$w(\sigma)$ for the representation $\sigma(w^{-1}\ . \ w)$ of $w(M)$. When $\sigma$ is of the form $\sigma=\sigma_1\otimes \dots \otimes \sigma_t$ then 
\[w(\sigma)= \sigma_{w^{-1}(1)}\otimes \dots \otimes \sigma_{w^{-1}(t)}.\] 
Let $Q=P_{w(\m)}$, $L=M_{w(\m)}$ and $U=N_{w(\m)}$. For $r\in \R$, and $w\in \S_t$, we set 
\[D(w,r)=\{\s\in \C^t, \ \forall \ (i,j)\in \Inv(w),\ \Re(s_i-s_j)>r.\}\]
It is proved for example in \cite[Section 2]{Sh81}, \cite[Chapter 10]{W92}, \cite[II.1.6]{MW94}, that there is 
 $r=r_{\sigma}\in \R$, such that for $\s\in D(w,r_\sigma)$, the following integral is absolutely convergent for all 
 $f_{\s}$ in $\Ind_P^G(\sigma,\s)$, and all $g\in G$:
 \[A_\sigma(w,\s)f_{\s}(g)=\int_{wNw^{-1}\cap U\backslash U} f_{\s}(w^{-1}ng)dn.\]
 
 In both the $p$-adic and real case, a way to give a meaning to the absolute convergence of the integral above 
 is to realize (it is always possible) the space $V_\sigma$ of $\sigma$ as a dense subset 
 of a Hilbert space $\overline{V_\sigma}$, such that $\sigma$ extends to a continuous representation of $M$ on this space ($V_\sigma$ is the space of smooth vectors in $\overline{V_\sigma}$ in the $p$-adic case, and $M\cap K$-finite vectors 
 in the real case). In the $p$-adic case, it is equivalent to say that for all $v^\vee\in V_\sigma^\vee$ the integral 
 \[\int_{wNw^{-1}\cap U\backslash U} <f_{\s}(w^{-1}ng),v^\vee> dn\] is absolutely convergent, in which case $\int_{wNw^{-1}\cap U\backslash U} f_{\s}(w^{-1}ng) dn$ is the only element of $V_\sigma$ such that 
 \[<\int_{wNw^{-1}\cap U\backslash U} f_{\s}(w^{-1}ng) dn,v^\vee>=\int_{wNw^{-1}\cap U\backslash U} <f_{\s}(w^{-1}ng),v^\vee> dn\] 
 for all $v^\vee\in V_\sigma^\vee$.
We set \[D_\sigma^A(w)=D(w,r_\sigma).\] 
 
\begin{LM}\label{lemma absolute convergence p-adic}
Suppose that we are in the $p$-adic case and take $\lambda \in V_\sigma^*$ the algebraic dual of $V_\sigma$. Suppose that 
for fixed $\s\in D_\sigma^A(w)$, $g\in G$, and $f_{\s}\in \Ind_P^G(\sigma[\s])$ the integral $\int_{wNw^{-1}\cap U\backslash U} \lambda(f_{\s}(w^{-1}ng)) dn$ is absolutely convergent. Then 
\[\lambda(\int_{wNw^{-1}\cap U\backslash U} f_{\s}(w^{-1}ng) dn)=\int_{wNw^{-1}\cap U\backslash U} \lambda(f_{\s}(w^{-1}ng)) dn\] 
\end{LM} 
\begin{proof} Take a small enough open subgroup $M_0$ of $M$ with Iwahori decomposition such that 
 $g^{-1}w(M_0) g$ fixes $f_{\s}$ on the right, and $U'$ any compact open subgroup of $U$ stable under conjugation by $w(M_0)\subset L$, then for 
 $m\in M_0$:
 \[\sigma(m)(\int_{wNw^{-1}\cap U'\backslash U'} f_{\s}(w^{-1}ng)dn)=\int_{wNw^{-1}\cap U'\backslash U'} \sigma(m)(f_{\s}(w^{-1}ng))dn\]
 \[=\int_{wNw^{-1}\cap U'\backslash U'} f_{\s}(mw^{-1}ng)dn= \int_{wNw^{-1}\cap U'\backslash U'} f_{\s}(w^{-1}w(m)ng)dn\]
 \[= \int_{wNw^{-1}\cap U'\backslash U'} f_{\s}(w^{-1}nw(m)g)dn= \int_{wNw^{-1}\cap U'\backslash U'} f_{\s}(w^{-1}ng)dn.\] 
 Applying any linear form in $V_{\sigma}^\vee$ to the above equality and considering the increasing limit over such groups $U'$ we deduce that $M_0$ fixes 
 the vector $\int_{wNw^{-1}\cap U\backslash U} f_{\s}(w^{-1}ng)dn$. 
In particular if we set  
\[\lambda^\infty=\int_{M_0}\l\circ \sigma(m)dm\in V_\sigma^\vee\] for $dm$ giving volume one to $M_0$, we get  
\[\lambda^\infty(\int_{wNw^{-1}\cap U\backslash U} f_{\s}(w^{-1}ng)dn)=\lambda(\int_{wNw^{-1}\cap U\backslash U} f_{\s}(w^{-1}ng)dn).\]

Now take $(U_k)_{k\in \N}$ an increasing sequence of compact open subgroups of $U$ the union of which is equal to $U$, and 
suppose moreover that all of the subgroups $U_k$ are stable under conjugation by $w(M_0)\subset L$, then 
\[\int_{wNw^{-1}\cap U_k\backslash U_k} <f_{\s}(w^{-1}ng),\lambda^\infty>dn= \int_{wNw^{-1}\cap U_k\backslash U_k}\lambda(f_{\s}(w^{-1}ng))dn\] and both sides being absolutely convergent we deduce the equality: 
\[\lambda^\infty(\int_{wNw^{-1}\cap U\backslash U} f_{\s}(w^{-1}ng)dn)=\int_{wNw^{-1}\cap U \backslash U}\lambda(f_{\s}(w^{-1}ng))dn\] where the left hand side is equal to what we look for as already observed.
\end{proof}

 It is shown in the same references (\cite[IV.1]{MW94} for the adelic case) that if $f_{\s}$ is a flat section, then for $g\in G$, the integral $A_\sigma(w,\s)f_{\s}(g)$ 
 extends to a meromorphic function of $\s$. Whenever $A_\sigma(w,\s)$ is holomorphic at $\s_0$ (meaning that $A_\sigma(w,\s)f_{\s}(g)$ is holomorphic at $\s=\s_0$ for all $f$ and $g$), which is the case for $\s_0$ in a dense open subset of $\C^t$, then 
 \[A_\sigma(w,\s_0)f_{\s_0}\in \Ind_Q^G(w(\sigma),w(\s_0)).\]
   Moreover, for each $\s_0\in \C$, there is a nonzero polynomial $P_{\s_0}(\s)$, such 
 that $P_{\s_0}(\s)A_\sigma(w,\s)f_{\s}(g)$ is holomorphic at $\s_0$ for all $f$ and $g$. In the $p$-adic case, one can in fact choose $P\in \C[q^{-\s}]-\{0\}$, such that $P(q^{-\s})A_\sigma(w,\s)f_{\s}(g)$ belongs to 
 $\C[q^{-\s}]\otimes V_\sigma$ for all $f$ and $g$.\\
 Finally, in all cases, if $w_1$ and $w_2$ are elements of $\S_t$ such that $\ell(w_1\circ w_2)=\ell(w_1)+\ell(w_2)$, then 
 \[A_\sigma(w_1\circ w_2,\s)= A_{w_2(\sigma)}(w_1,w_2(\s))\circ A_\sigma(w_2,\s).\]

\subsection{Poles of certain $p$-adic intertwining operators}

In this section $G$ is as in Section \ref{p-adic}. We will mainly recall some results from \cite{MW89}, and explain why they hold for inner forms of $\GL(n)$ as well. We say that a cuspidal segment $\D$ precedes 
a cuspidal segment $\D'$ and we write $\D \prec \D'$, if $\D$ is of the form $[b,e]_\rho$, $\D'$ of the form $[b',e']_\rho$ for $\rho$ cuspidal, with $b'-b\in \Z$ and $b\leq b'-1\leq e \leq e'-1$. We say that $\D$ and $\D'$ are linked if either $\D$ precedes $\D'$ or $\D'$ precedes $\D$. 
We say that $\D$ and $\D'$ are juxtaposed if they are linked, and either $b'=e+1$ or $b=e'+1$. We consider $M=M_{(m_1,\dots,m_t)}$, with $t=2r$ an even number. We denote by 
$\tau_r$ the transposition $(r\ r+1)$ in $\S_t$. We consider an essentially square-integrable representation $\d_1\otimes\dots \otimes \d_t$ of $M$, such that $\d_r=\d[s_r]$ and $\d_{r+1}=\mu[s_{r+1}]$ with $\d$ and $\mu$ unitary, $s_r$ and $s_{r+1}\in \R$, and moreover 
$\d=\St_{k_1}(\rho)$ and $\mu=\St_{k_2}(\rho)$ for $\rho$ the same cuspidal representation. Let's write 
\[\sigma=\d_1\otimes\dots \d_{r-1} \otimes \d \otimes \mu \otimes \d_{r+2}\otimes \dots \otimes \d_t \]

 For \[\s=(\underbrace{0,\dots,0,s}_{r},\underbrace{s_r+s_{r+1}-s,0,\dots,0}_{r})\in \C^t,\] we set 
\[A_\sigma(\tau_r,s)=A_\sigma(\tau_r,\s).\] 
 For each $\d_i$, we recall that $\d_i=\La(\D_i)$ for a cuspidal segment $\D_i$. 

\begin{prop}\label{intertwining-poles}
Suppose that the situation is as above and set $l=l_\rho$. If $s_r-s_{r+1}< -\frac{l|k_1-k_2|}{2}$, then the standard intertwining operator $A_\sigma(\tau_r,s)$ has a (necessarily simple) pole at $s=s_r$ if and only if $\D_r\prec\D_{r+1}$ but $\D_r$ and $\D_{r+1}$ are not juxtaposed, otherwise we recall that it is automatically (holomorphic and) nonzero. If $s_r>s_{r+1}$, then $A_\sigma(\tau_r,s)$ is defined by absolutely convergent integrals at $s=s_r$ and is in particular holomorphic at this point.
\end{prop}
\begin{proof}
We use the notations of Lemma \ref{surjective restriction}. If \[R(f_{\s})=v_1\otimes \dots \otimes v_{r-1}\otimes h_{\s}\otimes v_{r+2}\otimes \dots \otimes v_t,\] then 
\[R(A_\sigma(\tau_r,s)f_{s})=v_1\otimes \dots \otimes v_{r-1}\otimes
 A_{\d[s]\otimes \mu[s_r+s_{r+1}-s]}((1\ 2),0)h_{s}\otimes v_{r+2}\otimes \dots \otimes v_t\] hence it is enough to treat the case where $t=2$ thanks to Lemma \ref{surjective restriction}. The second assertion follows from \cite[Proposition IV.2.1.]{W03}. Let's justify the first. 
 We want to use Lemma I.4 of \cite{MW89}, in the context of inner forms of $\GL(n)$. We claim that it is still valid. It is proved in \cite[Lemma 2.1]{AC89} that the normalization factors of normalized intertwining operators can be taken to be the Langlands-Shahidi normalizing factors of the Jacquet-Langlands lifts, in order for the expected properties stated in 
\cite[I.1]{MW89} to be satisfied. More precisely Properties I.1(1) follows from \cite[Lemma 2.1]{AC89}, I.1(3) follows from the references stated in \cite{MW89}, as well as from the fact that parabolically induced representations from irreducible unitary representations remain irreducible, a result which is also true for inner forms of $\GL(n)$ thanks to \cite{S09}. The proof of  I.1(2) and (4) then holds without modification (notice that the case of (2) where $\rm{Re}(s_i-s_j)>0$ is true for general reductive groups by \cite[Proposition IV.2.1.]{W03}). This implies that I.2(1), hence Lemma I.2 (ii) hold 
too. Finally, replacing the reference to Zelevinsky by the reference to \cite{T90}, the proof of Lemma I.4 in \cite{MW89} is reduced to the classical result of \cite{O74} concerning poles of standard intertwining operators between representations induced by two cuspidal ones, and this result is for all inner forms of $\GL(n)$. Now set $\tau=(1\ 2)$, we recall that $\JL(\rho)=\St_{l}(\rho')$ for $l=l_\rho$. We set 
$\d'=\JL(\d)=\St_{k_1l}(\rho')$, $\mu'=\JL(\mu)=\St_{k_2l}(\rho')$ and denote by $r_{\d\otimes \mu}(\tau,s)$ the normalizing factor of $A_{\d\otimes\mu}(\tau,s)$. 
By definition, and using the equality of the Langlands-Shahidi factors and the factors defined in \cite{JPSS83} (see 
\cite{Sh84}), we have 
\[r_{\d\otimes \mu}(\tau,s)=r_{\d'\otimes \mu'}(\tau,s).\] By a reformulation of \cite[Lemma I.4]{MW89}, the normalized intertwining operator 
\[r_{\d'\otimes \mu'}(\tau,s)^{-1}A_{\d'\otimes \mu'}(\tau,s)\] has a simple pole at $s=s_r$ if and only if 
$\D_r\prec \D_{r+1}$. By [ibid.] again, one has:
\[r_{\d'\otimes \mu'}(\tau,s)\underset{\C[q^{\pm s}]^\times}{\sim}\]
\[ L(2s-(s_r+s_{r+1})+l\frac{|k_1-k_2|}{2},\rho',(\rho')^\vee) 
L(2s-(s_r+s_{r+1})+l\frac{(k_1+k_2)}{2},\rho',(\rho')^\vee)^{-1}.\]
Hence $r_{\d'\otimes \mu'}(\tau,s)$ has no pole at $s_r$ as $s_r-s_{r+1}<-l\frac{|k_1-k_2|}{2}$, and it has a 
zero at $s_r$ if and only if $\D_r\prec \D_{r+1}$ and $\D_r$ and $\D_{r+1}$ are juxtaposed. The statement follows.
\end{proof}

\section{Intertwining periods}\label{intertwining-periods} 

The intertwining periods appear naturally in the study of the relative trace formula investigated in \cite{JLR} and \cite{LR}, and more precisely in the formula which computes the period integral of truncated Eisenstein series. We will recall results from \cite{LR} in this context. However, we notice that \cite{LR} extends the results of \cite{JLR} from the pair 
$(G_n(\A_l),G_n(\A_k))$ to general reductive Galois pairs. Here the set of double cosets $P_{\overline{m}}(\G_l)\backslash G_m(\G_l)/G_m(\G_k)$ canonically identifies with $P_{\overline{m}}(l)\backslash G_m( l)/G_m(k)$, which is the case considered in \cite{JLR}. Hence, the results of \cite[Chapter VII]{JLR} as well as their proofs hold without modification for the non split case, so we could most of the time refer directly to \cite{JLR}.\\

If $w$ is a involution in $\S_t$, we set $\C^t(w,-1)=\{\s\in \C^t,\ w(\s)=-\s\}.$ 

\subsection{Global open intertwining periods}\label{global-periods}

Here $G$ is as in Section \ref{automorphic}, we take $t=2r$ with $r\in \N-\{0\}$, and $\m=(m_1,\dots,m_t)$ a self-dual partition of $m$. We set $P=P_{\m}=P_{(m_1,\dots,m_t)}=MN$ as before. For $S$ a subgroup of $G$, we will write 
$S(l)$ for $S\cap G(l)$.\\

We set \[u=\begin{pmatrix} I_{m_1} &  &  & & & & & -\d_{l/k} I_{m_1} \\
  & I_{m_2} &  & & &  & -\d_{l/k} I_{m_2} &  \\
  &  &  \ddots  & & &  \iddots &  &  \\
   &  &   & I_{m_r} & -\d_{l/k} I_{m_{r}} &    &  &  \\ 
   &  &   & I_{m_r} & \d_{l/k} I_{m_{r}} &    &  &  \\    
  &  &  \iddots  & & &  \ddots &  &  \\  
  & I_{m_2} &  &  &  &  & \d_{l/k} I_{m_2}  &  \\ 
I_{m_1}  &  &   &   &  &  &  & \d_{l/k} I_{m_1} \end{pmatrix}.\]

It is the representative $u\in R(P_{\overline{m}}(\G_l)\backslash G_m(\G_l)/G_m(\G_k))$ such that $u\theta(u)^{-1}$ is 
the longest Weyl element $w=w_t$ of $\S_t$ (associated to the partition $\m$). It is 
$P$-admissible. For any subgroup $S$ of $G$, we denote by 
$S^u$ the subgroup of fixed points of the involution $g\mapsto w\theta(g)w^{-1}$ of $G$ in 
$S$, and we denote by $S(u)$ the group $u^{-1}S^u u$, in particular $S(u)\subset H=G^\theta$.\\

 Let $\sigma$ be a cuspidal automorphic representation of $M$, the space of which is included in $L^2(A_M M(l)\backslash M)$, i.e. $\sigma$ is unitary, and its central character is trivial on $A_M$. We moreover suppose that 
 \[\sigma=\sigma_1\otimes \dots \otimes \sigma_t\] with $\sigma_{t+1-i}^\theta=\sigma_i^\vee$ for all $i$. To $f$ in $\Ind_P^G(\sigma)$ and 
 $\s\in \C^t(w,-1)$, one can attach (\cite[Chapter VII]{JLR} where the parabolic induction is not normalized, or \cite[Definition 5.1]{LR}) 
 the intertwining period (note that $P(u)=M(u)$):

\[J_\sigma(w,f_{\s},\s)=\int_{P(u)\backslash H}\int_{A_M^uM(l)^u\backslash M^u}f_{\s}(muh)dm dh.\]

Notice that our definition is the same as that of \cite{JLR}, as $\eta_{\s}$ is left invariant under $M^u$, but the notation slightly differs as we write $J_\sigma(w,f_{\s},\s)$ instead of $J(w,f,\s)$.\\

By \cite[Theorem 31]{JLR}, \cite[Proposition 5.2.1 and Theorem 10.2.1]{LR}, there is $q_\sigma\in \R$ such that 
all such integrals are absolutely convergent for \[\s\in D_\sigma^J(w)=\C^t(w,-1)\cap D(w,q_\sigma)\] 
and extends to a meromorphic function on $\C^t(w,-1)$. Moreover, by [ibid.], it satisfies a series of functional equations with respect to standard intertwining operators associated to certain Weyl elements. Here we will 
only be interested with the Weyl element $\tau_r=(r \ r+1)\in \S_t$.
Clearly $w$ commutes with $\tau_r$, hence, in order to state the functional equation of the intertwining periods, we only need to check that 
$\tau_r$ belongs to the set denoted by $W(w,w)$ in Section 3.4 of \cite{LR} ($\Omega(w,w)$ in \cite[Chapter VII]{JLR}). 
But noticing that there is no positive root of the center of $M$ which is fixed by $w$, the first part of \cite[Proposition 3.4.1]{LR} tells us that $\tau_r$ indeed belongs to $W(w,w)$. Then, for $\s \in \C^t(w,-1),$ according to \cite[Theorem 10.2.1]{LR}, the intertwining period satisfies the functional equation:

\begin{equation}\label{global equation} J_\sigma(w,A_\sigma(\tau_r,\s)f_{\s},\tau_r({\s}))=J_\sigma(w,f_{\s},\s).\end{equation}

\subsection{Local open periods and their functional equation}\label{open-periods}

\subsubsection{The non archimedean non-split case}\label{non-archi non-split}
Here $G=G_m$ and the other notations are as in Section \ref{p-adic}. We take $t=2r$ with $r\in \N-\{0\}$, and $\m=(m_1,\dots,m_t)$ a self-dual partition of $m$. We set $P=P_{\m}=P_{(m_1,\dots,m_t)}=MN$. The element $u$ is as in Section \ref{global-periods}, but with $\d_{E/F}$ instead of $\d_{l/k}$, and for $S$ a subgroup of $G$, the definitions of $S^u$ and $S(u)$ are the same. Moreover, in this situation, for $w=u\theta(u)^{-1}$, the 
involution $\theta^w$ sends $P$ to $P^-={}^t\!P$. In particular the double coset $PuH$ is open in $G$.\\

For each $i$, let $\sigma_i$ be a finite length representation of $G_{m_i}$, and assume that $\sigma_{t+1-i}^\theta=\sigma_i^\vee$ for all $i$. We set 
$\sigma=\sigma_1\otimes \dots \otimes \sigma_t$ as usual. The representation $\sigma$ is $M^u$-distinguished, and we denote by $L$ the nonzero $M_u$-invariant linear form on $V_\sigma$ defined by 
\[L:v_1\otimes \dots \otimes v_t\mapsto \prod_{i=1}^r <v_i,v_{t+1-i}>\] where $<\ , \ >$ is the natural pairing between a representation and its contragredient.
More generally for $\s=(s_1,\dots,s_t)\in \C^n(w,-1),$ the linear form $L$ belongs to 
$\Hom_{M^u}(\sigma[\s],\C)$.  For $f_{\s}$ a flat section in the space of $\pi_{\s}=\Ind_P^G(\sigma[\s])$, 
we define: 
\[J_\sigma(w, f_{\s},\s, L)=\int_{P(u)\backslash H} L(f_{\s}(uh)) dh.\]

This is the situation studied in \cite{BD08}. Here we summarize how their results apply in our particular situation. The following theorem follows from \cite[Theorems 2.8 and 2.16]{BD08}, and the fact that the condition
on $\eta$ in Theorem 2.16 of [ibid.] is always satisfied by Theorem 4(i) of \cite{L08}. 

\begin{thm}\label{BD}
There is $q_\sigma\in \R$, such that the integral $J_\sigma(w, f_{\s},\s, L)$ is absolutely convergent for 
\[s\in D_\sigma^{J,L}(w)=\C^t(w,-1)\cap D(w,q_\sigma).\] 
For $\s\in D_\sigma^{J,L}(w)$, the map $f_{\s}\mapsto J_\sigma(w, f_{\s},\s, L)$ defines 
a nonzero $H$-invariant linear form on $V_{\pi_{\s}}$. Moreover, there is a nonzero Laurent polynomial $P$ such that 
$P(q^{\pm {\s}})J_\sigma(w, f_{\s},\s, L)$ belongs to $\C[q^{\pm {\s}}]$ for all $f$.
\end{thm}

\begin{rem}\label{nonvanishing-of-global-periods}
Notice that for some well chosen $f=f_{\underline{0}}$ in the subspace $\mathcal{C}_c^\infty(P\backslash PuH,\d_P^{1/2}\sigma) \subset \pi=\pi_{\underline{0}}$, the local intertwining/open period $J_\sigma(w, f_{\s},\s, L)$ is identically equal to $1$. 
\end{rem}

We will say that $J_\sigma(w, . ,\s, L)$ has a singularity (a pole if $\s$ consists of one complex variable) at $\s=\s_0$ if for some flat section $f_{\s}$, the function $J_\sigma(w,f_{\s},\s, L)$ has a singularity at $\s=\s_0$. Otherwise we will say that 
$J_\sigma(w,.,\s,L)$ is holomorphic or regular at $\s=\s_0$. By Remark \ref{nonvanishing-of-global-periods}, if 
$J_\sigma(w,.,\s,L)$ is holomorphic or regular at $\s=\s_0$, then $J_\sigma(w,.,\s_0,L)$ is nonzero. \\

We now suppose that $\sigma$ is irreducible. By Theorem \ref{mult1}, and because $\pi_{\s}$ is irreducible for all $\s$ in a dense open subset of $\C^t(w,-1)$, we deduce the following functional equation.

\begin{prop}\label{non-archi-functional-equation} If $\sigma$ is irreducible, there exists an element $\alpha_\sigma$ (depending on $w$, $\tau_r$ and $L$...) of $\C(q^{-\s})$, such that for any flat section $f_{\s}$ in $\pi_{\s}$:
\[J_{\tau_r(\sigma)}(w,A_\sigma(\tau_r,\s)f_{\s},\tau_r(\s),L)=\alpha_{\sigma}(\s)J_{\sigma}(w,f_{\s},\s,L).\]
\end{prop}

\subsubsection{The non-Archimedean split case}\label{non-archi-split}

As above, the reference for the results of this paragraph can be chosen to be \cite{BD08}, but the results also follow 
from the usual properties of intertwining operators. The field $F$ is $p$-adic and 
$E=F\times F$. The involution $\theta$ of $E$ is just $\theta:(x,y)\mapsto (y,x)$. The group 
$G$ is $G_m(\G_E)=G_m(\G_F)\times G_m(\G_F)$, and $H=G^\theta$ is just $G_m(\G_F)$ embedded diagonally inside 
$G$. We take $t=2r$ with $r\in \N-\{0\}$, and $\m=(m_1,\dots,m_t)$ a self-dual partition of $m$. We set 
\[P=P_{\m}(\G_F)\times P_{\m}(\G_F)=MN.\] We consider $\k_i$ a finite length representation of $G_{m_i}(\G_F)$ 
for all $i$ from $1$ to $t$, and set 
\[\kappa=\k_1\otimes \dots \otimes \k_t.\] We denote by $w$ the element $w_t$ of $\S_t$, then set \[w(\k)^\vee=\k_t^\vee \otimes \dots \otimes \k_1^\vee,\] 
and define the representation $\sigma$ of $M=M_{\m}(\G_F)\times M_{\m}(\G_F)$ to be the tensor product 
\[\sigma=\k\otimes w(\k)^\vee.\]

We then set $u=(I_n,w)\in G$, hence $u\theta(u)^{-1}=(w,w)$, or just $w$ seen as an element inside $H$. The definitions of 
$\theta^w$, $S^u$ and $S(u)$ for $S$ a subgroup of $G$ are as before. For $\s\in \C^t$, we set  
\[\sigma[\s]= \k[\s]\otimes w(\k[\s])^\vee,\] and 
\[\pi_{\s}=Ind_P^G(\sigma[\s])=Ind_{P_{\m}(\G_F)}^{G_m(\G_F)}
(\k[\s])\otimes Ind_{P_{\m}(\G_F)}^{G_m(\G_F)}(w(\k[\s])^\vee).\]
We define the linear form 
$L\in \Hom_{M^u}(\sigma_{\s},\C)-\{0\}$ by the formula:
\[L:(\otimes_{i=1}^t v_i,\otimes_{i=1}^t w_i)\mapsto \prod_{i=1}^t <v_i,w_{t+1-i}>.\]

With these notations the definition of the intertwining period $J_\sigma(w,f_{\s},\s,L)$ is the same as above for 
$f_{\s}$ a flat section of $\pi_{\s}$. In fact, if $f_{\s}$ is a pure tensor $h_{\s}\otimes g_{\s}$, the intertwining period
is related to the standard operator associated to $w$ by the following identity:

\begin{equation}\label{split period vs intertwining} J_\sigma(w,h_{\s}\otimes g_{\s},\s,L)=<A_{\sigma}(w,\s)h_{\s},g_{\s}>,\end{equation} where 
$<\ , \ >$ is the natural pairing between a representation and its contragredient.

The immediate analogues of Theorems \ref{BD}, Remark \ref{nonvanishing-of-global-periods}, and Proposition 
\ref{non-archi-functional-equation} are then valid, and we don't state them (in fact we state their exact analogues in the archimedean split case hereunder), they either follow from \cite{BD08}, or from the usual properties of standard intertwining operators.

\subsubsection{The Archimedean split case}\label{archi-split}
Now we recall the similar results  in the archimedean situation. They can either be deduced from \cite{BD92} and \cite{CD94} or from 
the usual properties of standard intertwining operators in the archimedean situation. The only reason why we need to consider archimedean intertwining periods being for our local-global application, we thus restrict to the split case (in the sense that the group and the quadratic algebra are split). Hence $\K=\R$ 
or $\C$ and we consider the involution of $\K\times \K$ given by $\theta(x,y)=(y,x)$. We set $G=G_n(\K\times \K)$, and define $H$, and $P=MU$ (associated to  self dual partition $\n=(n_1,\dots,n_t)$ of $n$, 
with $t=2r$ an even positive integer), $u$, $w$, $S^u$ and $S(u)$ for 
$S$ a closed subgroup of $G$ as in Section \ref{non-archi-split}. 
We consider for each $i$ from $1$ to $t$ a Harish-Chandra module $\k_i$ of $G_{n_i}(\K)$, and set 
\[\k= \k_1\otimes \dots \otimes \k_t.\]
For $\s\in \C^t$, we define $\k[\s]$ and $w(\k[\s])^\vee$ as in Section \ref{non-archi-split}, and set 
\[\sigma[\s]=\k[\s]\otimes w(\k[\s])^\vee\] together with its 
Casselman-Wallach completion $\overline{\sigma[\s]}^\infty.$ We define 
$L\in \Hom_{M^u}(\sigma[\s],\C)-\{0\}$ as in Section \ref{non-archi-split}. We then set 
\[\overline{\pi_{\s}}^\infty=\Ind_P^G(\overline{\sigma[\s]}^\infty),\] and 
$\pi_{\s}$ the Harish-Chandra subspace of $K$-finite vectors in $\overline{\pi_{\s}}^\infty$. For $f_{\s}$ a flat section in the normalized smooth induced representation $\overline{\pi_{\s}}^\infty$, we 
set \[J_\sigma(w,f_{\s},\s,L)=\int_{P(u)\backslash H} L(f_{\s}(uh))dh.\] 
Because of Equation (\ref{split period vs intertwining}), the following result can be seen either as a consequence of the basic properties of standard intertwining operators, 
or of \cite[Theorem 3]{CD94}. Its statement will be enough for our purpose. 

\begin{thm}
There is $q_\sigma\in \R$, such that the integral $J_\sigma(w, f_{\s},\s,L)$ is absolutely convergent for 
\[\s\in D_\sigma^{J,L}(w)= D(w,q_\sigma).\] 
For $\s\in D_\sigma^{J,L}(w)$, the map $f_{\s}\mapsto J_\sigma(w, f_{\s},\s,L)$ defines 
a nonzero $H$-invariant linear form on $V_{\overline{\pi_{\s}}^\infty}$, which is moreover continuous.
\end{thm}

As the representation $\pi_{\s}$ is dense in $\overline{\pi_{\s}}^\infty$, because it is by definition the subspace of $K$-finite vectors in there, we deduce the following corollary.

\begin{cor}\label{archimedean non vanishing}
There is $f\in \pi=\pi_0$, such that the archimedean intertwining period $J_\sigma(w, f_{\s},\s,L)$ is nonzero.
\end{cor}

  We suppose that $\sigma$ is irreducible, hence $\overline{\pi_{\s}}^\infty$ is irreducible for all $\s$ in a dense  
  open subset of $\C^t(w,-1)$, and so by Schur's lemma, there is up to scaling a unique nonzero $H$-invariant continuous linear form on $\overline{\pi_{\s}}^\infty$. This implies the following functional equation.

\begin{prop}\label{archi-functional}
If $\sigma$ is irreducible, there is a meromorphic function $\alpha_\sigma$ on $\C^t$, such that for all $f\in \overline{\pi}^\infty$, one has the equality:
\[J_{\tau_r(\sigma)}(w,A_\sigma(\tau_r,\s)f_{\s},\tau_r(\s),L)=\alpha_\sigma(\s)J_\sigma(w,f_{\s},\s,L).\]
\end{prop}

In fact, when we moreover suppose that $\sigma$ is generic unitary, which is the case when it is the local component of a cuspidal  automorphic representation, this functional equation can be obtained by a direct computation, which moreover gives a formula for $\alpha_\sigma(\s)$ that we shall use in our local-global application. 

\begin{prop}\label{archimdean-alpha}
Suppose that $\sigma$ is generic unitary, and that $\psi$ is a non trivial character of $\K$, then \[\alpha_\sigma(\s)\sim\gamma(2s_r,\k_r,\k_{r+1}^\vee,\psi)^{-1}\gamma(-2s_r,\k_r^\vee,\k_{r+1},\psi)^{-1}.\]
\end{prop}
\begin{proof}
Take $f=h\otimes g\in \pi$, then it is shown in \cite[Section 4]{BD92} that 
\[J_\sigma(w,f_{\s},{\s},L)=<A_\k(w,\s)h_{\s},g_{\s}>,\] where \[<u,v>=\int_K <u(k),v(k)>dk,\] and 
\[<v_1\otimes \dots \otimes v_t,v_1^\vee\otimes \dots \otimes v_t^\vee>
=\prod_{i=1}^t<v_i,v_i^\vee>.\]
Set $w''=w\circ (\tau_r)^{-1}$, so that $\ell(w)=\ell(\tau_r)+\ell(w'')$. We first suppose that $\sigma$ (equivalently $\k$) is tempered. We have  
\[J_{\tau_r(\sigma)}(w,A_\k(\tau_r,\s)f_{\s},\tau_r(\s))=<A_{\tau_r(\k)}(w,\tau_r(\s))A_\k(\tau_r,\s)h_{\s},
A_{w(\k)^\vee}(\tau_r,-\s)g_{\s}>\]
\[=<A_\k(w'',\s)A_{\tau_r(\k)}(\tau_r,\tau_r(\s))A_\k(\tau_r,\s)h_{\s},A_{w(\k)^\vee}(\tau_r,-\s)g_{\s}>\]
\[\sim\gamma(2s_r,\k_r,\k_{r+1}^\vee,\psi)^{-1}\gamma(-2s_r,\k_{r+1},\k_r^\vee,\psi)^{-1}
<A_\k(w'',\s)h_{\s},A_{w(\k)^\vee}(\tau_r,-\s)g_{\s}>,\] 
where the relation \[A_{\tau_r(\k)}(\tau_r,\tau_r(\s))A_\k(\tau_r,\s)\sim\gamma(2s_r,\k_r,\k_{r+1}^\vee,\psi)^{-1}
\gamma(-2s_r,\k_{r+1},\k_r^\vee,\psi)^{-1}Id\] follows from \cite[Section 3]{A89} which applies here because 
the representation $\pi$ is tempered. Finally 
\[<A_\k(w'',\s)h_{\s},A_{w(\k)^\vee}(\tau_r,-\s)g_{\s}>=<A_{w''(\k)}(\tau_r,w''(\s))A_\k(w'',\s)h_{\s},g_{\s}>\] by 
\cite[(J4)]{A89}, which is in turn equal to $<A_\k(w,\s)h_{\s},g_{\s}>$, and this ends the computation when $\k$ is tempered, or more generally when $\k[\s]$ is tempered for some $\s$. In general one writes 
\[\kappa_i=\Ind_{P_i}^{G_{n_i}(\K)}(\mu_i[\u_i])\] with $P_i$ a standard parabolic subgroup of $G_{n_i}(\K)$, $\mu_i$ a tempered representation of its standard Levi subgroup $L_i$ and 
$\u_i$ an element of $\C^{a_i}$ for some appropriate $a_i\in \N^*$, so $\k_i=\k'_{i,\u_i}$ with $\k'_i=\Ind_{P_i}^{G_{n_i}(\K)}(\mu_i)$. We set 
\[\mu=\mu_1\otimes \dots \otimes \mu_t,\] it is a representation of $L=L_1\times \dots \times L_t$ which is itself the standard Levi subgroup of a standard parabolic subgroup $Q$ of $G_n(\K)$. 
The inclusion $Q\subset P_{(n_1,\dots,n_t)}(\K)$ gives a natural inclusion of $\S_t$ inside $\S_{t'}$, where $t'$ is the number of $\GL$-blocks of $L$. For any $w\in \S_t$, we denote by $P(w)$ the standard parabolic subgroup of 
$G_n(\K)$ with standard Levi subgroup $w(M_{(n_1,\dots,n_t)}(\K))$ and by $Q(w)$ that with standard Levi-subgroup $w(L)$. 
We denote by $f_{w(\u)}\mapsto F_{f_{w(\u)}}^w$ the inverse of the canonical isomorphism $F_{w(\u)}\mapsto F_{w(\u)}(I_{w^{-1}(n_1)},\dots,I_{w^{-1}(n_t)})$ from 
$\Ind_{P(w)}^{G_n(\K)}(w(\k'_{\u}))$ to $\Ind_{Q(w)}^{G_n(\K)}(w(\mu[\u]))$.
Seeing $\C^t$ as a subspace of $\C^{t'}$ (again thanks to the inclusion $L\subset M$), one checks with arguments similar to those in the proof of Proposition \ref{proposition compatibility of transitivity of parabolic induction and intertwing periods} that for any $w_1$ and $w_2$ in $\S_t$, one has 
\[[A_{w_1(\kappa'_{\u})}(w_2,w_1(\s))(F_{f_{w_1(\u)}}^{w_1})_{w_1(\s)}](I_{(w_2w_1)^{-1}(n_1)},\dots,I_{(w_2w_1)^{-1}(n_t)})=A_{w_1(\mu)}(w_2,w_1(\s+\u))f_{w_1(\s+\u)}.\]
This together with the multiplicativity relation of $\gamma$-factors (see \cite{J09}) implies that we could suppose that $\kappa[\s]$ is tempered for some $\s$ in 
the proof of the proposition (replacing $\k$ by $\mu[\u]$). 
\end{proof}

\subsection{Convergence and meromorphic continuation of non-Archimedean intertwining periods attached to admissible orbits}\label{Section p-adic intertwining periods properties}\label{section admissible intertwining periods}

The notations are again as in Section \ref{p-adic}. In this section we consider intertwining periods attached to admissible orbits which are not 
necessarily open in $G$. We first prove a non archimedean version of the very useful 
\cite[Proposition 33]{JLR}, the proof of which we follow.

\begin{prop}\label{proposition JLR key prop p-adic analogue}
Let $P=P_{(m_1,\dots,m_t)}=MU$ be a standard parabolic subgroup of $G$ and 
let $\sigma$ be an admissible representation of $M$. Take a $P$-admissible
$u\in \Rep$, set $w=w_u$, and suppose moreover that $\lambda\in \Hom_{M^{\theta_w}}(\sigma,\C)-\{0\}$ (in particular 
we assume that $\sigma$ is $M^{\theta_w}$-distinguished). Let $\tau=(i,i+1)$ be a transposition in $\S_t$ such that 
$w(i)>w(i+1)$ but $(w(i),w(i+1))\neq (i+1,i)$, set $w'=\tau w \tau^{-1}$, $u'=\tau u$ so that 
$u'u'^{-\theta}=w'$, $M'=\tau(M)$ and let $P'=M'U'$ be the standard parabolic subgroup of $G$ with Levi subgroup $M'$, so that 
$u'$ is a $P'$-admissible element of $R(P'\backslash G/H)$. 
Choose $r\in \R$ large enough such that for $\s\in D(\tau,r)$, the operator $A_{\sigma}(\tau,\s)$ is given by absolutely convergent integrals. If for fixed $\s\in \C(w,-1)\cap D(\tau,r)$ and 
$f_{\s} \in \Ind_P^G(\sigma[\s])$ the left hand side of the following equation is absolutely convergent, then so is the right hand side, in which case the equation is satisfied:
\[\int_{P(u)\backslash H} \lambda(f_{\s}(uh))dh=\int_{P'(u')\backslash H} \lambda(A_{\sigma}(\tau,\s)f_{\s}(u'h))dh\]
\end{prop} 
\begin{proof}
We denote by $\mathfrak{U'}_\alpha$ the subspace of the Lie algebra of $U'$ on which the center of $M'$ acts by the additively written character $\alpha=e_i-e_{i+1}$. The reader should not be misled by the notation $\mathfrak{U'}_\alpha$ and observe that this vector space might be of dimension greater than $1$. We set $U'_\alpha=I_n+\mathfrak{U'}_\alpha$, and denote by $R=L.V$ the standard parabolic subgroup of $G$ generated by 
$P'$ and ${}^t(U'_\alpha)=U'_{-\alpha}$ (the image of $U'_\alpha$ under the transpose map). We introduce $p:U'^{u'}\rightarrow L$, which to an element in $U'^{u'}\subset U'\subset R$, associates its $L$-part. By definition of $R$, one has $U'=U'_\alpha.V$, with $U'_\alpha\subset L$, hence $p$ is in fact a morphism from 
$U'^{u'}$ to $U'_\alpha$. Clearly, the kernel of $p$ is equal to $V^{u'}=V\cap U'^{u'}$.  Moreover we have $w'(U'_\alpha)=U'_{\tau(-w(\alpha))}\subset U'$ because $\alpha$ is the only positive root sent to a negative one by 
$\tau$, but $w(\alpha)\neq -\alpha$ by assumption. Now the map $i:x\mapsto I_n+x+\theta_{w'}(x)$ sends $\mathfrak{U'}_\alpha$ to $U'^{u'}$ and satisfies $p(i(x))=I_n+x$. Thus $p$ induces a continuous isomorphism still denoted $p$ from 
$V^{u'}\backslash U'^{u'}$ to $U'_\alpha$. By Lemma \ref{lemma absolute convergence p-adic}, one has (up to the convergence issue, which will follow from Fubini's theorem at the end of the proof) 
\[\int_{P'(u')\backslash H} \lambda(A(\tau,\s)f_{\s}(u' h))dh= 
\int_{h\in P'(u')\backslash H} \int_{n\in U'_{\alpha}}\lambda(f_{\s}(\tau^{-1}nu' h))dndh.\]
Notice that for fixed $g$, the map $n\in U'\mapsto \lambda(f_{\s}(\tau^{-1}ng))$ is left $V$-invariant, hence 
the above integral equals 
\[\int_{h\in P'(u')\backslash H} \int_{n\in V^{u'}\backslash U'^{u'}}\lambda(f_{\s}(\tau^{-1}nu' h))dndh= \int_{P'(u')\backslash H} \int_{V(u')\backslash U'(u')}\lambda(f_{\s}(\tau^{-1}u' nh))dndh\]
\begin{equation}\label{equation reloux} =
\int_{P'(u')\backslash H} \int_{V(u')\backslash U'(u')}\lambda(f_{\s}(u nh))dndh.\end{equation}

Before finishing our computation we take a break to show that $P(u)=M'(u')V(u')$ or what amounts to the same:
\[\tau P^u \tau^{-1}=M'^{u'}V^{u'}.\]
   Because $u$ is $P$-admissible we have 
 \[\tau P^{u} \tau^{-1}=(\tau M^u \tau^{-1})(\tau
 U^u \tau^{-1}).\] Now notice that for any subgroup $S$ of $G$ normalized by $\tau$ (for example $V$), one has 
 the equality $S^{u'}=\tau S^u \tau^{-1}$.
 As $U^u$ is the set of elements $n$ in $U$ such that 
 $\theta(n)=w^{-1} nw $, one has $U^u\subset U\cap w^{-1} Uw\subset V$ (the last inclusion because $w(\alpha)<0$) hence  
 \[\tau U^u \tau^{-1}= \tau( U^u\cap V) \tau^{-1}= 
\tau( U\cap G^u\cap V) \tau^{-1}=\tau(V^u) \tau^{-1}=V^{u'}.\] Finally 
\[m'\in \tau M^u \tau^{-1}\Leftrightarrow \tau^{-1}m'\tau \in M^u\Leftrightarrow \theta_w(\tau^{-1}m'\tau)=\tau^{-1}m'\tau
\Leftrightarrow w\theta(\tau^{-1}m'\tau)w^{-1}=\tau^{-1}m'\tau\]
\[\Leftrightarrow w\tau^{-1}\theta(m')\tau w^{-1}=\tau^{-1}m'\tau\Leftrightarrow  \theta_{w'}(m')=m',\] hence 
$\tau M^u \tau^{-1}=M'^{u'}$. So we can now use the equality $P(u)=M'(u')V(u')$. On the other hand as $u'$ is $P'$-admissible, one has $P'(u')=M'(u')U'(u')$. Now the map $n\in P'(u')\mapsto \lambda(f_{\s}(u nh))$ is left $\d_{P(u)}$-equivariant for all $h\in H$, hence the inner integral in Equation (\ref{equation reloux}) is equal to \[\int_{P(u)\backslash P'(u')}\d_{P'(u')}^{-1}(n)\lambda(f_{\s}(u nh))dn\]
and this finishes the proof by integration in stages.
\end{proof}

Let $P=MU$ be as in the statement of Proposition \ref{proposition JLR key prop p-adic analogue}. We recall that for $w\in \S_t$ and $x\in M$, we have defined $w(x)\in w(M)$. Following \cite{JLR} with slightly different notations, we set 
\[\Omega_2(M)=\{w\in \S_t,\ w^2=Id,\ w(M)=M\}.\]  In particular we have a bijection $u\mapsto uu^{-\theta}$ between $P$-admissible 
elements of $\Rep$ and $\Omega_2(M)$. For 
$w\in \Omega_2(M)$, we set $\ell_M(w)$ to be the number of positive roots of the center of $M$ acting on the 
Lie algebra of $U$ sent to negative roots by $w$, in particular seeing $w$ as an element of $\S_t$ one has 
$\ell_M(w)=\ell(w)$ where $\ell$ is the usual length counting the number of inversions. 

\begin{df}\label{definition maximal involution}
For $w$ an involution in $\S_t$, we say that $w$ is maximal if and only there is no simple root $\alpha\in \R^t$ such that  $w(\alpha)$ is positive but different from $\alpha$. 
\end{df}

Maximal involutions are easily characterized.

\begin{LM}\label{lemma characterization max involutions}
An involution $w\in \S_t$ is conjugate to a unique maximal involution. An involution is maximal if and only if it is the 
identity or of the form $(1,t)\circ (2,t-1)\circ \dots \circ (i,t+1-i)$ for $1\leq i \leq \lfloor t/2 \rfloor$. 
\end{LM}
\begin{proof}
It is clear that the involutions of the statement are maximal. Conversely, suppose that $w$ is a maximal involution in $\S_t$ different 
from $Id$. Let $i_0$ be the smallest integer such that $w(i_0)\neq i_0$, hence $w(i_0)>i_0$. 
If $1<i_0$ then considering the simple root $e_{i_0-1}-e_{i_0}$ gives a contradiction on the maximality of $w$, hence $i_0=1$. Now set $i=w(1)$, if $i$ 
was not equal to $t$, then considering the simple root $e_i-e_{i+1}$ would give a contradiction on the maximality of $w$ so $w=(1,t)\circ w'$ where $w'$ is an involution of $\{2,\dots,t-1\}$. Identifying the permutation group of $\{2,\dots,t-1\}$ with $\S_{t-2}$ by translating the indices by $-1$, an immediate induction shows that $w$ is of the expected form. This also shows that each conjugacy class of involution contains a unique maximal involution (an involution being a possibly empty product of disjoint transpositions). 
\end{proof}

Suppose that $w\in \S_t$ is an involution which is not maximal, then there is a simple root $\alpha$ such that $w(\alpha)$ is positive but different from $\alpha$. 
Then $w'=s_\alpha w s_\alpha$ where $s_\alpha=(i,i+1)$ if $\alpha=e_i-e_{i+1}$ is an involution such that $w'(\alpha)<0$ but different from $-\alpha$. By \cite[Lemma 28]{JLR} we deduce that:

\begin{LM}\label{lemma JLR reduction to minimal involutions} If $w\in \Omega_2(M)$ is not maximal, let $\alpha$ be a simple root such that $w(\alpha)>0$ but different 
from $\alpha$. Set $\tau=s_\alpha$, $w'=\tau w \tau^{-1}$ and $M'=\tau(M)$ so that $w'\in \Omega_2(M')$, then one has \[\ell_{M'}(w')=\ell_M(w)+2.\] In particular maximal involutions of 
$\S_t$ are those of maximal length in their conjugacy class.
\end{LM}

We fix a norm $||\ . \ ||$ on $G$ by embedding it via an algebraic monomorphism $\tau$ into $\GL(2md,F)$, namely we set 
$||g||=\max_{i,j}(|\tau(g)_{i,j}|,|\tau(g^{-1})_{i,j}|)$. We moreover request that $\tau$ sends $K$ inside $\GL_{2md}(O_F)$, and that if $P=MV$ is a standard Levi subgroup of $G$, then $\tau(P)$ lives inside a standard Levi subgroup $P'=M'V'$ of $\GL_{2md}(F)$ and that $\tau(M)\subset M'$ and $\tau(V)\subset V'$. It is well-known (see \cite[I,1]{W03}) that 
$||\ . \ ||$ is bi-K-invariant. Take $u$ a $P$-admissible element in $\Rep$, and $w=uu^{-\theta}$. Following \cite[Section 3]{L08}, we set 
$||M^{\theta_w} m||_{M^{\theta_w}\backslash M, \mathrm{Lag}}=||\theta_w(m)^{-1}m||$ on $M^{\theta_w}\backslash M$ whenever $M$ is a standard Levi subgroup of $G$. We recall that according to \cite[Proposition 7]{L08}, the norm $||\ . \ ||_{M^{\theta_w}\backslash M, \mathrm{Lag}}$ is equivalent to $||\ . \ ||_{M^{\theta_w}\backslash M, \mathrm{BD}}$ where 
$||\ . \ ||_{M^{\theta_w}\backslash M, \mathrm{BD}}$ is the norm on $M^{\theta_w}\backslash M$ defined in \cite[(2.26)]{BD08}. Note that $||\ . \ ||_{M^{\theta_w}\backslash M, \mathrm{Lag}}$ 
is right-$M\cap K$-invariant (both $\theta$ and $w$ stabilize $K$). The following two results will be used as steps 
in the proof of Theorem \ref{theorem convergence and meromorphy of admissible periods}.

\begin{LM}\label{lemma bounded coeff}
Let $M=M_{(m_1,\dots,m_t)}$ be a self-dual standard Levi subgroup of $G$, contained in the standard parabolic subgroup $P=MV$.  Set $L=M_{(m_1,\dots,m_{i_0},m',m_{t+1-i_0},\dots,m_{t})}$ where 
$m'=m_{i_0+1}+\dots+ m_{t-i_0}$, it is a self-dual standard Levi subgroup of $G$ contained in the standard parabolic subgroup 
$Q=LU$. Let $u\in \mathrm{R}(Q\backslash G/ H)$ be the element such that $QuH$ is open and set $w=uu^{-\theta}$. Let $\sigma$ be a finite length representation of $M$ and take $\lambda \in \Hom_{M^{\theta_w}}(\sigma,\C)$. Then there is $b>0$, such that for any $f\in \Ind_{P\cap L}^L(\sigma)$, there is a positive constant $C_f$ such that 
\[\int_{(P\cap L)^{\theta_w}\backslash L^{\theta_w}} |\lambda(f(hx))| dh\leq C_f||L^{\theta_w} x||_{ L^{\theta_w}\backslash L,\mathrm{Lag}}^b\] for all $x\in L$, and where $(P\cap L)^{\theta_w}\backslash L^{\theta_w}$ is compact.
\end{LM}
\begin{proof}
Note that if the absolute value was around the integral and not the integrand it would be an immediate consequence of \cite[Theorem 4, (i)]{L08}. In fact it is a consequence of Lagier's result. We denote by $\phi$ the spherical function in $\Ind_{P\cap L}^L(\1)$ equal to $1$ on 
$L\cap K$. For $h\in L^{\theta_w}$ and $x\in L$, write the Iwasawa decomposition $hx=v_{hx}m_{hx}k_{hx}$ with $v_{hx}\in V\cap L$, $m_{hx}\in M$ and 
$k_{h_x}\in K\cap L$. Then 
$|\lambda(f(hx))|=\phi(hx)|\lambda(\sigma(m_{hx})f(k_{hx}))|$ and $f(k_{hx})$ takes only a finite number of values $v_1,\dots,v_r$ in $V_\sigma$ when $h$ and $x$ vary. Hence $|\lambda(f(hx))|\leq \phi(hx)\max_{i=1,\dots,r} |\lambda(\sigma(m_{hx})v_i)|$. However by 
\cite[Theorem 4, (i)]{L08} there are positive constants $a, C_1,\dots, C_r$ such that \[|\lambda(\sigma(m_{hx})v_i)|\leq C_i||M^{\theta_w} m_{hx}||_{M^{\theta_w} \backslash M,\mathrm{Lag}}^a\] independently of $h$ and $x$. Now we observe that $\theta_w$ stabilizes $M$ and moreover that 
\[||M^{\theta_w} m_{hx}||_{M^{\theta_w}\backslash M,\mathrm{Lag}}\leq ||L^{\theta_w} x||_{L^{\theta_w}\backslash L,\mathrm{Lag}}.\] 
Indeed \[||L^{\theta_w} x||_{L^{\theta_w}\backslash L,\mathrm{Lag}}=
||L^{\theta_w} hx||_{L^{\theta_w}\backslash L,\mathrm{Lag}}=||L^{\theta_w} v_{hx}m_{hx}||_{L^{\theta_w}\backslash L,\mathrm{Lag}}\]
\[=|| \theta_w(m_{hx})^{-1}\theta_w(v_{hx})^{-1}v_{hx}m_{hx}||=|| \theta_w(m_{hx})^{-1}m_{hx}v||\]
for $v\in V\cap L$ because $V\cap L$ is stabilized by $\theta_w$ and normalized by $M$. Now \[|| \theta_w(m_{hx})^{-1}m_{hx}v||\geq || \theta_w(m_{hx})^{-1}m_{hx} ||=||M^{\theta_w} m_{hx}||_{M^{\theta_w}\backslash M,\mathrm{Lag}} \] which proves our claim.
Setting $A_f=\max_{i=1,\dots,r} C_i$ we thus get \[\int_{(P\cap L)^{\theta_w}\backslash L^{\theta_w}} |\lambda(f(hx))| dh\leq 
\int_{(P\cap L)^{\theta_w}\backslash L^{\theta_w}}\phi(hx) dh\times A_f ||L^{\theta_w} x||_{L^{\theta_w}\backslash L,\mathrm{Lag}}^a.\] Now 
by a more direct application of \cite[Theorem 4, (i)]{L08} there are $B>0$ and $\beta>0$ such that \[\int_{(P\cap L)^{\theta_w}\backslash L^{\theta_w}}\phi(hx) dh\leq 
 B ||L^{\theta_w} x||_{L^{\theta_w}\backslash L,\mathrm{Lag}}^\beta\] for all $x\in L$ and the result follows by taking 
 $C_f=BA_f$ and $b=a+\beta$. 
\end{proof}

The second result is a slight generalization of the proof of \cite[Theorem 2.16]{BD08} in our special context.

\begin{LM}\label{lemma proof BD}
Let $Q=LU$ be a self-dual standard parabolic subgroup of $G$, $u\in \mathrm{R}(Q\backslash G/H)$ the representative such that $QuH$ is open in $G$ and set $w=u\theta(u)^{-1}$. Let $\sigma_L$ be a finite length representation of $L$ and 
$\eta$ be a function from $V_{\sigma}$ to $\C$ such that $\eta(\sigma_L(l)v)=\eta(v)$ for all $l\in L^{\theta_w}$ and $v\in V_{\sigma_L}$. Suppose moreover that there exists $a\in \R$ such that for each $v\in V_{\sigma}$ there exists $C_v>0$ such 
$|\eta(\sigma_L(l)v)|\leq C_v ||L^{\theta_w} l||_{L^{\theta_w}\backslash L, \mathrm{Lag}}^a$ for all $l\in L$. Then denoting by $r$ the number of $\GL$-blocks of $L$, there exists $x\in \R$ 
such that for $\s\in \C^r(w,-1)\cap D(w,x)$, the function $h\mapsto \eta(f_{\s}(uh))$ is integrable on $Q(u)\backslash H$ for 
all $f_{\s}\in \Ind_Q^G(\sigma_L[\s])$.
\end{LM}
\begin{proof}
First note that by equivalence of $||\ . \ ||_{L^{\theta_w}\backslash L, \mathrm{Lag}}$ and 
$||\ . \ ||_{L^{\theta_w}\backslash L, \mathrm{BD}}$ one can prove the statement with the latter norm. Now replace the map $\epsilon$ of \cite[Theorem 2.16]{BD08} by the map $\epsilon_{\s}$ from 
$G$ to the set of functions on $V_{\sigma_L}$ which is by definition $0$ outside $QuHu^{-1}$ and defined by 
\[\epsilon_{\s}(qh')=\d_Q^{1/2}(q)\eta\circ \sigma_L[\s](q^{-1})\] for $q\in Q$ and $h'\in uHu^{-1}$. The proof of \cite[Theorem 2.16]{BD08} applies without modification to this map and gives the statement.
\end{proof}

We are now in position to prove a non-Archimedean analogue of \cite[Theorem 31, (1) and (2)]{JLR}.

\begin{thm}\label{theorem convergence and meromorphy of admissible periods}
Let $P=MV$ be a standard parabolic subgroup of $G$ associated to the partition $(m_1,\dots,m_t)$, let $u$ be a $P$-admissible element in $\Rep$ and set $w=uu^{-\theta}\in \Omega_2(M)$. Let $\sigma$ be a finite length representation of 
$M$, and take $\lambda \in \Hom_{M^{\theta_w}}(\sigma,\C)-\{0\}$. Then there is $q_\sigma\in \R$ such that 
for $\s\in D_\sigma^{J,\lambda}(w)=\C^t(w,-1)\cap D(w,q_\sigma)$, the integral 
\[J_{\sigma}(w,f_{\s},\s,\lambda)=\int_{P(u)\backslash H} \lambda(f_{\s}(uh))dh\] is absolutely convergent for all 
$f_{\s}\in \Ind_P^G(\sigma[\s])$. Moreover there is $R\in \C[q^{\pm \s}]-\{0\}$ ($\s\in\C^t(w,-1)$) such that if $f_{\s}$ is any flat section, one has 
$R(q^{\pm \s})J(w,f_{\s},\s,\lambda)\in \C[q^{\pm \s}]$ and $J(w,f_{\s},\s,\lambda)$ is a nonzero element of 
$\C(q^{\pm \s})$ for a well-chosen flat section.
\end{thm}
\begin{proof}
We do a decreasing induction on $\ell(w)$, so we first suppose that $w$ is maximal, i.e. either trivial or of the form 
$(1,t)\circ \dots \circ (i_0,t+1-i_0)$ for $1\leq i_0\leq \lfloor t/2 \rfloor$. The case $w=Id$ is obvious so we suppose that $w=(1,t)\circ \dots \circ (i_0,t+1-i_0)$. We write $M=M_{(m_1,\dots,m_t)}$ so that 
$m_i=m_{t+1-i}$ for $i=1,\dots,i_0$. We also introduce $L=M_{(m_1,\dots,m_{i_0},m',m_{t+1-i_0},\dots,m_{t})}$ where 
$m'=m_{i_0+1}+\dots+ m_{t-i_0}$ and set $\sigma_L=\Ind_{P\cap L}^L(\sigma)$. We denote by $Q=LU$ the standard parabolic subgroup of $G$ with $L$ as a standard Levi subgroup.

If $u\in R(Q\backslash G/H)$ is such that $u\theta(u)^{-1}=w$, then $\sigma_L$ is $L^{\theta_w}$-distinguished with distinguishing linear form 
\[\Lambda(f)= \int_{(L\cap P)^{\theta_w}\backslash L^{\theta_w}} \lambda(f(h))dh.\]

Now we consider the non-negative 
function $\eta$ on $V_{\sigma_L}$ defined by 
\[\eta(f)=
 \int_{(L\cap P)^{\theta_w}\backslash L^{\theta_w}} |\lambda(f(h))|dh.\] 
Then by Lemma \ref{lemma bounded coeff}, for each $f\in V_{\sigma_L}$ there is $C_f>0$ such that for all $l\in L$:
\[\eta(\sigma_L(l).f)\leq C_f||l||^a.\]
Let $r$ be the number of $\GL$-blocks of $L$, there is in particular $x\in \R$ such that for $\s\in \C^r(w,-1)\cap D(w,x)$ and for any $F_{\s}\in \Ind_Q^G(\sigma_L[\s])$ the integral $\int_{Q(u)\backslash H}\eta(F_s(uh))dh$ is 
absolutely convergent according to Lemma \ref{lemma proof BD}. Using the canonical isomorphism  $F\mapsto F(\ . \ )(I_{m_1},\dots,I_{m'},\dots, I_{m_t})$ from $\Ind_Q^G(\sigma_L[\s])$ to 
$\Ind_{P}^G(\sigma[\s])$ and an integration in stages, this exactly says that $\int_{P(u)\backslash H}\lambda (f_{\s}(uh))dh$ is absolutely convergent for 
any $\s\in \C^r(w,-1)\cap D(w,x)$ and any $f_{\s}\in \Ind_P^G(\sigma[\s])$, which proves our first assertion when $w$ is maximal. Once we know this it is now easy to prove the part on meromorphic continuation. 
Indeed for $f_{\s}$ a flat section 
of $\Ind_{P}^G(\sigma[\s])$ we denote by $F_{f_{\s}}$ the associated flat section in $\Ind_Q^G(\sigma_L[\s])$ via the inverse of the canonical isomorphism. Now up to taking $x$ larger if necessary, we have for 
$\s\in \C^t(w,-1)\cap D(w,x)$ and any flay section $f_{\s}\in \Ind_{P}^G(\sigma[\s])$, by Theorem \ref{BD} and integration in stages again 
the equality of absolutely convergent integrals 
\[J_{\sigma_L}(w,\s,F_{f_{\s}},\Lambda)=\int_{P(u)\backslash H}\lambda (f_s(uh))dh\] which gives the part on meromorphic continuation.

Once the result is known for $w$ maximal, the result in general follows by immediate descending induction on $\ell(w)$ thanks to Proposition \ref{proposition JLR key prop p-adic analogue}. 
Indeed, using temporarily the notations of Proposition \ref{proposition JLR key prop p-adic analogue}, the intertwining operator $A_{\sigma}(\tau,\s)$ is invertible for $\s\in D(w,r)$ when $r$ is chosen sufficiently large.
\end{proof}

We now observe a compatibility property of transitivity of parabolic induction with intertwining periods already used in a special case in the proof of Theorem \ref{theorem convergence and meromorphy of admissible periods}. The situation is the following: we are given two standard parabolic subgroups $P=P_{\m}$ and $Q=P_{\m'}$ of $G$ with $Q\subset P$. In particular if $\m$ is of length $t$ and $\m'$ of length $t'$, this gives an associated embedding of $\S_{t}$ inside $\S_{t'}$. We write the standard Levi decompositions $P=MN$ and 
$Q=LU$, and take $w\in \Omega^2(M)\subset \S_{t}\subset \S_{t'}$. Seen as an element of $\S_{t'}$, we assume that $w\in \Omega^2(L)$ as well. We consider $\sigma$ an $L^{\theta_w}$-distinguished representation of $L$ of finite length, and take $\ell\in \Hom_{L^{\theta_w}}(\sigma,\C)-\{0\}$. 
The linear form $\ell$ induces the closed intertwining period $L$ on $\Ind_{Q\cap M}^M(\sigma)$ defined as 
\[L:f\mapsto \int_{Q\cap M^{\theta_w}\backslash M^{\theta_w}} \ell(f(m))dm \] where the quotient $Q\cap M^{\theta_w}\backslash M^{\theta_w}$ is compact. 
The linear form $L$ is a nonzero element of 
$\Hom_{M^{\theta_w}}(\Ind_{Q\cap M}^M(\sigma),\C)-\{0\}$. We denote by $f\mapsto F_f$ the inverse of the canonical isomorphism 
$F\mapsto F(\ . \ )(I_{m_1},\dots, I_{m_t})$ from $\Ind_P^G(\Ind_{Q\cap M}^M(\sigma))$ to $\Ind_Q^G(\sigma)$.

\begin{prop}\label{proposition compatibility of transitivity of parabolic induction and intertwing periods}
In the situation described above, with $u\in \Rep$ such that $u\theta(u)^{-1}=w$. There exists $x\in \R$ such that for $\s\in \C^t(w,-1)\cap D(w,x)\subset \C^{t'}(w,-1)\cap D(w,x)$ and any flat section $f_{\s}$ of $\Ind_Q^G(\sigma[\s])$, both integrals $\int_{P(u)\backslash H}L(F_{f_{\s}}(uh))dh$ and $\int_{Q(u)\backslash H}\ell(f_{\s}(uh))dh$ converge absolutely and are equal. In particular the first integral has meromorphic continuation which we denote by $J_{\Ind_{Q\cap M}^M(\sigma)}(w, F_{f_{\s}},\s,L )$ and we have:
\[J_{\Ind_{Q\cap M}^M(\sigma)}(w, F_{f_{\s}},\s,L )=J_{\sigma}(w,f_{\s},\s,\ell).\]
\end{prop}
\begin{proof}
First we note that $F_{f_{\s}}$ is a flat section of $\Ind_P^G(\Ind_{Q\cap M}^M(\sigma)[\s])$. The appropriate $x$ exists for the second integral thanks to Theorem \ref{theorem convergence and meromorphy of admissible periods} hence the same $x$ works for the first integral and the equality follows by an integration in stages argument. 
\end{proof}

Here is another useful observation:

\begin{rem}\label{remark restatement of proposition JLR-key}
With notations and hypothesis as in Proposition \ref{proposition JLR key prop p-adic analogue}, the equality in Proposition \ref{proposition JLR key prop p-adic analogue} implies the following meromorphic identity on $\C^t(w,-1)$:
\begin{equation}\label{small functional equation} J_\sigma(w,f_{\s},\s,\lambda)=J_{\tau(\sigma)}(w',A_{\sigma}(\tau,\s)f_{\s},\tau(\s),\tau(\lambda)) \end{equation}
\end{rem}

For later use, we also observe an immediate consequence of the above equality.

\begin{prop}\label{proposition preparation to unramified computation}
Let $M_{(m_1,\dots,m_t)}$ be a standard Levi subgroup of $G$ with $t=2r$ even and $m_i=m_{t+1-i}$ for all $i$, and $P$ be the standard parabolic subgroup of $G$ with Levi subgroup $M$. See $\S_r$ as the subgroup of $\S_t$ fixing $\{r+1,\dots,t\}$ set $w_t'=w_r w_t w_r^{-1}$. Set $M'=w_r(M)=M_{(m_r,\dots,m_1,m_r,\dots,m_1)}$ and $P'$ the standard parabolic subgroup of 
$G$ with Levi subgroup $M'$, note that $u_t$ is $P$-admissible and that $u'_t=w_r u_t$ is $P'$-admissible. Let $\sigma$ be an $M^{{\theta_{w_t}}}$-distinguished representation of $M$ of finite length, and take 
$\lambda\in \Hom_{M^{{\theta_{w_t}}}}(\sigma,\C)-\{0\}$. Then one has for any flat section $f_{\s}$ of $\Ind_P^G(\sigma[\s])$ the meromorphic identity on $\C^t(w,-1)$:
\[J_\sigma(w_t,f_{\s},\s,\lambda)=J_{w_r(\sigma)}(w_t',A_{\sigma}(w_r,\s)f_{\s},w_r(\s),w_r(\lambda)).\]
\end{prop}
\begin{proof}
It suffices to notice that $\ell(w'_t)=\ell(w_t)-2\ell(w_r)$, and apply Equation (\ref{small functional equation}) in a row, using 
a reduced expression of $w_r$.
\end{proof}

\section{Computation of the local proportionality constants}

\subsection{Unramified computations}\label{unramified computations}

In this section, we take $D=F$ so that $D_E=E$. We now focus on the explicit computation of intertwining periods at the unramified places following 
\cite[Chapter VII, 20]{JLR}, so the notations will be those of Section \ref{p-adic}. Setting $B_n=B_n(E)$ , we recall that a generic unramified representation $\pi$ of $G_n$ can always be written as a commutative product 
\[\pi=\chi_1\times\dots \times \chi_n=\Ind_{B_n}^{G_n}(\chi_1\otimes \dots \otimes \chi_n),\] where 
the $\chi_i$'s are unramified characters of $G_1$, which don't differ from one another by 
$\nu^{\pm 1}$. We set $G=G_n$, we write $H$ or $H_n$ for $G_n^\theta$, $A=A_n$, $U=N_n(E)$ (or $N_n$), so $B=AU$. If $\pi$ is an unramified generic representation, we denote by $\phi_0$ the normalized spherical vector in $\pi_0=\pi$, hence $\phi_{\s}$ that in $\pi_{\s}$. Let's consider
 $\overline{n}^B= (1,\dots,1)$, and take $u=u_c$ for $c=(n_{i,j})\in I(\overline{n}^B)$. It is automatically $B$-admissible, and we denote by $\xi$ the involution $uu^{-\theta}$ in $\mathfrak{S}_n<G$. 
 Let \[\chi=\chi_1\otimes \dots \otimes \chi_n\] be an unramified character of $A$ which is distinguished 
by $A^u$ and $\s$ an element of 
$\C^n(\xi,-1)$, 
then $\chi[\s]$ is still $A^u$-distinguished. Set 
\[\pi_{\s}=\chi_1\nu^{s_1}\times \dots \times \chi_n \nu^{s_n},\] and suppose that 
$\pi$ is irreducible (i.e. unramified generic). For $f_{\s}$ a flat section for 
$\pi_{\s}$, we formally define the following integral:

\[J_\chi(\xi,f_{\s},\s)=\int_{B(u)\backslash H}f_{\s}(uh)dh.\]

It is absolutely convergent for $\s$ in a cone of the form 
\[D_{\chi}^J(\xi)=\C^n(\xi,-1)\cap D(\xi,q_\chi)\] for some
$q_\chi>0$ and extends to a rational function of $q^{-\s}$ for $\s\in \C^t(\xi,-1)$ thanks to Theorem \ref{theorem convergence and meromorphy of admissible periods}.

Now consider two generic representations of $G_n$ and $G_m$ respectively written as products 
\[\pi=\sigma_1\times \dots \times \sigma_t\] and \[\pi'=\sigma_1'\times\dots \times \sigma_t'\] of generic representations of smaller general linear groups. 
Then by the inductivity relation of the Rankin-Selberg $L$-factors proved in \cite[Proposition (9.4)]{JPSS83}, one has: \[L(s,\pi,\pi')=\prod_{i,j}L(s,\sigma_i,\sigma_j'),\] and by \cite[Theorem 5.3]{M11}, 
the Asai $L$-factors also satisfy a similar relation:
\[L^+(s,\pi)= \prod_{i<j}L(s,\sigma_i,\sigma_j^\theta)\prod_{k}L^+(s,\sigma_k),\] and 
\[L^-(s,\pi)= \prod_{i<j}L(s,\sigma_i,\sigma_j^\theta)\prod_{k}L^-(s,\sigma_k).\]

Now suppose that $\sigma_i=\chi_{i,1}\times\dots\times\chi_{i,n_i}$ is an unramified generic representation of $G_{n_i}$ for 
$i=1,\dots,t$. We identify $\pi=\sigma_1\times \dots \times \sigma_t$ with $\chi_{1,1}\times \dots \times \chi_{t,n_t}$ via the map \[f\mapsto [g\mapsto f(g)(I_{n_1},\dots,I_{n_t})].\] In particular we will talk of the normalized spherical vector 
$\phi\in \pi$, hence of the normalized spherical vector $\phi_{\s}$ in $\pi_{\s}$ for $\s\in \C^t$. The following lemma is a consequence of the Gindikin-Karpelevic formula of \cite{L71}.

\begin{LM}\label{intertwine-spherical}
Suppose that $n=n_1+\dots+n_t$, and let $\sigma=\sigma_1\otimes \dots \otimes \sigma_t$ be a generic unramified representation of $G_{n_1}\times \dots \times G_{n_t}$. We write $[1,n]=[I_1,\dots,I_t]$, with $I_i$ of length $n_i$, and for $1\leq j\leq t-1$  let $\tau\in \S_t$. If $\phi$ is the normalized spherical vector in 
$\pi=\sigma_1\times \dots\times \sigma_t$, one has 
\[A_\sigma(\tau,\s)\phi_{\s}=
\prod_{(i,j)\in \Inv(\tau)}\frac{L(s_i-s_j,\sigma_i,\sigma_{j}^\vee)}{L(s_i-s_{j}+1,\sigma_i,\sigma_{j}^\vee)}\phi_{\tau(\s)}.\]
\end{LM}

Before we state the main result of the section, we identify two intertwining periods seen as intertwining periods for the same representation, but induced from different Levi subgroups. We start with $\overline{n}=(n_1,\dots,n_t)$ a partition of $n$, and $P=P_{\overline{n}}=MN$ the corresponding parabolic subgroup of $G$. As in Section \ref{non-archi non-split}, we let $u\in \Rep$ be a $P$-admissible element, and $w=u\theta(u)^{-1}$ the associated permutation matrix. Moreover we suppose that $u$ is such that $w$ has no fixed points. We then consider $\sigma=\sigma_1\otimes \dots \otimes \sigma_t$ a generic unramified representation of $G_{n_1}\times \dots \times G_{n_t}$, with $\sigma_i=\chi_{i,1}\times\dots\times\chi_{i,n_i}$. We set $\chi_i=\chi_{i,1}\otimes\dots\otimes\chi_{i,n_i}$ and 
$\chi=\chi_1\otimes \dots \otimes \chi_t$, it is an unramified character of $A$. Note that 
$\sigma_i^\vee$ canonically identifies with $\chi_{i,1}^{-1}\times\dots\times\chi_{i,n_i}^{-1}$ via the pairing given 
by integrating on $B_{n_i}\backslash G_{n_i}$. We moreover suppose that $\sigma$ is $M^u$-distinguished, i.e. that 
$\sigma_{w(i)}=(\sigma_i^\theta)^\vee$ for all $i$, or equivalently $\sigma_{w(i)}=\sigma_i^\vee$ as the $\sigma_i$'s are unramified. With the chosen identifications, this is equivalent (up to an appropriate choice of ordering) to $\chi_{w(i),j}=\chi_{i,j}^{-1}$ for all $i$ and 
$j$. Note that for $\s\in \C^t(w,-1)$, the representation $\sigma[\s]$ is still $M^u$-distinguished, and in this situation there is a canonical $M^u$-invariant linear form on 
$\sigma[\s]$ given by:

\[L(\psi_{\s})=\int_{B\cap M^u\backslash M^u}\psi_{\s}(m)dm.\]

The corresponding open intertwining period attached to a flat section 
$f_{\s}:G\rightarrow V_\sigma$ is then 
\[J_\sigma(w,f_{\s},\s,L)=\int_{P(u)\backslash H}L(f_{\s})(h)dh.\]

As the linear form $L$ is canonical, we set \[J_\sigma(w,f_{\s},\s)=J_\sigma(w,f_{\s},\s,L)\] Now identify $\pi=\sigma_1\times \dots \times \sigma_t$ and $\chi_{1,1}\times\dots\times\chi_{t,n_t}$ as already explained, 
then by Proposition \ref{proposition compatibility of transitivity of parabolic induction and intertwing periods} (see also the computation before \cite[Theorem 36]{JLR}) we have:

\[J_\sigma(w,f_{\s},\s)=J_\chi(w,f_{\s},\s).\]

We can now state the main result of this section, which is the explicit computation of the open intertwining period 
of Section \ref{intertwining-periods} for spherical vectors.

\begin{thm} \label{spherical-period}
Let $r$ be a positive integer, $t=2r$, and $\overline{n}=(n_1,\dots,n_t)$ be a self-dual partition of $n$, and $w=w_t\in \S_t$. Let $\sigma_i$ be an unramified generic representation of $G_{n_i}$ for $i=1,\dots,t$, and suppose that $\sigma_{w_t(i)}^\vee=\sigma_i^\theta$. 
Write $[1,n]=[I_1,\dots,I_r,J_r,\dots,J_1]$, with each interval $I_k$ and $J_k$ of length $n_k$.
Set for $\s\in \C^t(w,-1)$, $\pi_{\s}=\nu^{s_1}\sigma_1\times \dots \times \nu^{s_t}\sigma_t$
 and $\phi_{\s}$ the normalized spherical vector in $\pi_{\s}$. Then 
\[J_\sigma(w,\phi_{\s},\s)=
\left[\prod_{1\leq i<j\leq r}\frac{L(s_i-s_j,(\sigma_j^\theta)^\vee, \sigma_i^\theta)L(s_i+s_j,\sigma_i, \sigma_j^\theta)}
{L(s_i-s_j+1,(\sigma_j^\theta)^\vee, \sigma_i^\theta)L(s_i+s_j+1,\sigma_i, \sigma_j^\theta)}\right]
\prod_{k=1}^r\frac{L^+(2s_k,\sigma_k)}{L^-(2s_k+1,\sigma_k)}.\]
\end{thm}
\begin{proof}
Let $w_r\in \S_t$ be the involution defined in Proposition \ref{proposition preparation to unramified computation} and set $w'=w_r \circ w \circ w_r^{-1}$, then by this proposition we have \[J_\sigma(w,\phi_{\s},\s)=J_{w_r(\sigma)}(w',A(w_r,\s)\phi_{\s},w_r(\s)).\] The quantity $A(w_r,\s)\phi_{\s}$ is computed by Lemma \ref{intertwine-spherical} and is equal to 
\[A(w_r,\s)\phi_{\s}=\prod_{1\leq i<j \leq r}\frac{L(s_i-s_j,\sigma_i,\sigma_{j}^\vee)}{L(s_i-s_{j}+1,\sigma_i,\sigma_{j}^\vee)}\phi_{w_r(\s)}=\prod_{1\leq i<j\leq r}\frac{L(s_i-s_j,(\sigma_j^\theta)^\vee, \sigma_i^\theta)}
{L(s_i-s_j+1,(\sigma_j^\theta)^\vee, \sigma_i^\theta)}\phi_{w_r(\s)}\] because $\sigma_k$ is $\theta$-invariant 
and the $L$-factors are symmetric in their last two variables, hence 
\[J_\sigma(w,\phi_{\s},\s)=\prod_{1\leq i<j\leq r}\frac{L(s_i-s_j,(\sigma_j^\theta)^\vee, \sigma_i^\theta)}
{L(s_i-s_j+1,(\sigma_j^\theta)^\vee, \sigma_i^\theta)}J_\sigma(w',\phi_{w_r(\s)},w_r(\s)).\]
 On the other hand with notations of 
Proposition \ref{proposition preparation to unramified computation}, one has 
\[\Ind_{P'}^G(w_r(\sigma[\s]))=\pi'_{(s_r,\dots,s_1)}\times ({\pi'}_{(s_r,\dots,s_1)}^\theta)^\vee\] where 
\[\pi'=\sigma_r\times \dots \times \sigma_1.\] This implies that we can apply 
\cite[Proposition 39]{JLR} to obtain 
\[J_\sigma(w',\phi_{w_r(\s)},w_r(\s))=\frac{L^+(0,\pi'_{(s_r,\dots,s_1)})}{L^-(1,\pi'_{(s_r,\dots,s_1)})}\] 
which by the multiplicativity relation of the Asai $L$-factor is equal to 
\[\left[\prod_{1\leq i<j\leq r}\frac{L(s_i+s_j,\sigma_i, \sigma_j^\theta)}
{L(s_i+s_j+1,\sigma_i, \sigma_j^\theta)}\right]
\prod_{k=1}^r\frac{L^+(2s_k,\sigma_k)}{L^-(2s_k+1,\sigma_k)}.\]
\end{proof}

Finally, we just say a word about the easier case where $E=F\times F$, with $\theta(x,y)=(y,x)$ for $x$ and $y$ in $F$. In this situation the statement of Theorem \ref{spherical-period} 
is still valid, and can be proved in a similar but easier fashion. It is in fact a direct consequence of the Gindikin-Karpelevich 
formula thanks to Equality (\ref{split period vs intertwining}).

\subsection{Ramified proportionality constants}\label{proportionality}

In this paragraph, we will use a local-global method to obtain an explicit expression of the constant 
$\alpha_\sigma(s)$ of Proposition \ref{non-archi-functional-equation}, when the intertwining period is induced from an essentially square-integrable representation. We take $E/F$ a quadratic extension of $p$-adic fields, and $\G_F$ a division algebra of odd index $d$ over its center $F$.
Thanks to Lemma 5 and the proof of Theorem 6 in \cite{K04}, we can choose $l/k$ a quadratic extension of 
number fields, such that:\\
1) there is a unique place $v_0$ of $k$ lying over $p$ and $F\simeq k_{v_0}$.\\
2) $v_0$ is non split in $l$, and if $w_0$ is the place of $l$ dividing it, then $E\simeq l_{w_0}$.\\
3) every infinite place of $k$ splits in $l$.\\

\noindent Thanks to the Brauer-Hasse-Noether theorem (\cite[Theorem 1.12]{PR91}), we can also choose $\G$ a division algebra with center $k$, such that $\G_{v_0}=\G_F$: it automatically splits at every infinite place because its index is odd. We denote by $\theta$ the the involution associated to $l/k$, hence to $l_v/k_v$ for every place $v$ of $k$. The central simple $l$-algebra $\G_l=\G\otimes_k l$ is also a division algebra. Now the group $G_i$ is as in Section \ref{p-adic} for $i\in \N$, as well as the other associated subgroups such as $H_i=G_i^\theta$. 
For $t=2r$ a positive even integer, and $\m=(m_1,\dots,m_t)$ a self-dual partition of $m$, we consider for each 
$i$ between $1$ and $r$ a quasi-square-integrable representation $\d_i$ of $G_{m_i}$, and suppose that 
$\d_{t+1-i}=(\d_i^\theta)^\vee$ for all $i$ between $1$ and $t$. We set $w=w_t$ and $\A=\A_l$. Thanks to Corollary \ref{globalize NM}, for $i=1,\dots,r$, we can consider $\d_i$ as the component at the (only) place (dividing) $v_0$ of a cuspidal automorphic representation $\sigma_i$ of $G_{m_i}(\G_l\otimes_l\A)$ such that $\JL(\sigma_i)$ is also cuspidal, and 
$\JL(\sigma_i)_{v_0}=\JL(\sigma_i)_{w_0}=\JL(\d_i)$ (the first equality is our convention).
We then set 
\[\sigma=\sigma_1\otimes \dots \otimes \sigma_r \otimes (\sigma_r^\theta)^\vee \otimes \dots \otimes (\sigma_1^\theta)^\vee
= \sigma_1\otimes \dots \otimes \sigma_t.\]
It is a cuspidal automorphic representation of $M(\A)$, and there is a unique 
$\underline{r}\in \C^t(w,-1)$ such that $\sigma_{\underline{r}}$ is trivial on $A_M$, \textit{this observation will thus allow to consider the global intertwining period attached to $\sigma$ and $w$}. For $\s\in \C^t(w,-1)$, we define 
as usual $\pi_{\s}=\Ind_{P_{\m}}^{G_m}(\sigma[\s]).$ The linear form
\[L:\psi_{\s}\in \sigma[\s] \mapsto \int_{(A_M)^u M(l)^u\backslash M(\A)^u}\psi_{\s}(m)dm\] is defined by convergent integrals as the Petersson inner product of two cusp forms (see also \cite[Proposition 1]{AGR93}), and it provides a linear form $L_v$ (independant of $\s$) on each $\sigma_{v}[\s]$, which is $(M_v)^u\cap K_v$-invariant and cancelled by $\Lie((M_v)^u)$ when $v$ is infinite, and $(M_v)^u$-invariant when $v$ is finite. Notice that when $v$ is infinite, such a linear form automatically extends to a unique continuous $(M_v)^u$-invariant linear form on 
$\overline{\sigma_{v}[\s]}^\infty$ by \cite[Theorem 1]{BD92}. By local multiplicity one, this implies that the global intertwining period attached to 
a flat section $f_{\s}=\otimes'_{v\in \mathcal{P}(k)} f_{v,\s}$ in $\pi_{\s}$ decomposes for $\s\in \C^t(w,-1)\cap D(w,q_\sigma)$ with $q_\sigma$ large enough, as an infinite product of the intertwining periods studied in Section \ref{open-periods} for good choices of $L_v$:

\[J_\sigma(w,f_{\s},\s)=\prod_{v\in \mathcal{P}(k)} J_{\sigma_v}(w,f_{v,\s},\s,L_v).\]

Moreover, thanks to Theorem \ref{spherical-period}, Remark \ref{nonvanishing-of-global-periods} and its non inert analogue, and Corollary \ref{archimedean non vanishing}, for some good choice of decomposable $f$, we obtain the following result.

\begin{prop}
For some choice of decomposable $f$, the function $\s\mapsto J_\sigma(w,f_{\s},\s)$ is nonzero.
\end{prop}

In fact, taking $S$ as in Propositions \ref{Asai global equation proposition} and \ref{RS global equation proposition}, the  $f_{\s}$ above can be taken such that $f_{v,\s}$ is spherical for $v\notin S$. 
We set \[J_{\sigma,S}(w,f_{\s},\s)=\prod_{v\in S} J_{\sigma_v}(w,f_{v,\s},\s,L_v)\] and \[J_{\sigma}^S(w,f_{\s},\s)=\frac{J_{\sigma}(w,f_{\s},\s)}{J_{\sigma,S}(w,f_{\s},\s)}.\] 
We enlarge $S$ if necessary, so that it contains all places $v$ of $k$ such that $\G$ is non split. Hence for $\s\in \C^t(w,-1)\cap D(w,q_\sigma)$ with $q_\sigma$ large enough, Theorem \ref{spherical-period} and its split analogue yields the following equality where both sides converge:

\[\prod_{1\leq i<j\leq r}
L^S(s_i-s_j+1,(\JL(\sigma_j)^\theta)^\vee, \JL(\sigma_i)^\theta)L^S(s_i+s_j+1,\JL(\sigma_i), \JL(\sigma_j)^\theta)\times\]
\[\prod_{k=1}^r L^{S,-}(2s_k+1,\JL(\sigma_k)) J_{\sigma}^S(w,f_{\s},\s)\]
\begin{equation}\label{global partial period equation 1}
=\prod_{1\leq i<j\leq r}L^S(s_i-s_j,(\JL(\sigma_j)^\theta)^\vee, \JL(\sigma_i)^\theta)L^S(s_i+s_j,\JL(\sigma_i), \JL(\sigma_j)^\theta)
\prod_{k=1}^r L^{S,+}(2s_k,\JL(\sigma_k)). 
\end{equation}

By extension of meromorphic identities Equality (\ref{global partial period equation 1}) is true everywhere. Using Lemma \ref{intertwine-spherical}, the relation 
$L^S(s,\JL(\sigma_r),\JL(\sigma_r)^\theta)=L^{S,+}(s,\JL(\sigma_r))L^{S,-}(s,\JL(\sigma_r))$ and after simplifications, we obtain the following equality of meromorphic functions:

\begin{equation}\label{partial proportionality constant equation -1}
\frac{J_{\tau_r(\sigma)}^S(w,A(\tau_r,\s)f_{\s},\tau_r(\s))}{J_\sigma^S(w,f_{\s},\s)}
=\frac{L^{S,-}(2s_r,\JL(\sigma_r))}{L^{S,-}(1-2s_r,\JL(\sigma_r)^\vee)}\frac{L^{S,+}(-2s_r,\JL(\sigma_r)^\vee)}{L^{S,+}(1+2s_r,\JL(\sigma_r))}.
\end{equation}

Applying Proposition \ref{Asai global equation proposition}, we deduce:

\begin{equation}\label{partial proportionality constant equation}
\frac{J_{\tau_r(\sigma)}^S(w,A(\tau_r,\s)f_{\s},\tau_r(\s))}{J_\sigma^S(w,f_{\s},\s)}
=\prod_{v\in S}\gamma^{-}(2s_r,\JL(\sigma_{r,v}),\psi_v)\gamma^{+}(-2s_r,\JL(\sigma_{r,v})^\vee,\psi_v).
\end{equation}

We will need the following elementary lemma.

\begin{LM}\label{elementaire}
Let $(p_1,\dots,p_l)$ be a finite family of prime numbers, and let $r$ belong to $\N^\times$, then the family 
of functions $(p_i^{-s_j})_{i=1,\dots,l,j=1,\dots,r}$ of the variable $s\in \C^r$ is algebraically independent over $\C$.
\end{LM}
\begin{proof}
We recall that if $u_1,\dots,u_m$ are different complex numbers, then the functions of the complex variable $t$: $e^{u_1t},\dots,e^{u_mt}$, are linearly independent over $\C$ (by a Vandermonde determinant argument for example). This implies at once that if $a_1,\dots,a_l$ are linearly independent over $\Z$, then the functions $e^{a_1t},\dots,e^{a_lt}$ are algebraically 
independent over $\C$. By the prime factorization theorem, this applies to the family $a_1=\ln(p_1),\dots,a_l=\ln(p_l)$, hence 
for all $i=1,\dots,r$, the functions $(p_1^{s_i},\dots,p_l^{s_i})$ are algebraically independent over $\C$. Now suppose 
that $P(p_1^{s_1},\dots,p_l^{s_1},\dots,p_1^{s_r},\dots,p_l^{s_r})=0$ for $P$ a polynomial with coefficients in $\C$. Writing this equality 
\[\sum_{i_1,\dots,i_l}A_{i_1,\dots,i_l}(p_1^{s_1},\dots,p_l^{s_1},\dots,p_1^{s_{r-1}},\dots,p_l^{s_{r-1}})
p_1^{i_1s_r}\dots p_l^{i_ls_r}=0,\] and fixing $s_1,\dots,s_{r-1}$, we deduce that $A_{i_1,\dots,i_l}(p_1^{s_1},\dots,p_l^{s_1},\dots,p_1^{s_{r-1}},\dots,p_l^{s_{r-1}})=0$ for all values of $s_1,\dots,s_{r-1}$. By induction we deduce that $P=0$, which proves the statement of the Lemma.
\end{proof}

It has the following corollary, which generalizes \cite[Lemma 3]{K04}.
 
\begin{LM}\label{unique-factors}
Let $\mathcal{F}$ be a finite set of primes, and for each prime $p\in \mathcal{F}$, let $R_p(s_1,\dots,s_r)$ be an element 
of $\C(p^{-s_1},\dots,p^{-s_r})$, such that \[\prod_{p\in \mathcal{F}} R_p(s_1,\dots,s_r)=1,\]
Then each $R_p \sim 1$.
\end{LM}
\begin{proof}
Write $R_p=A_p/B_p$, with $A_p$ and $B_p$ in $\C[p^{-s_1},\dots,p^{-s_r}]-\{0\}$. This implies that 
\[\prod_{p\in \mathcal{F}} A_p(s_1,\dots,s_r)=\prod_{p\in \mathcal{F}} B_p(s_1,\dots,s_r).\] By Lemma \ref{elementaire}, we can read this 
\[\prod_{p\in \mathcal{F}} A_p(X_p)=\prod_{p\in \mathcal{F}} B_p(X_p),\] with the variables $X_{p,i}$ algebraically independent. 
 Hence we can specialize all $X_{p}$ except one (say $p_0$) to $x_p\in \C^{r}$ such that $A_p(x_p)\neq 0$, from which we get
 $A_{p_0} \sim B_{p_0}$. 
\end{proof}

Now considering for each place $v$ of $k$, the proportionality constants $\alpha_v=\alpha_{\sigma,v}$ defined in Propositions \ref{non-archi-functional-equation} and \ref{archi-functional}. Using the functional equation of global intertwining periods and global Asai $L$-functions with our unramified and Archimedean computations, we obtain the following formula.

\begin{thm}\label{alpha}
Let $\d=\d_1\otimes\dots \otimes \d_t$ be an essentially square-integrable representation of $M$ with $t=2r$ and $\d_{t+1-i}=(\d_i^\theta)^\vee$, let $\psi$ be a non trivial character of $E$, trivial on $F$. Then one has the up to scalar equality of meromorphic functions on $\C^{n}(w,-1)$:
\[\alpha_\d(\s)\sim\gamma^-(2s_r,\JL(\d_r),\psi)^{-1}\gamma^+(-2s_r,(\JL(\d_r))^{\vee},\psi)^{-1}.\]
\end{thm}
\begin{proof}
We consider $l$,$k$, $\sigma$, $\pi$ etc. as in the beginning of the section, hence 
$\sigma_{v_0}=\d$. Note that if the statement of the theorem is valid for one non trivial character $\psi$ of $E/F$, it is valid for any of them. Hence we take $\prod_{v\in \mathcal{P}(k)}\psi_v$ a non trivial 
character of $\A_l$ trivial on $l+\A_k$, and set $\psi:=\psi_{v_0}$. We take $S$ and $f_{\s}$ with the same requirements as those before Equality (\ref{partial proportionality constant equation}), in particular $v_0\in S$. Thanks to the functional equation of $J_\sigma(w,f_{\s},\s)$ (Equation (\ref{global equation}) of Section \ref{global-periods}) and the local functional equations for $v\in S$, we have:
\[\prod_{v\in S}  \alpha_v(s)=\frac{J_{\tau_r(\sigma),S}(w,A(\tau_r,\s) f_{\s},\tau_r(\s))}{J_{\sigma,S}(w,f_{\s},\s)}=\frac{J_{\sigma}^S(w,f_{\s},\s)}{J_{\tau_r(\sigma)}^S(w,A(\tau_r,\s) f_{\s},\tau_r(\s))}\]
which is in turn equal to 
\[\prod_{v\in S} \gamma^-(2s_r,\JL(\sigma_{r,v}),\psi_v)^{-1}\gamma^+(-2s_r,(\JL(\sigma_{r,v})^\theta)^\vee,\psi_v)^{-1}\] thanks to Equation (\ref{partial proportionality constant equation}).
Then thanks to Proposition \ref{archimdean-alpha} and Lemma \ref{unique-factors}, using our assumption on $v_0$, we deduce as expected that 
\[\alpha_{v_0}(s)\sim\gamma^-(2s_r,\JL(\sigma_{r,v_0}),\psi_{v_0})^{-1}\gamma^+(-2s_r,(\JL(\sigma_{r,v_0})^\theta)^\vee,\psi_{v_0})^{-1}.\]
\end{proof}

\section{Distinguished representations of $\GL(m,D)$}\label{basic}

Following the authors of \cite{LM14}, we make the following definition.

\begin{df}\label{definition ladder}
\begin{itemize}
\item Let $(\D_1,\dots,\D_t)$ be a sequence of cuspidal segments, we say that they form a ladder if there is a cuspidal representation $\rho$ such that $\D_i=[a_i,b_i]_\rho$, with $a_1>\dots>a_t$ and $b_1>\dots>b_t$. We also say 
that the essentially square-integrable representation $\d_1=\La(\D_1),\dots,\d_t=\La(\D_t)$ form a ladder in this case. We say that 
the $\D_i$'s (or the $\d_i$'s) form an anti-ladder if the $\D_{w_t(i)}$ form a ladder.
\item We say that the ladder is proper if $\D_{i+1}\prec \D_i$ for all $i$, and that the anti-ladder is proper 
if $\D_i\prec \D_{i+1}$ for all $i$.
\item If $(\d_1,\dots,\d_t)$ is a ladder of essentially square-integrable representations, we call the Langlands' quotient $\La(\d_1,\dots,\d_t)$ of the standard module $\d_1\times \dots \times \d_t$ 
a ladder representation, or just a ladder. We say that it is proper if $(\D_1,\dots,\D_t)$ is proper.
\end{itemize}
\end{df}

In this section, we will use the functional equation of $p$-adic open periods to classify distinguised ladder representations. 
We will also show that thanks to our method, we in fact only need Beuzart-Plessis' result in the cuspidal case (see Section 
\ref{reduction discrete cuspidal} for the reduction to this case). The main 
step will be to understand the singularities of these open periods, this will be done in Section \ref{pole open}. Many proofs of known results for $G=\GL(m,E)$ are valid for $G=\GL(m,\G_E)$. \textit{Whenever we refer for non split $G$ to a result proved only in the split case, this means that the proof still holds for the inner form without modification}.

\subsection{Distinguished induced representations}

We complete the results of Paragraph \ref{Double cosets}, the notations are the same. Most of the results here 
are consequences of Proposition \ref{Geometric lemma}. We first recall Proposition \cite[Proposition 7.1]{O17}.

\begin{prop}\label{closed-orbit-contribution}
Let $\sigma$ be a representation of $M$ such that $\Hom_{M\cap H}(\sigma,\C)\neq 0$, then $\Ind_P^G(\sigma)$ is 
$H$-distinguished. Moreover for $\lambda\in \Hom_{M\cap H}(\sigma,\C)$, the linear map 
\[J_\sigma( . ,\lambda):f\mapsto J_\sigma(f,\lambda)=\int_{P\cap H\backslash H}\lambda(f(h))dh\] is $H$-invariant on $\Ind_P^G(\sigma)$, and the map $\l\mapsto J_\sigma( . ,\lambda)$ is injective.
\end{prop}

We now suppose that $P=P_{(m_1,\dots,m_t)}$ is self dual, and let $u$ be the element of $\Rep$ attached to $w=w_t$, with $t$ 
even or odd. Then the following is a consequence of Theorem \ref{BD} (see \cite[Proposition 2.3]{G15} and more generally \cite[Proposition 7.2]{O17}). The definition of $J_\sigma$ hereunder is straightforward when $t$ is odd (we only defined it when $t$ 
is even) and the properties of such a map are the same according to \cite{BD08}.

\begin{prop}\label{open-orbit-contribution}
With notations as above, and $\sigma=\sigma_1\otimes \dots \otimes \sigma_t$ is a finite length representation of $M$, such that $\Hom_{M^u}(\sigma,\C)\neq \{0\}$, then 
if $L\in \Hom_{M^u}(\sigma,\C)-\{0\}$, there is an element $\underline{a}\in \C^t(w,-1)-\{\underline{0}\}$, and $o\in \Z$ (and in fact necessarily in $\N$), such that 
\[\Gamma_L:f\mapsto \lim_{s\to 0} s^o J_\sigma(w,f_{s\underline{a}},s\underline{a},L)\] defines 
a nonzero $H$-invariant linear form on $\Ind_P^G(\sigma)$.
\end{prop}

\begin{rem}\label{m=0}
The integer $o$ is indeed in $\N$ because of Remark \ref{nonvanishing-of-global-periods}.
\end{rem}

Let us state some consequences of the results stated above. We have the following result due to Gurevich.

\begin{prop}\label{standard}
A standard module $\d_1\times \dots \times \d_t$ of $G$ is distinguished if and only if there is 
$\e\in \S_t$ such that $\d_{\e(i)}=(\d_i^\theta)^\vee$ for all $i$, and $\d_{\e(i)}$ is moreover $\theta$-distinguished 
if $\e(i)=i$.
\end{prop}
\begin{proof}
One direction is \cite[Proposition 3.4]{G15}, its proof is essentially based on Remark \ref{pure tensor}. The converse direction is explained in \cite[Proposition 2.9]{MO17}, and uses 
Propositions \ref{closed-orbit-contribution} and \ref{open-orbit-contribution}. 
\end{proof}

\begin{cor}\label{distinguished conjugate self-dual}
If $\pi$ is an irreducible representation of $G$, which is $H$-distinguished, then $\pi^\vee=\pi^\theta$.
\end{cor}
\begin{proof}
It is a consequence of the statement \cite[Proposition 2.9]{MO17}, the Langlands' quotient theorem, and (2) of Proposition 
\ref{distinction of essentially square-integrable representations}.
\end{proof}

Let's go back to the setting of Proposition \ref{open-orbit-contribution} and Remark \ref{m=0}. A favorable situation is when $o=0$. This is the case in the following situation. 

\begin{prop}\label{open-orbit-contribution+}
In the situation of Proposition \ref{open-orbit-contribution}, suppose that among all $u_i\in \Rep$, the only $M_{u_i}^{u_i}$-distinguished Jacquet module $r_{M_{u_i},M}(\sigma)$ is $r_{M_{u},M}(\sigma)=\sigma$, then $J_\sigma(w,\ . \ ,\s,L)$ is holomorphic at $\underline{0}$. Moreover, the map \[L \mapsto J_\sigma(w,\ . \ ,\underline{0},L)\] is an isomorphism between 
$\Hom_{M^u}(\sigma,\C)$ and $\Hom_{H}(\ind_P^G(\sigma),\C)$. The inverse map between $\Hom_{H}(\ind_P^G(\sigma),\C)$ 
and $\Hom_{M^u}(\sigma,\C)\simeq \Hom_{H^u}(\ind_{M^{u}}^{H^u}(\sigma),\C)\simeq \Hom_{H}(\mathcal{C}_c^\infty(P\backslash PuH,\d_P^{1/2}\sigma),\C) $ is given by the following isomorphism composed 
with the natural isomorphisms above 
\[J\mapsto J_{|\mathcal{C}_c^\infty(P\backslash PuH,\d_P^{1/2}\sigma)}.\]
\end{prop}
\begin{proof}
It follows from Proposition \ref{Geometric lemma} that there is a natural injection of $\Hom_{H}(\ind_P^G(\sigma),\C)$ 
into $\Hom_{M^u}(\sigma,\C)$ in this situation. However \cite[Proposition 7.2]{O17} and its proof show that the map \[L \mapsto J_\sigma(w,\ . \ ,\underline{0},L)\] is an injection in the other direction. The two maps can be shown to be inverse of each other.
\end{proof}

The proposition above applies in the following situation.

\begin{prop}\label{holomorphic standard}
Let $\sigma=\d_1\otimes \dots \otimes \d_t$ be such that $Ind_P^G(\sigma)$ is a standard module, and 
suppose moreover that $\d_i$ and $\d_j$ are not isomorphic whenever $i\neq j$. If $\Hom_{M^u}(\sigma,\C)\neq 0$, then we are in the situation of Proposition \ref{open-orbit-contribution+}, and 
$\Hom_H(\Ind_P^G(\sigma),\C)$ is one dimensional.
\end{prop}
\begin{proof}
This is the content of the proof of \cite[Proposition 3.6]{G15}.
\end{proof}

We will also need to apply Proposition \ref{open-orbit-contribution+} to the following extra cases.

\begin{prop}\label{holomorphic anti-standard}
Let $\d_1=\La(\D_i),\dots,\d_t=\La(\D_t)$ be an anti-ladder of essentially square-integrable representations, such that there is a cuspidal unitary representation $\rho$, with $\D_i=[a_i,b_i]_\rho$, and $a_{t+1-i}=-b_i$ for all $i$. Suppose moreover that 
the $\D_i$ and $\D_j$ are juxtaposed if they are linked for $1\leq i\neq j\leq t$. Set \[\sigma=\d_1\otimes \dots \otimes \d_t.\] 
Then the condition 
$\Hom_{M^u}(\sigma,\C)\neq 0$ is equivalent to the condition $\Hom_H(\Ind_P^G(\sigma),\C)\neq 0$.
 If $\Hom_{M^u}(\sigma,\C)\neq 0$, then we are in the situation of Proposition \ref{open-orbit-contribution+} (and 
$\Hom_H(\Ind_P^G(\sigma),\C)$ is of dimension $1$). 
\end{prop}
\begin{proof}
We do an induction on $t$, the case $t=1$ being trivial. If $t\geq 2$, we only need to consider the Jacquet modules 
$r_{M_a,M}(\d_1\otimes \dots \otimes \d_t)$ with $a\in \Rep$, and $\m_a$ a $\rho$-adapted sub-partition (see Remark \ref{pure tensor} and the conventions there for the notations here) of $\m$.  By Proposition \ref{Geometric lemma}, it is enough to show by induction that if $r_{M_a,M}(\d_1\otimes \dots \otimes \d_t)$ is $M_a^{u_a}$-distinguished, then $u_a=u$. Hence we assume that $r_{M_a,M}(\d_1\otimes \dots \otimes \d_t)$ is $M_a^{u_a}$-distinguished. We write 
$\d_j=[\d_{j,1},\dots,\d_{j,t}]$, with each $\d_{j,l}$ corresponding to a possibly empty segment (i.e. not appearing in 
$\d_j$) of length $m_{j,l}$ where $a=(m_{j,l})$. Let $i_0$ be the smallest integer $\geq 1$ such that $\d_{1,i_0}$ is non empty, so that $\d_1=[\d_{1,i_0},\dots]$. The Jacquet module 
\[r_{M_a,M}(\d_1\otimes \dots \otimes \d_t)=(\d_{1,i_0}\otimes \dots \otimes \d_{1,t}) \otimes \dots \otimes (\d_{i_0,1} \otimes \dots \otimes \d_{i_0,t}) \otimes \dots \otimes (\d_{t,1} \otimes \dots \otimes \d_{t,t})\] being $M_a^{u_a}$-distinguished forces
$\d_{i_0,1}=(\d_{1,i_0}^\theta)^\vee$. In particular $\rho=(\rho^\theta)^\vee$ because $\rho$ is unitary. Then because of the condition on the left and right ends of the segments, the representation $(\d_{i_0}^\theta)^\vee$ must appear as a $\d_j$, and $j$ must in fact be equal to 
$t+1-i_0$. This is possible only if $\d_j=\d_1$ otherwise the segment attached to $\d_1$ would have to precede that attached to $\d_j$ because of the anti-ladder condition, but it would not be juxtaposed to it because $\d_{1,i_0}=(\d_{i_0,1}^\theta)^\vee$, contradicting one of the hypothesis. Hence $t+1-i_0=1\Longleftrightarrow i_0=t$. All in all 
$\d_t=\d_{t,1}=(\d_{1,t}^\theta)^\vee=(\d_1^\theta)^\vee$. Writing $b=(a_{i,j})_{2\leq i\leq t-1,2\leq j\leq t-1}$, and $M'=M_{(m_2,\dots,m_{t-1})}$, then $r_{M_b,M'}(\d_2\otimes \dots \otimes \d_{t-1})$ is $M_b^{u_b}$-distinguished, and we conclude by induction. The space $\Hom_H(\Ind_{P}^G(\sigma),\C)$ has dimension 
one because $\Hom_{M^u}(\sigma,\C)$ has dimension 
one.
\end{proof}

\begin{prop}\label{holomorphic extra-case}
Let $\d_1,\dots,\d_t$ be a ladder of essentially square-integrable representations, with $t=2r$ a positive even integer. Set $\sigma=\d_1\otimes \dots \otimes \d_t$. Suppose moreover that 
$\d_r$ and $\d_{r+1}$ are juxtaposed and set $\tau_r=(r\ r+1)$. If 
$\Hom_{M^u}(\sigma,\C)\neq 0$, then Proposition \ref{open-orbit-contribution+} applies to $\Ind_P^G(\tau_{r}(\sigma))$ (and 
$\Hom_H(\Ind_P^G(\tau_r(\sigma)),\C)$ is of dimension $1$).
\end{prop}
\begin{proof}
We do an induction on $r$. If $r=1$, then we are in the anti-ladder situation of Proposition \ref{holomorphic anti-standard}. If $r\geq 2$, we have $\tau_r(1)=1$ and $\tau_r(t)=t$. There is a cuspidal representation $\rho$ such that each $\d_i$ is of the form $\La([c_i,d_i]_\rho)$, and we only need to consider the Jacquet modules 
$r_{M_a,M}(\d_{\tau_r(1)}\otimes \dots \otimes \d_{\tau_r(t)})$ with $a\in \Rep$, and $\m_a$ a $\rho$-adapted sub-partition. Suppose that $r_{M_a,M}(\d_{\tau_r(1)}\otimes \dots \otimes \d_{\tau_r(t)})$ is $M_a^{u_a}$-distinguished. We proceed as in the proof above, and write again 
$\d_j=[\d_{j,1},\dots,\d_{j,t}]$. Let $i_0$ be the smallest integer $\geq 1$ such that $\d_{1,i_0}$ appears in $\d_1$, so that $\d_1=[\d_{1,i_0},\dots]$. Then 
$\d_{i_0,1}=(\d_{1,i_0}^\theta)^\vee$, this implies that $i_0=t$ because of the ladder condition and because 
$\d_t=(\d_1^\theta)^\vee$. Hence $\d_t=[\d_{t,1},\dots]$ with $\d_{t,1}=(\d_{1,i_0}^\theta)^\vee$. As $\d_t=(\d_1^\theta)^\vee$, this in turn implies that 
$\d_t=\d_{t,1}$ and that $\d_1=\d_{1,t}$. Writing $b=(a_{i,j})_{2\leq i\leq t-1,2\leq j\leq t-1}$, and $M'=M_{(m_2,\dots,m_{t-1})}$, then $r_{M_b,M'}(\d_{\tau_r(2)}\otimes \dots \otimes \d_{\tau_r(t-1)})$ is $M_b^{u_b}$-distinguished, and we conclude by induction. 
The multiplicity one statement is proved as in Proposition \ref{holomorphic anti-standard}.
\end{proof}

\subsection{Properties of p-adic open periods}\label{basic open}

The aim of this section is to prove Proposition \ref{compatibility intertwining periods parabolic induction} which relates  transitivity of parabolic induction, intertwining operators and intertwining periods.

Let $m_1,\dots,m_r$ be positive integers, with $m_r=n+n'$ for $n$ and $n'\in \N$. 
We set $m=2(m_1+\dots+m_r)$, $\m=(m_1,\dots,m_r,m_r,\dots,m_1)$, $\m'=(m_1,\dots,m_{r-1},n,n',n',n,m_{r-1},\dots,m_1)$, and 
$\m''=(m_1,\dots,m_{r-1},n,n',n,n',m_{r-1},\dots,m_1)$. We then introduce the following self-dual parabolic subgroups of 
$G=G_{2(m_1+\dots+m_r)}$ with their standard Levi decompositions: 
\[P=P_{\m}=MN, \ Q=P_{\m'}=LU,\] as well as the non self-dual parabolic subgroup 
$Q'=P_{\m''}=L'U'$.

We consider a representation of $L$ parabolically induced from a finite length representation of the form:
\[\tau=\tau_1\otimes \dots \otimes \tau_{r-1}\otimes \mu\otimes \mu' \otimes ({\mu'}^\vee)^\theta\otimes({\mu}^\vee)^\theta \otimes (\tau_{r-1}^\vee)^\theta\otimes \dots \otimes (\tau_1^\vee)^\theta.\]
We define the linear form $\ell$ on $\tau$ by:
\[\ell(v_1\otimes\otimes \dots \otimes v_{r-1}\otimes v \otimes v'\otimes {v'}^\vee\otimes v^\vee \otimes v_{r-1}^\vee \otimes \dots \otimes v_1^\vee)=  <v,v^\vee> <v',{v'}^\vee>\prod_{i=1}^{r-1} <v_i,v_i^\vee>\]

We set $t=2r$, $w=w_{t+2}\in \S_{t+2}$ and \[\s(u)=(s_1,\dots,s_{r-1},s_r+u,s_r,-s_r, -u-s_r,-s_{r-1},\dots,-s_1)\in 
\C^{t+2}(w,-1).\] Seeing $\S_t$ inside $\S_{t+2}$ thanks to the inclusion $Q\subset P$ and setting $w'=w_t\in \S_t$, we can 
see $\C^t(w',-1)$ inside $\C^{t+2}(w,-1)$ and $\s(0)$ belongs to $\C^t(w',-1)$ and identifies with $\s$. We then set $\tau[\s,u]=\tau[\s(u)]$, $t_r=(r+2,r+3)\in \S_{t+2}$ so that $w'=t_r\circ w\circ t_r^{-1}$, and define the representations of $L'$:
\[\tau'=t_r(\tau)=\tau_1\otimes \dots \otimes \tau_{r-1}\otimes \mu\otimes \mu' \otimes({\mu}^\vee)^\theta \otimes ({\mu'}^\vee)^\theta \otimes (\tau_{r-1}^\vee)^\theta\otimes \dots \otimes (\tau_1^\vee)^\theta\] and more generally 
\[\tau'[\s,u]=t_r(\tau[\s,u])=\tau'[t_r(\s(u))].\] For $u\in\C$, we also consider the finite length representations of $M$:
\[\tau_u=\tau_1\otimes \dots \otimes \tau_{r-1}\otimes (\mu[u]\times \mu' )\otimes (({\mu'}^\vee)^\theta\times({\mu}^\vee)^\theta[-u]) \otimes (\tau_{r-1}^\vee)^\theta\otimes \dots \otimes (\tau_1^\vee)^\theta\]
and 
\[\tau'_u=\tau_1\otimes \dots \otimes \tau_{r-1}\otimes (\mu[u]\times \mu' )\otimes (({\mu}^\vee)^\theta[-u]\times ({\mu'}^\vee)^\theta) \otimes (\tau_{r-1}^\vee)^\theta\otimes \dots \otimes (\tau_1^\vee)^\theta.\] We denote by $F:f\mapsto F_f$ the inverse of the canonical isomorphism 
$F\mapsto F(\ . \ )(I_{m_r},I_{m_r})$ corresponding to transitivity of parabolic induction from 
$\Ind_P^G(\tau_u[\s])$ to \[\pi_{\s,u}:=\Ind_Q^G(\tau[\s,u]).\] Similarly we denote by $G:g\mapsto G_g$ the inverse of the canonical isomorphism 
$G\mapsto G(\ . \ )(I_{m_r},I_{m_r})$ corresponding to transitivity of parabolic induction from 
$\Ind_P^G(\tau'_u[\s])$ to \[\pi'_{\s,u}:=\Ind_{Q'}^G(\tau'[\s,u]).\]

Now introduce \[A(u)=
Id\otimes\dots \otimes Id\otimes  A_{({\mu'}^\vee)^\theta\otimes({\mu}^\vee)^\theta}(w_2,(0,-u))\otimes Id\otimes\dots \otimes Id\] the standard intertwining operator from $\tau_u$ to 
$\tau'_u$ and  
\[\mathcal{A}(u)=A_{\tau}(t_r,(0,\dots,0,0,0,-u,\dots,0))\] the standard intertwining operator from $\pi_{\s,u}$ to $\pi'_{\s,u}$. We denote by $R(A)$ the subset of $\C$ where the intertwining operator $A$ is holomorphic. For any $u\in \C$ and 
$\s\in \C^t(w',-1)$, one defines
the linear form \[\ell':v_1\otimes\otimes \dots \otimes v_{r-1}\otimes v \otimes v'\otimes v^\vee \otimes {v'}^\vee\otimes v_{r-1}^\vee \otimes \dots \otimes v_1^\vee\mapsto  <v,v^\vee> <v',{v'}^\vee>\prod_{i=1}^{r-1} <v_i,v_i^\vee>\] in 
$\Hom_{L^{\theta_{w'}}}(\tau'[\s,u],\C)-\{0\}$ and the associated closed period 
\[L':
v_1\otimes\otimes \dots \otimes v_{r-1}\otimes f \otimes f' \otimes v_{r-1}^\vee \otimes \dots \otimes v_1^\vee\mapsto  <f,f'>\prod_{i=1}^{r-1} <v_i,v_i^\vee>\] in 
$\Hom_{M^{\theta_{w'}}}(\tau'_u,\C)-\{0\}$, where \[<f,f'>=\int_{P_{(n,n')}\backslash G_{m_r}} <f(g),f'(\theta(g))> dg.\]
For $u$ in $R(A)$, this gives birth to a linear form $L_u\in \Hom_{M^{\theta_{w'}}}(\tau_u,\C)-\{0\}$ defined by 
\[L_u=L'\circ A(u).\] 

The following lemma is a consequence of Proposition \ref{proposition compatibility of transitivity of parabolic induction and intertwing periods}.

\begin{LM}\label{lemma compatibility bis}
With notations as above, one has for fixed $u\in \C$ and $g_{t_r(\s(u))}$ a flat section of $\pi'_{\s,u}$, the equality of meromorphic functions in the variable $\s$: 
\[J_{\tau'_u}(w',G_{g_{t_r(\s(u))}},\s,L')=J_{\tau'}(w',g_{t_r(\s(u))},t_r(\s(u)),\ell').\] In particular 
$J_{\tau'_u}(w',G_{g_{t_r(\s(u))}},\s,L')$ depends meromorphically on $t_r(\s(u))$.
\end{LM}

We can now state a useful consequence of Lemma \ref{lemma compatibility bis} and Remark \ref{remark restatement of proposition JLR-key}.

\begin{prop}\label{compatibility intertwining periods parabolic induction}
Take $u\in R(A)$, and let $f_{\s(u)}$ be a flat section of $\pi_{\s,u}$, then the following identity of meromorphic functions in  $\s\in \C^t(w,-1)$ is satisfied: 
\[J_{\tau_u}(w',F_{f_{\s(u)}},\s,L_u)=J_\tau(w,f_{\s(u)},\s(u),\ell).\]
\end{prop}
\begin{proof}
First we set \[z=\begin{pmatrix} & I_{n'} \\ I_n & \end{pmatrix},\] 
\[b(z)=\diag(I_{m_1},\dots,I_{m_{r-1}},I_{m_r}, z,I_{m_{r-1}},\dots,I_{m_1}),\] 
\[a(x)= \begin{pmatrix} I_n & x \\ & I_{n'} \end{pmatrix}\] and \[c(x)= 
\diag(I_{m_1},\dots,I_{m_{r-1}},I_{m_r},a(x),I_{m_{r-1}},\dots,I_{m_1})\] for $x\in \mathcal{M}_{n,n'}$.
By Remark \ref{remark restatement of proposition JLR-key} we have 
\[J_\tau(w,f_{\s(u)},\s(u),\ell)=J_{\tau'}(w',\mathcal{A}(u) f_{\s(u)},t_r(\s(u)),\ell').\] Then by Lemma \ref{lemma compatibility bis}, we deduce that 
\[J_\tau(w,f_{\s(u)},\s(u),\ell)=J_{\tau'_u}(w',\s,G_{\mathcal{A}(u) f_{\s(u)}},L').\]
The seeked equality will thus be true if we prove that \[L_u (F_{f_{\s(u)}}(h))=L'(G_{\mathcal{A}(u) f_{\s(u)}}(h))\] for $h\in H$, and this equality will itself follow if we prove 
\[A(u)(F_{f_{\s(u)}}(h))=G_{\mathcal{A}(u) f_{\s(u)}}(h).\] It then suffices to verify 
\[A(u)(F_{f_{\s(u)}}(h))(I_{m_r},I_{m_r})=G_{\mathcal{A}(u) f_{\s(u)}}(h)(I_{m_r},I_{m_r})\] for any $\s \in \C^t(w,-1)$. Now for fixed $\s$,  both sides of the identity above are meromorphic in $u$, so it is enough to prove it when $\mathcal{A}(u)$ and 
$A(u)$ are given by an absolutely convergent integral. In this case we have 
\[ A(u)(F_{f_{\s(u)}}(h))(I_{m_r},I_{m_r})=\int_{x\in \mathcal{M}_{n,n'}}F_{f_{\s(u)}}(h)(I_{m_r},z a(x))dx= \int_{x\in \mathcal{M}_{n,n'}}f_{\s(u)}(b(z)c(x)h)dx\]
\[=(\mathcal{A}(u) f_{\s(u)})(h) =G_{\mathcal{A}(u) f_{\s(u)}}(h)(I_{m_r},I_{m_r})\] 
\end{proof}

\subsection{Poles of certain p-adic open periods}\label{pole open}

We send the reader back to Propositions \ref{closed-orbit-contribution} and \ref{open-orbit-contribution} for a better understanding of the following result, which is a consequence of multiplicity one. In what follows, if we write 
$\lim_{s\to a} L_s=L_a$ for some linear form $L_s$ on an induced representation $\pi_s$, it will mean that $L_a$ is a linear form 
on $\pi_a$, and $\lim_{s\to a} L_s(f_s)$ tends to $L_a(f_a)$ for any flat section $f_s$.

\begin{prop}\label{open-pole-easy}
Let $\pi$ be a distinguished representation of $G_{m}$, such that $\pi\times \pi$ is irreducible (in particular $\pi$ is irreducible hence $\pi^\theta=\pi^\vee$ by Corollary \ref{distinguished conjugate self-dual}). Take \[\l\in {\rm{\Hom}}_{M_{(m,m)}^{\theta}}(\pi\otimes \pi,\C)-\{0\},\] and 
\[L\in {\rm{\Hom}}_{M_{(m,m)}^{\theta_{w_2}}}(\pi\otimes \pi,\C)-\{0\}.\]
Then $J_{\pi\otimes \pi}(w_2,.,(-s,s),L)$ has a pole at zero, and if this pole has order $k\geq 1$, one has: 
\[\lim_{s\to 0} s^k J_{\pi\otimes \pi}(w_2,.,(-s,s),L)\sim J_{\pi\otimes \pi}(.,\l).\]
\end{prop}
\begin{proof}
By proposition \ref{open-orbit-contribution}, there is $k\geq 0$ such that the linear form $\lim_{s\to 0} s^k J_{\pi\otimes \pi}(w_2,.,(-s,s),L)$ is nonzero. As it lives together with $J_{\pi\otimes \pi}(.,\l)$ in the line (Theorem \ref{mult1}) 
${\rm{\Hom}}_{H_{2m}}(\pi\times\pi,\C)$, the relation $\lim_{s\to 0} s^k J_{\pi\otimes \pi}(w_2,.,(-s,s),L)\sim J_{\pi\otimes \pi}(.,\l)$ follows. It remains to show that $k\geq 1$, but if $k$ was equal to $0$, this would imply that $J_{\pi\otimes \pi}(.,\l)$ 
is non vanishing on the space 
\[\mathcal{C}_c^\infty(P_{(m,m)}\backslash P_{(m,m)}u_{w_2}H_{2m},\d_{P_{m,m}}^{1/2}\pi\otimes \pi),\] 
a contradiction.
\end{proof}

We can extend the result above to the following useful situation. Let $\pi$ be distinguished representation of $G_{m_{r+1}}$, such that $\pi\times \pi$ is irreducible (in particular $\pi^\theta=\pi^\vee$). Let $\sigma_1,\dots,\sigma_r$ be finite length representations of $G_{m_1},\dots,G_{m_r}$. Set $m=2(\sum_{i=1}^{r+1} m_i)$, $t=2r$, and 
\[\overline{m}=(m_1,\dots,m_{r+1},m_{r+1},\dots,m_1).\] Set $w=w_{t+2}$, $\sigma= \sigma_1\otimes \dots \sigma_r \otimes 
\pi\otimes \pi\otimes  (\sigma_r^\theta)^\vee\otimes \dots \otimes (\sigma_1^\theta)^\vee,$
 and define $\ell\in {\rm{\Hom}}_{M_{\m}^{\theta_{w}}}(\sigma,\C)-\{0\},$ by
\[\ell:v_1\otimes \dots \otimes v_r\otimes v\otimes v^\vee\otimes v_r^\vee\otimes \dots \otimes v_1^\vee\mapsto 
<v, v^\vee>\prod_{i=1}^r<v_i,v_i^\vee>,\]
We set $P=P_{\m}$, $G=G_m$, $H=H_m$ and denote by $u_0$ the element of $R(P\backslash G/H)$ corresponding to $w$.

\begin{prop}\label{open-pole-2}
With notations as above, set \[\mathcal{H}=\{\s\in \C^{t+2}(w,-1), s_{r+1}=0\}.\] Then there is $f\in 
\Ind_{P}^G(\sigma)$ such that for any $\underline{a}\in \mathcal{H}$,  
\[\lim_{\s \to \underline{a}} J_{\sigma}(w,f_{\s},\underline{s},\ell)=\infty.\]
\end{prop}
\begin{proof}
We set $\m'=(m_1,\dots,m_r,2m_{r+1},m_r,\dots,m_1)$, $Q=P_{\m'}$, $w'=w_{t+1}$ and denote by $u_1$ the (admissible) element 
of $R(Q\backslash G/H)$ corresponding to $w'$. Let $LU$ be the standard Levi decomposition of $Q$, and $\ell'$ be the linear form 
\[\ell':v\otimes v^\vee\mapsto <v,v^\vee>\] on $\pi\otimes \pi$, and take any $h_{s_{r+1}}\in \nu^{s_{r+1}}\pi\times \nu^{-s_{r+1}}\pi$. For $i=1,\dots,r$, take $v_i$ and $v_i^\vee$ 
such that $<v_i,v_i^\vee>\neq 0$, and take $k$ large enough so that $L\cap K_m(k)$ fixes 
$v_1\otimes \dots \otimes v_r \otimes h_{r+1} \otimes v_r^\vee \otimes \dots \otimes v_1^\vee$. We set $H(k)=H\cap  u_1^{-1}K_m(k)u_1$, and define a map $f_{\s}$ on $G$ supported on $Qu_1H(k)=ULu_1 H(k)$ (which is open in $G$) by the formula 
\[f_{\s}(ulu_1 h)=\d_Q^{1/2}(l)\sigma_1(l_1)v_1\otimes \dots \otimes  \sigma_{r}(l_r)v_r\otimes 
h_{s_{r+1}}(l_{r+1}) \otimes (\sigma_r^\theta)^\vee (l_{r+2})v_r^\vee
\otimes \dots \otimes (\sigma_1^\theta)^\vee(l_{t+1})v_1^\vee.\] 
This map belongs to $\pi_{\s}$, and 
\[J_{\sigma}(w,f_{\s},\underline{s},\ell)\sim  J_{\pi\otimes \pi}(w_2,h_{s_{r+1}},(s_{r+1},-s_{r+1}),\ell').\]
Indeed, set $n=m_1+\dots+m_r$, \[G'=\{g'(g):=\diag(I_n,g,I_n),g\in G_{2m_{r+1}}\}< G,\] $H'=H\cap G'$, $P'=P\cap G'$, and $u'_0=g'\begin{pmatrix} I_{m_{r+1}} & -\d I_{m_{r+1}} \\ I_{m_{r+1}} & \d I_{m_{r+1}} \end{pmatrix}$. Then 
for $\s\in D_{\sigma}^J(w)$, one has 
\[J_{\sigma}(w,f_{\s},\s,\ell)= \int_{Q(u_1)\backslash H}\int_{P'(u'_0)\backslash H'}\ell(f_{\underline{s}}(u_0h'h))dh'dh
=\int_{Q(u_1)\backslash H}\int_{P'(u'_0)\backslash H'}\ell(f_{\underline{s}}(u'_0h'u_1h))dh'dh\]
\[=\int_{Q(u_1)\cap H(k)\backslash H(k)}\int_{P'(u'_0)\backslash H'}\ell(f_{\underline{s}}(u'_0h'u_1h))dh'dh \]
\[\sim \prod_{i=1}^r<v_i,v_i^\vee>\int_{P'(u'_0)\backslash H'}\ell'(h_{s_{r+1}}(u'_0h'))dh'\sim \int_{P'(u'_0)\backslash H'}\ell'(h_{s_{r+1}}(u'_0h'))dh'\]
\[= J_{\pi\otimes \pi}(w_2,h_{s_{r+1}},(s_{r+1},-s_{r+1}),\ell').\] The statement of the proposition now follows from 
Proposition \ref{open-pole-easy}.
\end{proof}

\begin{rem} 
 Keeping the notations of the proof of Proposition \ref{open-pole-2}, we denote by 
$\l$ the up to scalar unique element of $\Hom_{M_{(m_{r+1},m_{r+1})}^\theta}(\pi\otimes \pi,\C)$. Note that $\theta_{w'}$ stabilizes 
$M_{\m}$ and we define $\Lambda \in {\rm{\Hom}}_{M_{\m}^{\theta_{w'}}}(\sigma,\C)-\{0\}$ by
\[\Lambda:v_1\otimes \dots \otimes v_r\otimes v\otimes v^\vee\otimes v_r^\vee\otimes \dots \otimes v_1^\vee\mapsto 
\l(v\otimes v^\vee)\prod_{i=1}^r<v_i,v_i^\vee>.\]
For $\underline{s}\in \C^{t+2}(w,-1)$, we set $\s^-=(s_1,\dots,s_r,0,-s_r,\dots,-s_1)\in \C^{t+1}(w',-1)$ and we define $J_\sigma(w',.,\s^-,\L)$ to be the meromorphic continuation of the following linear form (given by convergent integrals for 
$\s^-\in D(w',r_\sigma)$ for $r_\sigma$ large according to \cite[Theorems 2.8 and 2.16]{BD08} again)
\[J_\sigma(w',.,\s^-,\L):f_{\s^-}\mapsto \int_{Q(u_1)\backslash H} \L(f_{\s^-} (u_1h))dh.\] 

For $k$ as in Proposition \ref{open-pole-easy}, we define \[F(\s)=s_{r+1}^k J_{\sigma}(w,.,\underline{s},\ell),\] it would be interesting to know if
it satisfies the relation 
\[F(\s)_{|\mathcal{H}}=J_{\sigma}(w',.,\underline{s}^-,\L).\] 
We notice that 
$P(u_0)\subset Q(u_1)$, and also that $P(u_0)\backslash Q(u_1)\simeq P'(u'_0)\backslash H'$. We formally have
\[J_{\sigma}(w,f_{\s},\underline{s},\ell)
=\int_{Q(u_1)\backslash H}\int_{P'(u'_0)\backslash H'}\ell(f_{\underline{s}}(u_0h'h))dh'dh.\]
By Proposition \ref{open-pole-easy}, we know that 
\[\lim_{s_{r+1}\to 0} s_{r+1}^k\int_{P'(u'_0)\backslash H'}\ell(f_{\underline{s}}(u_0h'h))dh'= \int_{P'\cap H'\backslash H'}
\L(f_{\underline{s}^-}(u_1h'h))dh',\] hence 
\[\lim_{s_{r+1}\to 0} s_{r+1}^k J_{\sigma}(w,f_{\s},\underline{s},\ell)
=\int_{Q(u_1)\backslash H}\int_{P'\cap H'\backslash H'}
\L(f_{\underline{s}^-}(u_1h'h))dh'dh\] 
\[= J_{\sigma}(w',f_{\underline{s}^-},\underline{s}^-,\L).\] 
But the intertwining periods are defined by integrals only in some cone, and the inversion of the limit and of the integral does not make sense as there is no reason that the cone of convergence should intersect $\mathcal{H}$...
\end{rem}

Now, as a consequence of the results above and those of Section \ref{basic open}, we state and prove the main result of the section (and one of the main results of the paper as well).

\begin{thm}\label{sufficient pole open periods ladders}
Let $t=2r$ be a positive even integer, $(\d_1=\La(\D_1),\dots,\d_t=\La(\D_t))$ be a ladder of essentially square-integrable representations such that $\D_r$ and $\D_{r+1}$ are linked, write $\sigma=\d_1\otimes \dots \otimes \d_t$, and suppose that 
\[\Hom_{M_{(m_1,\dots,m_t)}^{\theta_w}}(\d_1\otimes \dots \otimes \d_t,\C )\] is nonzero for $w=w_t$. Take
$L\in \Hom_{M_{(m_1,\dots,m_t)}^{\theta_w}}(\d_1\otimes \dots \otimes \d_t,\C )-\{0\}$, and write 
$\d_r=\d[s_r]$ with $\d^\vee=\d^\theta$, so that $\d_{r+1}=\d[-s_r]$. More generally write $\sigma[s]$ for 
$\sigma[\s]$, with \[\s=(0,\dots,0,s-s_r,s_r-s,0,\dots,0),\] i.e. 
\[\sigma[s]=\d_1\otimes \dots \otimes \d_{r-1}\otimes \d[s] \otimes \d[-s] \otimes \d_{r+2}\dots \otimes \d_t.\]
If $\D_r$ and $\D_{r+1}$ are juxtaposed, then the open period 
\[J_{\sigma[s]}(w, . ,s,L):=J_{\sigma[\s]}(w, . ,\s,L)\] is holomorphic at $s=-s_r$, whereas if $\D_r$ and $\D_{r+1}$ are non juxtaposed and if $\La(\D_r\cap \D_{r+1})$ is 
distinguished, then the open period $J_{\sigma[s]}(w,.,s,L)$ has a pole at $s=-s_r$.
\end{thm}
\begin{proof}
First, according to Proposition \ref{holomorphic standard}, the intertwining period $J_{\sigma[s]}(w,.,s,L)$ is well defined and holomorphic at $s=s_r$, in particular it defines a meromorphic function of $s$. Notice that Proposition \ref{holomorphic extra-case} gives the first part of the statement, i.e. that $J_{\sigma[s]}(w,.,s,L)$ is holomorphic at $s=-s_r$ if 
$\D_r$ and $\D_{r+1}$ are juxtaposed. 
Suppose now that $\D_r$ and $\D_{r+1}$ are linked non juxtaposed. Write $\D(a)=\D_r\cap \D_{r+1}$, it is a non empty centered segment. 
Write $\D(b)=\D_{r+1}-\D(a)$, it is a possibly empty segment. Write moreover $\d_a=\La(\D(a))[s_r]$, and $\d_b=\La(\D(b))[s_r]$, so that 
$\d$ is the unique irreducible quotient of $\d_b\times \d_a$. Notice that 
$\d_a[-s_r]=\La(\D(a))=(\La(\D(a))^\vee)^\theta=(\d_a^\vee)^\theta[s_r]$. Call $A:\d_b[s]\times \d_a[s]\rightarrow \d[s]$ the (unique up to scaling) nonzero intertwining operator, which we can choose independant of $s$, and pick similarly $B:(\d_a^{\vee})^\theta[-s]\times (\d_b^{\vee})^\theta[-s]\rightarrow \d[-s]$.
 Set \[\gamma[s]=\d_1\otimes \dots \otimes \d_{r-1}\otimes ( \d_b[s]\times \d_a[s])\otimes 
((\d_a^{\vee})^\theta[-s]\times (\d_b^{\vee})^\theta[-s])\otimes \d_{r+2} \otimes \dots \otimes \d_t.\] 
 Then $A$ and $B$ define a surjective intertwining operator 
 \[C=Id_{\d_1}\otimes \dots \otimes Id_{\d_{r-1}}\otimes A\otimes B\otimes Id_{\d_{r+2}}\dots \otimes Id_{\d_r}
 :\gamma[s]\rightarrow \sigma[s].\] 
 Notice that if $F_s$ is a flat section in $\Ind_P^G(\gamma[s])$, then $C(F_s)$ is a flat section in 
 $\pi_s=\Ind_P^G(\sigma[s])$, and any flat section in $\pi_s$ is of that form. 
Setting $L_0=L\circ C$, and taking a flat section $F_s\in \Ind_P^G(\gamma[s])$, then by definition, with obvious notations, one has 
\[J_{\sigma[s]}(w,C(F_s),s,L)=J_{\gamma[s]}(w,F_s,s,L_0).\] 
Hence we just need to prove that $J_{\gamma[s]}(w,F_s,s,L_0)$ has a pole at $s=-s_r$ for some $F$. Now notice that 
\[\gamma[-s_r]=\d_1\otimes \dots \otimes \d_{r-1}\otimes (\d_b[-s_r]\times \La(\D(a)))\otimes (\La(\D(a))\times (\d_b^{\vee})^\theta[s_r])\otimes \d_{r+2} \otimes \dots \otimes \d_t.\] 
To compare with the notations of Section \ref{basic open}, one has $\d_i=\tau_i$ for $i\leq r-1$, $\La(\D(b))[s_r]=\d_b=\mu$, and 
$\La(\D(a))[s_r]=\d_a=\mu'$, a flat section $F_s$ is of the form $F_{\s,u}$, for 
\[\s=(0,\dots,0,s,-s,0,\dots,0),\] and $u=0$, $\gamma=\tau_0$ (i.e. $\tau_u$ for $u=0$) and finally $L_0$ is in fact $L_u$ for $u=0$. In this case, according to Proposition \ref{intertwining-poles}, the standard intertwining operator from $\La(\D(a))[-u]\times (\La(\D(b))^{\vee})^\theta$ to 
$(\La(\D(b))^{\vee})^\theta\times \La(\D(a))[-u]$ is holomorphic at $u=0$ because 
$\D(a)$ and $(\D(b)^{\vee})^\theta$ are juxtaposed segments.  Then we set 
\[\gamma'[s]=\d_1\otimes \dots \otimes \d_{r-1}\otimes \d_b[s]\otimes \d_a[s] \otimes (\d_a^\vee)^\theta[-s]\otimes 
(\d_b^\vee)^\theta[-s]\otimes \d_{r+2} \otimes \dots \otimes \d_t\] and take $f_s$ to be a flat section 
in \[\pi_s=\d_1\times \dots \times \d_{r-1}\times \d_b[s]\times \d_a[s] \times (\d_a^\vee)^\theta[-s]\times 
(\d_b^\vee)^\theta[-s]\times \d_{r+2} \times \dots \times \d_t,\] so with the notations of Section \ref{basic open} is 
$\tau[\s,u]$ for $u=0$ and $\s$ as above. We define \[\ell\in \Hom_{M_{(m_1,\dots,m_t)}^{\theta_{w_{t+2}}}}(\gamma',\C)\] as in Section 
\ref{basic open}. According to 
Proposition \ref{compatibility intertwining periods parabolic induction}, one has:
\[J_{\gamma[s]}(w,F_{f_s},s,L_0)=J_{\gamma'[s]}(w_{t+2},f_s,s,\ell)\] for any 
flat section $f_s\in \pi_s$. However, as $\La(\D(a))$ is distinguished, the local open period 
$J_{\gamma'[s]}(w_{t+2},f_s,s,\ell)$ has a pole at $s=-s_r$ for some flat section $f_s\in \pi_s$ according to Proposition \ref{open-pole-2}. This implies that $J_{\gamma[s]}(w,F_{f_s},s,L_0)$, i.e. that 
$J_{\sigma[s]}(w,C(F_{f_s}),s,L)$ has a pole at $s=-s_r$ and as $F_{f_s}$ (hence $C(F_{f_s})$) is also a flat section, this ends the proof.
\end{proof}

\begin{rem}
It can be shown that the sufficient condition above is also necessary when the standard module above is induced from 
two essentially square-integrable representations.
\end{rem}

\subsection{Distinction of essentially square-integrable representations}\label{reduction discrete cuspidal}

We keep the notations of Section \ref{p-adic}. We first show that our method reduces the use of 
\cite[Theorem 1]{BP18} for essentially square-integrable representations of $G$ to its cuspidal case, which we thus assume, i.e. 
we assume that a cuspidal representation $\rho$ of $G$ is distinguished if and only if $\JL(\rho)$ is.\\

We first start with a simple case of our main result to come on ladder representations. Its proof already contains the main idea of the general proof, which is to use the functional equation of the corresponding intertwining period. If $\D=[a,b]_\rho$, we set 
$\Ze(\D)=\La(\nu^{l_\rho b}\rho,\dots,\nu^{l_\rho a}\rho)$, it is a ladder representation, and it is also the unique irreducible submodule of $\nu^{l_\rho a}\rho\times \dots \times \nu^{l_\rho b}\rho$. 

\begin{prop}\label{distinguished ladder toy}
Let $\rho$ be a cuspidal representation of $G$ and set $l=l_\rho$, then $\Ze([-1/2,1/2]_\rho)$ is distinguished if and only if 
$\rho$ is $\eta^{l-1}$-distinguished.
\end{prop}
\begin{proof}
If $G=G_m$, we denote by $P=MN$ the standard parabolic subgroup of $G_{2m}$ attached to the partition $(m,m)$. We have $\JL(\rho)=\St_l(\rho')$. First, we notice that if $\Ze([-1/2,1/2]_\rho)$ is distinguished, then $\nu^{l/2}\rho\times \nu^{-l/2}\rho$ is distinguished. It then follows from Proposition \ref{standard} that $\rho^\vee=\rho^\theta$. Hence we can assume that $\rho$ is conjugate 
self-dual while proving the statement (we recall that by the cuspidal case of \cite[Theorem 1]{BP18}, if $\rho$ is distinguished, it is indeed conjugate self-dual as this is true in the split case). Set $\sigma=\rho\otimes \rho$, by Proposition \ref{open-orbit-contribution+}, the intertwining period
 $J_\sigma(w_2, . , l/2, \ell)$ is (up to scaling) the only nonzero invariant linear form on $\nu^{l/2}\rho \times \nu^{-l/2}\rho$, and $J_\sigma(w_2, . , -l/2, \ell)$ is (up to scaling) the only nonzero invariant linear form on $\nu^{-l/2}\rho \times \nu^{l/2}\rho$. 
 By Proposition \ref{intertwining-poles}, the intertwining operator $A_\sigma(w_2,-l/2)$ is well defined and nonzero. As the space of intertwining operators between $\nu^{-l/2}\rho\times \nu^{l/2}\rho$ and $\nu^{l/2}\rho\times \nu^{-l/2}\rho$ is one dimensional by adjunction and an easy Jacquet module computation, its image must be the unique proper submodule of $\nu^{l/2}\rho\times \nu^{-l/2}\rho$ as this submodule is a quotient of $\nu^{-l/2}\rho\times \nu^{l/2}\rho$ (see \cite[Proposition 2.7]{T90}), so we deduce that $\Ze([-1/2,1/2]_\rho)$ is distinguished if and only if
$J_\sigma(w_2, . ,l/2, \ell)$ vanishes on the image $A_\sigma(w_2,-l/2)$. By the functional equation (Theorem \ref{proportionality}) of $J_\sigma(w_2, . , s, \ell)$ and 
Proposition \ref{asai gamma poles}, denoting by $\eta$ any extension of $\eta$ to $\GL(m/l,E)$, we have: 
\[J_{\sigma}(w_2,A_{\sigma}(w_2,s)f_s,-s,\ell)
\underset{\C[q^{\pm s}]^\times}{\sim}
\frac{L^+(2s,\eta^l\rho')}{L^+(2s+l,\rho')}\frac{L^+(-2s,\eta^{l+1}\rho')}{L^+(-2s+l,\eta\rho')}J_{\sigma}(w_2,f_s,s,\ell).\] 
However, thanks to Proposition \ref{poles Asai L factor}, the quotient \[\frac{L^+(2s,\eta^l\rho')}{L^+(2s+l,\rho')}\frac{L^+(-2s,\eta^{l+1}\rho')}{L^+(-2s+l,\eta\rho')}\] has a zero at $s=-l/2$ if and only if $\rho'$ is distinguished, i.e. 
if and only if $\St_l(\rho')$ is $\eta^{l-1}$-distinguished, which is the same as $\rho$ being $\eta^{l-1}$-distinguished. It thus follows that $J_\sigma(w_2, . ,l/2, \ell)$ vanishes on the image of $A_\sigma(w_2,-l/2)$ if and only if $\rho$ is $\eta^{l-1}$-distinguished.
\end{proof}

As a consequence, we have the following result.

\begin{LM}\label{antistandard}
Let $\rho$ be a conjugate self-dual cuspidal representation of $G_r$, and set $l=l_\rho$. The representation  \[\pi_j(\rho)=\nu^{\frac{l(1-k)}{2}}\rho\times \dots \times \nu^{\frac{l(j-3)}{2}}\rho \times 
\Ze([\frac{j-1}{2},\frac{j+1}{2}]_\rho)\times \nu^{\frac{l(j+3)}{2}}\rho\times \dots \times \nu^{l\frac{l(k-1)}{2}}\rho\]
(with $j\equiv k \mod 2$) is not distinguished if $j\neq 0$, or if $j=0$ (hence $k$ is even) and $\rho$ is $\eta^l$-distinguished.
\end{LM}
\begin{proof}
Set $\pi=\pi_j(\rho)$ and \[\tau= \nu^{\frac{l(1-k)}{2}}\rho\otimes \dots \otimes \nu^{\frac{l(j-3)}{2}}\rho \otimes 
\Ze([\frac{j-1}{2},\frac{j+1}{2}]_\rho)\otimes \nu^{\frac{l(j+3)}{2}}\rho\otimes \dots \otimes \nu^{\frac{l(k-1)}{2}}\rho.\]
Let $m=kr$, and $\overline{m}=(r,\dots,r,2r,r,\dots,r)$  with $2r$ in position $\frac{k+j}{2}$. If $\pi$ is distinguished, then by Proposition \ref{Geometric lemma}, there is an element $a\in I(\overline{m})$ such that 
 $r_{M_a,M}(\tau)$ is $M_a^{u_a}$-distinguished. Notice that the only nonzero Jacquet modules of $\tau$ are obtained for $\overline{m}_a=\overline{m}$, in which case $r_{M_a,M}(\tau)=\tau$, and $\overline{m}_a=(r,\dots,r)$, in which case 
 \[r_{M_a,M}(\tau)=\nu^{l\frac{1-k}{2}}\rho\otimes \dots \otimes \nu^{l\frac{k-1}{2}}\rho.\]
In the first case, writing $a=(m_{i,j})$, for $r_{M_a,M}(\tau)$ to be $M_a^{u_a}$-distinguished one should have $m_i=m_{i,i}(=2r)$ for $i=\frac{k+j}{2}$. Hence in this case $\Ze([\frac{j-1}{2},\frac{j+1}{2}]_\rho)$ should be distinguished, but this is impossible if $j\neq 0$ because the central character of $\Ze([\frac{j-1}{2},\frac{j+1}{2}]_\rho)$ is not unitary, and it is also impossible if $j=0$ and $\rho$ is $\eta^l$-distinguished according to Proposition \ref{distinguished ladder toy}.
In the second case, the only possibility would be $m_i=m_{i,k+1-i}$, which is not the case.
\end{proof}

We now recover the distinction result concerning essentially square-integrable representations of $G$. We recall from the proof of \cite[Proposition 2.7]{T90} that the kernel of the surjective intertwining operator from $\nu^{\frac{l(1-k)}{2}}\rho\times \dots \times \nu^{\frac{l(k-1)}{2}}\rho$ to 
$\St_k(\rho)$ is equal to the sum of the representations 
$\pi_j(\rho)$ from $j=2-k$ to $k-2$ with $j\equiv k[2]$. 

\begin{prop}\label{Steinberg1}
Let $\rho$ be a cuspidal representation of $G_f$, and $l=l_\rho$.\\
If $k$ is odd: $\St_k(\rho)$ is distinguished if and only if $\rho$ is.\\
If $k$ is even: $\St_k(\rho)$ is distinguished if and only if $\JL(\rho)$ is $\eta^l$-distinguished, i.e 
if and only if $\rho$ is $\eta^l$-distinguished.\\
In both cases $\St_k(\rho)$ is distinguished if and only $\JL(\St_k(\rho))$ is distinguished.
\end{prop}
\begin{proof}
We set $M=M_{(f,\dots,f)}$. First we start with $k$ odd. If $\St_k(\rho)$ is distinguished, then by Proposition \ref{holomorphic anti-standard}, the 
representation $\nu^{\frac{l(1-k)}{2}}\rho\otimes \dots \otimes \nu^{\frac{l(k-1)}{2}}\rho$ is $M^u$-distinguished, hence $\rho$ is distinguished. Conversely, if $\rho$ is distinguished, as $\rho^\vee=\rho^\theta$, 
then $\nu^{\frac{l(1-k)}{2}}\rho\times \dots \times \nu^{\frac{l(k-1)}{2}}\rho$ is distinguished according to Proposition \ref{open-orbit-contribution}. However, no $\pi_j(\rho)$ is distinguished according to Lemma 
\ref{antistandard}, so $\St_k(\rho)$ is distinguished. When $k=2r$ is even, if $\St_k(\rho)$ is distinguished, then again by 
Proposition \ref{holomorphic anti-standard}, the Jacquet module $\nu^{\frac{l(1-k)}{2}}\rho\otimes \dots \otimes \nu^{\frac{l(k-1)}{2}}\rho$ is 
$M^u$-distinguished, hence $\rho^\theta\simeq \rho^\vee$. We can thus consider, for 
\[\sigma[s]=\rho[\frac{l(1-k)}{2}]\otimes \dots \otimes \rho[\frac{-3l}{2}]\otimes \rho[s]\otimes \rho[-s] \otimes 
\rho[\frac{3l}{2}]\dots \otimes \rho[\frac{l(k-1)}{2}],\] the intertwining period 
$J_{\sigma}(w,.,s,L)$ with $L\in \Hom_{M^u}(\sigma,\C)-\{0\}$. As we supposed that $\St_k(\rho)$ is distinguished, it in particular implies that the linear form $J_{\sigma}(w,.,-l/2,L)$ which is well defined according to Proposition \ref{holomorphic anti-standard}, must vanish on $\pi_0(\rho)=Im(A_{\sigma}(\tau_r,l/2))$ ($A_{\sigma}(\tau_r,l/2)$ is well defined by Proposition \ref{intertwining-poles}). However, by Theorem \ref{proportionality}, setting 
$\JL(\rho)=\St_l(\rho')$ and denoting again by $\eta$ any extension of $\eta$ to $\GL(f/l,E)$: 
\[J_{\sigma}(\tau_r,A_{\sigma}(\tau_r,s)f_s,-s,L)
\underset{\C[q^{\pm s}]^\times}{\sim}
\frac{L^+(2s,\eta^l\rho')}{L^+(2s+l,\rho')}\frac{L^+(-2s,\eta^{l+1}\rho')}{L^+(-2s+l,\eta\rho')}J_{\sigma}(\tau_r,f_s,s,L).\] 
Making $s$ tend to $l/2$ above, the only possibility is that the quotient 
\[\frac{L^+(2s,\eta^l\rho')}{L^+(2s+l,\rho')}\frac{L^+(-2s,\eta^{l+1}\rho')}{L^+(-2s+l,\eta\rho')}\] has a zero at $s=l/2$, i.e. that $\rho'$ is $\eta$-distinguished (and then also that 
$J_{\sigma}(\tau_r,.,s,L)$ is holomorphic at $l/2$, this is in fact automatic but we don't need to know it), i.e. 
that $\rho$ is $\eta^l$-distinguished. Conversely, if 
$\rho$ is $\eta^l$-distinguished, the linear form $J_{\sigma}(w,.,-l/2,L)$ vanishes on all 
$\pi_j(\rho)$ according to Lemma \ref{antistandard}, hence $\St_k(\rho)$ is distinguished.
\end{proof}

\begin{rem}
For $D=F$, Proposition \ref{Steinberg1}, which is \cite[Corollary 4.2]{M09}, was mainly proved in \cite{AR05}, as a consequence 
of their proof of the equality of the Flicker and Langlands-Shahidi Asai $L$-factor for essentially square-integrable representations. Here we give a different proof, even in the case $D=F$, which is still local/global as in \cite{AR05}. 
\end{rem}

\subsection{Distinction of proper ladders}\label{dist proper ladders}

Now we recall \cite[Theorem 1.1 (i)]{LM14}, which is valid for non split $G$ as noticed in the introduction of 
[ibid.].

\begin{thm}\label{subladder}
Let $\mathcal{S}$ be a standard module attached to a proper ladder 
$(\D_1,\dots,\D_t)$. Then 
the kernel of the nonzero intertwining operator from $\mathcal{S}$ to $\La=\La(\D_1,\dots,\D_t)$ is equal to 
the sum for $i$ between $1$ and $t-1$ of the standard modules 
\[\mathcal{S}_i=\La(\D_1)\times \dots \times \La(\D_{i-1})\times 
\La(\D_i\cup \D_{i+1})\times \La(\D_i\cap \D_{i+1})\times \La(\D_{i+2})\times \dots \times \La(\D_t).\]
\end{thm}

We want to know when certain $H$-invariant linear forms on $\mathcal{S}$ descend to 
$\La$. Clearly, this is the case if and only if each $\mathcal{S}_i$ lies in the kernel of such a 
linear form. The following easy consequence of Proposition \ref{standard}, which is a special case of a part of the proof of 
\cite[Theorem 4.2]{G15}, will thus be useful. 

\begin{LM}\label{distinguished sub-standard}
Let $\D_1,\dots,\D_t$ be a proper ladder, such that $\La(\D_{t+1-i})=(\La(\D_i)^\theta)^\vee$ for all $i$. If $t\geq 2$ is odd, then no $\mathcal{S}_i$ is distinguished. If $t=2r$ is even, then no $\mathcal{S}_i$ is distinguished for $i\neq r$, and $\mathcal{S}_r$ is distinguished if and only 
$\La(\D_r\cup \D_{r+1})$ is distinguished. This latter condition is equivalent to $\La(\D_r\cap \D_{r+1})$ being distinguished 
if $\D_r$ and $\D_{r+1}$ are non juxtaposed.
\end{LM}
\begin{proof}
Suppose that $\mathcal{S}_i$ is distinguished. We set $\D'_k=\D_k$ if $k\notin\{i,i+1\}$, $\D'_i=\D_i\cap \D_{i+1}$ and 
$\D'_{i+1}=\D_i\cup \D_{i+1}$. Suppose that $\mathcal{S}_i=\D'_1\times \dots \times \D'_t$ is distinguished, and let
 $\e$ be the associated involution of Proposition \ref{standard} (with the convention 
 that $\e(i)=i$ if $\D'_i$ is empty). We do an induction on $t\geq 2$. If $t=2$, then $\La(\D_1\cap \D_{2})\times \La(\D_1\cup \D_{2})$ is distinguished, and as the segments 
$\D_1\cup \D_{2}$ and $\D_1\cap \D_{2}$ have different lengths, the involution $\e$ must be the identity and we are done. 
Suppose that $t\geq 3$. First we consider the case $i\neq 1$.\\
 If $i< t-1$, then $\La(\D_{\e(t)})=(\La(\D_t)^\theta)^\vee$ must be equal to $\La(\D_1)$, so $\e(t)=1$. Hence if $i\neq t-1$, we conclude by induction applied to \[\La(\D_2)\times \dots \times \La(\D_{i-1})\times 
\La(\D_i\cup \D_{i+1})\times \La(\D_i\cap \D_{i+1})\times \La(\D_{i+2})\times \dots \times \La(\D_{t-1}).\]
If $i=t-1$, then $(\La(\D_{t-1}\cup \D_t)^\theta)^\vee$ should be equal to $\La(\D_1)$ considering the end of the segment associated to $(\La(\D_{t-1}\cup \D_t)^\theta)^\vee$, and this is impossible for length reasons. \\
Finally the remaining case $i=1$ is also impossible for the same argument as above, hence we are done with the first assertion. 
For the second, write $\La(\D_r\cap \D_{r+1})=\St_k(\rho)$ for $\rho^\theta=\rho^\vee$. Then $\La(\D_r\cup \D_{r+1})$ 
is of the form $\St_{k+2b}(\rho)$ for $b\in \N$, and the assertion follows from 
Proposition \ref{distinction of essentially square-integrable representations}.  
\end{proof}

Using Theorem \ref{subladder}, and Proposition \ref{standard}, we easily obtain the following result. It is stated and proved in \cite{G15} in different terms. 

\begin{prop}\label{odd ladder}
Let $\La=\La(\D_1,\dots,\D_t)$ be a proper ladder representation, with $t$ is odd. Then $\La$ is distinguished if and only if $\La^\vee=\La^\theta$ and the middle essentially square-integrable representations $\La(\D_{\frac{t+1}{2}})$ is distinguished.
\end{prop}
\begin{proof}
We set $\d_i=\La(\D_i)$, and $w=w_t$. If $\La$ is distinguished, then the conditions $L^\vee=L^\theta$ and $\La(\D_{\frac{t+1}{2}})$ distinguished are a consequence of Proposition \ref{standard}, where we notice that the ladder condition implies that $\e=w_t$. Conversely, let $M$ be the standard Levi subgroup of which $\sigma=\d_1\otimes \dots \otimes \d_t$ is a representation. If $\La^\vee=\La^\theta$ and $\La(\D_{\frac{t+1}{2}})$ is distinguished, i.e. if $\Hom_{M^{\theta_w}}(\sigma,\C)\neq 0$, then 
\[\mathcal{S}=\D_1\times \dots \times \D_t\] is distinguished by Proposition \ref{standard} again, but no $\mathcal{S}_i$ 
by Lemma \ref{distinguished sub-standard}. Hence $\La$ is distinguished.
\end{proof}

Finally we are now able to prove the main theorem of the paper.

\begin{thm}\label{even ladder}
Let $\La=\La(\D_1,\dots,\D_t)$ be a proper ladder representation, with $t=2r$ even. Then $\La$ is distinguished if and only if $\La^\vee=\La^\theta$ and $\La(\D_r\cup \D_{r+1})$ is $\eta$-distinguished.
\end{thm}
\begin{proof}
We set \[\sigma[s]=\d_1\otimes \dots \otimes \d_{r-1} \otimes \d[s] \otimes \d[-s] \otimes \d_{r+2} \otimes \dots \otimes \d_t .\] By Proposition \ref{standard}, we can assume that $\sigma$ is $M^w$-distinguished. In this situation, thanks to Proposition \ref{holomorphic standard}, the ladder $\La$ is distinguished if and only if the $H$-invariant linear form $\Lambda$ on $\mathcal{S}$ descends to $\La$. According to Lemma \ref{antistandard}, this is equivalent to 
$\Lambda$ being zero on $\mathcal{S}_r=\mathrm{Im}(M_r)$, where $M_r$ is the regularized intertwining operator from 
$\Ind_P^G(\sigma[-s_r])$ to $\mathcal{S}=\Ind_P^G(\sigma[s_r])$. We set $\d=\St_k(\rho)$, and $\JL(\rho)=\St_l(\rho')$, hence 
$\JL(\d)=\St_{kl}(\rho')$. Let us write the functional equation in this case:
\[J_{\sigma}(w,A_{\sigma}(\tau_r,s)f_s,-s,L)
\underset{\C[q^{\pm s}]^\times}{\sim}
\frac{L^+(2s,\eta^{kl}\rho')}{L^+(2s+kl,\rho')}\frac{L^+(-2s,\eta^{kl+1}\rho')}{L^+(-2s+kl,\eta\rho')}J_{\sigma}(w,f_s,s,L).\] 
Notice that $J_{\sigma}(w,.,s_r,L)$ is always well defined. There are then two cases. Either $\D_r$ and $\D_{r+1}$ are juxtaposed, which amounts to saying that $s_r=lk/2$. In this case $J_{\sigma}(w,.,-s_r,L)$ and $A_{\sigma}(\tau_r,-s_r)$ are well defined thanks to Theorem \ref{sufficient pole open periods ladders} and Proposition \ref{intertwining-poles}. 
The functional equation thus tells us that $J_{\sigma}(w,A_{\sigma}(\tau_r,-s_r)f_s,s_r,L)$ will vanish for all 
$f$ if and only if $L^+(-2s_r+kl,\rho')=\infty$, i.e. if and only is $\rho'$ is distinguished. This is then equivalent to $\St_{2lk}(\rho')$ being $\eta$-distinguished, i.e. to $\JL(\St_{2lk}(\rho'))=\La(\D_r\cup \D_{r+1})$ being $\eta$-distinguished by Proposition \ref{Steinberg1}.\\
Now if $\D_r$ and $\D_{r+1}$ are linked but not juxtaposed, i.e. $s_r<lk/2$. Then one has \[\frac{L^+(2s_r,\eta^{kl}\rho')}{L^+(2s_r+kl,\rho')}\frac{L^+(-2s_r,\eta^{kl+1}\rho')}{L^+(-2s_r+kl,\eta\rho')}\neq 0.\] 
On the other hand $A_{\sigma}(\tau_r,s)$ has a simple pole at $-s_r$, i.e. 
$M_r=\lim_{s\to -s_r} (s+s_r)A_{\sigma}(\tau_r,s)$. Hence $J_{\sigma}(w,M_r f_s,s_r,L)$ is zero for all 
$f$ if and only if $\lim_{s\to -s_r} (s+s_r)J_{\sigma}(w,.,s,L)$ is zero. This is the case if and only if 
$J_{\sigma}(w,.,s,L)$ is holomorphic at $s=-s_r$. In particular according to Theorem \ref{sufficient pole open periods ladders} this implies that $\La(\D_r\cap \D_{r+1})$ is not distinguished, i.e. is $\eta$-distinguished. Conversely, if $\La(\D_r\cap \D_{r+1})$ is $\eta$-distinguished, in particular
 not distinguished according to Proposition \ref{distinction of essentially square-integrable representations}, this implies that 
 $\mathcal{S}_r$ is not distinguished thanks to Lemma \ref{distinguished sub-standard}, hence $\La$ is distinguished.
\end{proof}

\begin{rem}
Notice that the above proof implies that $J_{\sigma}(w,.,s,L)$ is holomorphic at $s=-s_r$ if and only if $\La(\D_r\cap \D_{r+1})$ is not distinguished, and moreover that if $J_{\sigma}(w,.,s,L)$ has a pole at $s=-s_r$, then it is simple.
\end{rem}

\subsection{Distinguished ladder and unitary representations}\label{unitary and ladders}

The notations are as above. If $\d$ is a square-integrable representation of $G$ and $k\geq 1$, we set $l=l_\d$ and denote 
by $u(\d,k)$ the ladder representation 
\[u(\d,k)=\La(\nu^{\frac{l(k-1)}{2}}\d,\dots,\nu^{\frac{l(1-k)}{2}}\d).\]

By \cite{T90}, \cite{BR04}, and \cite{S09}, any unitary representation of $G$ can be written in a unique manner as a commutative product of representations 
of the type:

\begin{itemize}
\item $u(\d,k)$ for $\d$ a square-integrable representation. We call $u(\d,k)$ a Speh representation.
\item $\nu^{-l_\d\alpha}u(\d,k)\times \nu^{l_\d\alpha}u(\d,k)$ for $\d$ a square-integrable representation, and $\alpha\in ]0,1/2[$.
\end{itemize}

If $\pi=\La_1 \times \dots \times \La_t$ is a commutative product of proper ladders, we say that $\pi$ is $\theta$-induced 
if there is an involution $\e\in \S_t$, such that $\La_{\e(i)}^\vee=\La_i^\theta$ for all $i$, and 
$L_i$ is distinguished whenever $\e(i)=i$. We say that proper ladder representations $\La_1,\dots,\La_t$ are (mutually) unlinked if 
no segment occurring in a $\La_i$ is linked with a segment occurring in $\La_j$ if $i\neq j$, and in this case 
their product is commutative. The following result, which is \cite[Proposition 7.3]{GMM17}, is true for non split $G$ with the same proof.  

\begin{prop}
Let $\La_1,\dots,\La_t$ be mutually unlinked ladders, then the product $\La_1\times \dots \times \La_t$ is distinguished if and only if it is $\theta$-induced.
\end{prop}

In particular, as a ladder representation is a product of unlinked proper ladder representations, this classifies 
distinguished ladder representations in terms of distinguished essentially square-integrable representations (and in fact cuspidal representations 
by Proposition \ref{Steinberg1}). It also reduces the classification of distinguished unitary representations to the following statement. 

\begin{thm}
Let $\d_1,\dots,\d_t$ be square-integrable representations with each $\d_i$ of the form $\La([-a_i,a_i]_{\rho})$ with 
$a_i\in \frac{1}{2}\N$, and $\rho$ a conjugate self-dual cuspidal representation. The unitary representation $\pi=u(\d_1,k_1)\times \dots \times u(\d_t,k_t)$ is distinguished if and only if it is $\theta$-induced.
\end{thm}
\begin{proof}
Each $\d_i$ is a representation of $G_{l_i}$, and we set $m_i=k_i l_i$ and $\overline{m}=(m_1,\dots,m_t)$. We can always switch the order in the product, so that $m_i\geq m_{i+1}$ for all $1\leq i \leq t-1$. If $\pi$ is $\theta$-induced, it is certainly distinguished as an application of Propositions
 \ref{closed-orbit-contribution} and \ref{open-orbit-contribution}. For the converse direction, according to Proposition \ref{Geometric lemma}, it is enough to prove the following statement:
\textit{"let $\d_1,\dots,\d_t$ be square-integrable representations which have cuspidal support on the same cuspidal line. If the unitary representation $\pi=u(\d_1,k_1)\times \dots \times u(\d_t,k_t)$ induced from the standard Levi $M=M_{\overline{m}}$ is such that 
$\mu=u(\d_1,k_1)\otimes \dots \otimes u(\d_t,k_t)$ has a Jacquet-module $r_{M_a,M}(\mu)$ which is $\theta_{w_a}$-distinguished for $a\in I(\overline{m})$, then it is $\theta$-induced."}\\
 So let's prove it by induction on $t$. It will be more convenient to write 
\[u(\d_i,k_i)=\La(\d_{i,1},\dots,\d_{i,k_i})=\La_i,\] with $\d_{i,j}\succ \d_{i,j+1}$. Thanks to our assumption, and the description of Jacquet modules of ladder representations given in \cite{KL12} (see the picture there for a visual description of ladders and their Jacquet modules as well), we can write each 
$\d_{i,j}$ as $\d_{i,j}=[\d_{i,j}^t,\dots,\d_{i,j}^1]$ with $(\d_{i,1}^k,\d_{i,2}^k,\dots,\d_{i,k_i-1}^k,\d_{i,k_i}^k)$ 
forming a ladder of all $k$, so that if we set \[\La_i^k=\La(\d_{i,1}^k,\d_{i,2}^k,\dots,\d_{i,k_i-1}^k,\d_{i,k_i}^k),\]
then $\La_i^k$ is a (possibly trivial) ladder, and the representation 
\[\La_1^1\otimes 
\La_1^2\otimes \dots \otimes \La_1^{t-1}\otimes \La_1^t \otimes \dots \otimes \La_t^1 \otimes \La_t^2\otimes \dots 
\otimes \La_t^{t-1}\otimes \La_t^t\]
is $\theta_{w_a}$-distinguished.
We select $i_0$ the smallest integer between $1$ and $t$, such that $\d_{1,1}^{i_0}$ is non trivial.
 If $i_0=1$, then $\La_1^1$ is conjugate selfdual. As $\d_{1,1}^{i_0}$ corresponds to the upper right bit of the ladder $\La_1^1$, then 
 $((\d_{1,1}^{i_0})^\theta)^\vee$ corresponds to its lower left bit. The representation $\d_{1,1}^{i_0}$ is also the upper right bit of $\La_1$, and  
 $((\d_{1,1}^{i_0})^\theta)^\vee$ its lower left bit because $\La_1$ is a Speh representation.  In particular $\La_1^1$  necessarily has $k_1$ floors, and using the intuitive notation for concatenation of ladder representations 
\[\La_1=[\La_1^1,\dots,\La_1^t],\] we see that it implies that $\La_1^1=\La_1$. \textit{Notice that the visual picture 
of the ladder from \cite{KL12} that we use to describe it does not match with the notation $\La_1=[\La_1^1,\dots,\La_1^t]$ which is more adapted to Jacquet modules, right and left should be reversed}. Hence if $i_0=1$, we can conclude by induction 
applied to \[\La_2\times \dots \times \La_t.\]
If $i_0>1$, then consider \[\La_{i_0}=[\La_{i_0}^1,\dots,\La_{i_0}^t].\] 
As $(\La_{i_0}^{1})^\theta=(\La_{1}^{i_0})^\vee$, then $((\d_{1,1}^{i_0})^\theta)^\vee$ is the lower left bit of the "right" sub-ladder $\La_{i_0}^1$ of $\La_{i_0}$. As $m_1\geq m_i$ for all $i$, and because all $\La_i$ are Speh representations, the beginning of the segment corresponding to the representation $((\d_{1,1}^{i_0})^\theta)^\vee$ is the smallest of all beginnings of all 
segments occurring in the $\La_i$'s. This implies that $\La_{i_0}=\La_{i_0,1}$, hence that $\La_1=\La_{1,i_0}$. We conclude by induction applied to \[\La_2\times \dots \times \La_{i_0-1}\times \La_{i_0+1} \times\dots \times \La_r.\]

\end{proof}

\end{document}